\documentclass[12pt, a4paper, oneside, titlepage]{report}
\usepackage[utf8]{inputenc}
\usepackage[english]{babel}
\usepackage{amsmath,amssymb,amsthm,geometry,microtype,graphicx,booktabs,caption,subcaption,textcomp,float}
\usepackage{hyperref}

\theoremstyle{definition}
\newtheorem{definition}{Definition}[chapter]
\newtheorem{lemma}[definition]{Lemma}
\newtheorem{proposition}[definition]{Proposition}
\newtheorem{theorem}[definition]{Theorem}
\newtheorem{corollary}[definition]{Corollary}
\newtheorem{example}[definition]{Example}
\newtheorem{conjecture}[definition]{Conjecture}
\newtheorem{game_rule}[definition]{Rule}

\theoremstyle{remark}
\newtheorem{remark}[definition]{Remark}
\newtheorem{historical_remark}[definition]{Historical Remark}

\title{Transfinite game values in infinite games}
\author{Davide Leonessi}
\date{Academic year 2020-21}

\begin{document}
\makeatletter
\begin{titlepage}

\vspace{8cm}

\begin{center}

~

~

{\huge  \@title }\\[25ex] 
\includegraphics[width=40mm]{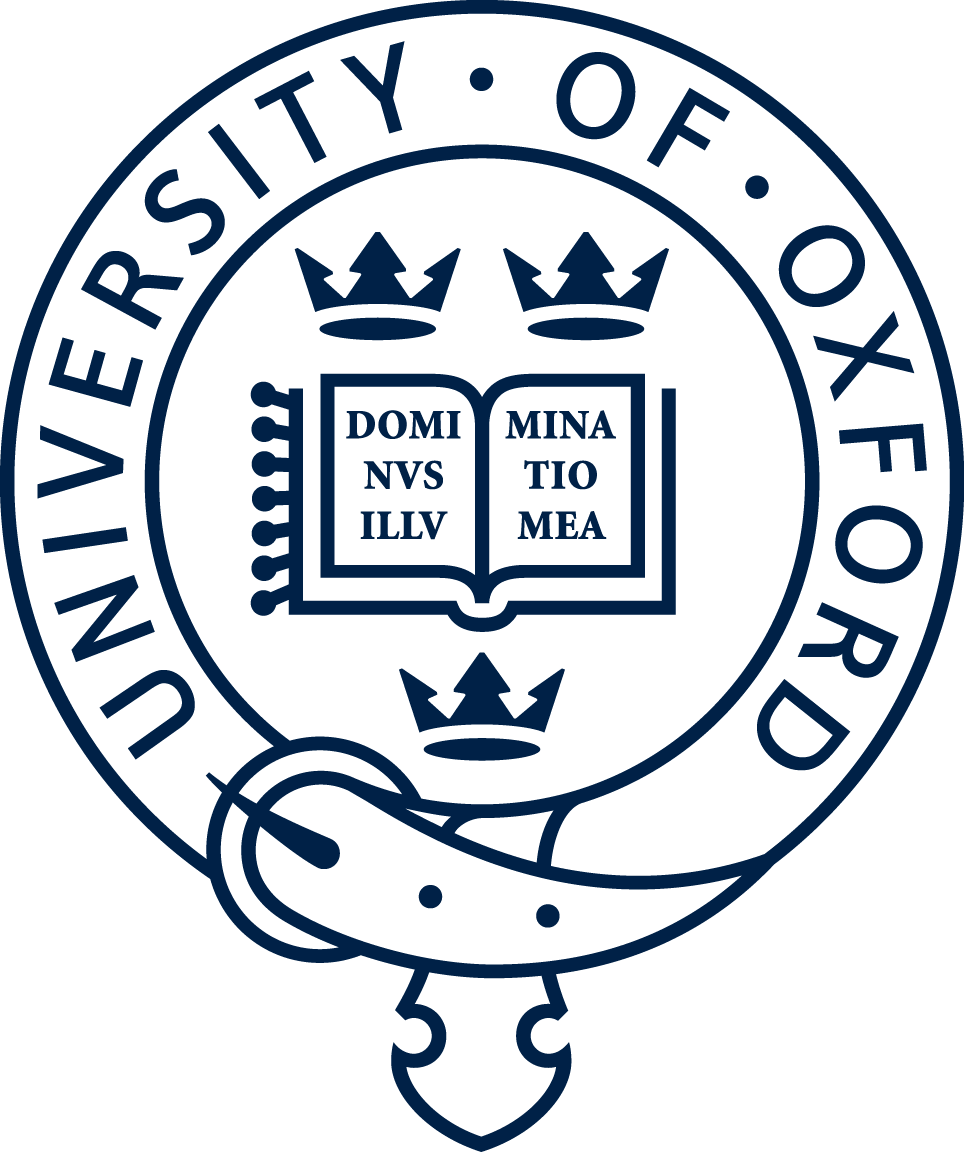}\\[18ex]
{\Large \@author}\\[4ex]
{\Large Supervised by Prof.~Joel David Hamkins}\\[12ex] 
{\Large MSc in Mathematics and Foundations of Computer Science}\\[4ex]
{\Large \@date}\\

\end{center}
\end{titlepage}
\makeatother

\newpage
\thispagestyle{empty}
\begin{abstract}
The object of this study are countably infinite games with perfect information that allow players to choose among arbitrarily many moves in a turn; in particular, we focus on the generalisations of the finite board games of Hex and Draughts.

\medskip

In Chapter \ref{ch_inf_games} we develop the theory of transfinite ordinal game values for open infinite games following \cite{hamkins14}, and we focus on the properties of the \emph{omega one}, that is the supremum of the possible game values, of classes of open games; we moreover design the class of climbing-through-$T$ games as a tool to study the omega one of given game classes.

\medskip

The original contributions of this research are presented in the following two chapters.

\medskip

In Chapter \ref{chapter_hex} we prove classical results about finite Hex and present Infinite Hex, a well-defined infinite generalisation of Hex.

We then introduce the class of \emph{stone-placing} games, which captures the key features of Infinite Hex and further generalises the class of positional games already studied in the literature within the finite setting of Combinatorial Game Theory.

The main result of this research is the characterization of open stone-placing games in terms of the property of \emph{essential locality}, which leads to the conclusion that the omega one of any class of open stone-placing games is at most $\omega$.
In particular, we obtain that the class of open games of Infinite Hex has the smallest infinite omega one, that is $\omega_1^{\rm Hex}=\omega$.

\medskip

In Chapter \ref{chapter_draughts} we show a dual result; we define the class of games of Infinite Draughts and explicitly construct open games of arbitrarily high game value with the tools of Chapter \ref{ch_inf_games}, concluding that the omega one of the class of open games of Infinite Draughts is as high as possible, that is $\omega_1^{\rm Draughts}=\omega_1$.
\end{abstract}

\newpage 
\vspace*{\stretch{1}}
\pdfbookmark{Dedication}{dedication}
\thispagestyle{empty}
\begin{center}
  \large \emph{To my Mother,\\
  who gave me the strength to embark on my journey.}
\end{center}
\vspace{\stretch{3}}

\newpage
\pagenumbering{roman}
\tableofcontents
\newpage
\pagenumbering{arabic}
\chapter{Infinite Games}\label{ch_inf_games}
The object of the present study are games played on infinite structures that generalise finite board games and that can thus allow players to choose among arbitrarily many moves in a turn.

In game-theoretic terms, we will consider two-player zero-sum infinite games with perfect information and no chance moves; we will define such games in extensive form, so to represent them as their induced game tree.\footnote{Infinite games in extensive form were first studied by \cite{gs53}.}

\medskip

We will focus on the class of games in which the order type of each play is at most $\omega$, so that plays consist of at most countably many turns, and, at each turn, players can choose among at most countably many moves. The game trees of such games are naturally seen as sub-trees of the following.

\begin{definition}\label{def_full_tree}
The \emph{full countable tree} $\omega^{<\omega}$ is a directed tree whose vertex set is the set of finite sequences of natural numbers, that is
\[\omega^{<\omega}=\bigcup_{n\in\omega} \omega^n= \bigcup_{n\in\omega} \{v:n \rightarrow\omega\},
\]
and whose directed edges order such finite sequences by inclusion.
\end{definition}

Note that we can label the edges of the full countable tree by ordinals smaller than $\omega$ so that each node $v$ is exactly the sequence of edge labels that compose the unique path that ends at $v$ and starts at the root node $\emptyset$, i.e.~the empty sequence.

Observe that the set of \emph{branches} of $\omega^{<\omega}$, which are the maximal paths in $\omega^{<\omega}$ starting at the root node $\emptyset$, is the collection of infinite sequences $\omega^\omega=\{s:\omega \rightarrow \omega\}$. 

\medskip

For the rest of this work, we will say that a triple $T=(V,E,v_0)$ is a \emph{tree} if $T$ is a directed tree with vertex set $V$, edge set $E$, and root node $v_0\in V$, which can be embedded into the full countable tree $\omega^{<\omega}$; that is, there is a map $\widetilde{\iota} :T \rightarrow \omega^{<\omega}$ induced by an injection $\iota :V \hookrightarrow \omega^{<\omega}$ which preserves the tree structure.\footnote{Note that $\omega^{<\omega}$ refers to both the vertex set of the full countable tree and the tree itself.}

If $(u,v)\in E$ is an edge of $T$, then we say that $v$ is a \emph{child} node, or \emph{immediate successor}, of $u$; conversely, we say that $u$ is the \emph{parent} node, or \emph{immediate predecessor}, of $v$; we write $u\rightarrow_T v$.

With this notation we can say that the injection $\iota :V \hookrightarrow \omega^{<\omega}$ induces a structure-preserving embedding in the sense that it preserves the root node and the edges, so that $\iota(v_0)= \emptyset$ and $\iota(u)\rightarrow_{\omega^{<\omega}} \iota(v)$ if and only if $u\rightarrow_T v$.

\bigskip

We can now define the games we consider in extensive form.

\begin{definition}\label{def_game}
A \emph{game} $\mathcal{G}$ is a collection of objects $((V,E,v_0),S,S_1,S_2)$, such that;
\renewcommand{\theenumi}{\roman{enumi}}
\begin{enumerate}
    \item $V$ is the set of \emph{positions}, and $v_0\in V$ is the \emph{initial position};
    
    \item\label{def_game_tree} $(V,E,v_0)$ is a tree, which we call the \emph{game tree};
    
    \item\label{turns} At each position, a player chooses a child node of that position in the game tree as the next position; the players alternate making moves and the turn is determined by the parity of the number of moves already made;
    
    
    \item\label{def_play} $S$ is the collection of branches of $(V,E,v_0)$, which we call the set of \emph{plays};
    
    \item\label{winning_set} $S_1,S_2\subset S$ are the \emph{winning conditions} of, respectively, the first and second player, such that $S_1\cap S_2= \emptyset$;
    a play $s\in S$ is a \emph{win} for the first or second player if, respectively, $s\in S_1$ or $s\in S_2$;
    moreover, a play $s\in S$ is a \emph{draw} if $s\not\in S_1\cup S_2$.
\end{enumerate}
\end{definition}

\begin{remark}\label{rmk_position}
As the game tree $T$ of condition \ref{def_game_tree}.~is a directed rooted tree, we are assuming that every non-initial position has a unique predecessor; since each branch of such $T$ has order type at most $\omega$, then each position defines a unique history of finitely many moves from the initial position $v_0$.
In other words, any partial play consists of only finitely many moves.

Hence, we can see that condition \ref{turns}.~has a well-defined meaning in the sense that the first player is the next to move at all positions whose unique path from the initial position comprises an odd number of vertices; this includes the initial position $v_0$. Conversely, the second player moves at all positions whose unique path from the initial position consists of an even number of positions.
\end{remark}

\begin{remark}
Note that in the majority of popular board games occurs some \emph{board position}, or \emph{board configuration}, which can be achieved by several distinct move sequences; we consider each move sequence as leading to distinct \emph{positions}.

Thus, we can see that positions carry more information than board positions; in addition to determining whose turn it is, a position uniquely specifies the history of moves from the initial position, and so all the board positions already achieved over the play.
\end{remark}

\begin{remark}
In the games that we consider in this work we will show that not both players can win, satisfying the disjointness of $S_1,S_2$ required in condition \ref{winning_set}.

More in general, plays that allow both players to win can be considered draws without loss of generality.
\end{remark}

\begin{remark}\label{rmk_countability}
As mentioned above, we are interested in games which feature positions with infinitely many successors, that is, with some positions $v\in V$ which have infinitely many children nodes.

It is important to note that the games defined here allow each position to have at most \emph{countably} many successors; this follows from our definition of trees as subtrees of the full countable tree.
\end{remark}

\medskip

Fixing a game $\mathcal{G}$ as in definition \ref{def_game}, we define the following common game-theoretic concepts; the definitions for the first player have a natural analogue for the second player.

\begin{definition}
Say that $V_1\subset V$ is the collection of positions at which the first player is the next to move.
A \emph{strategy} for the first player of $\mathcal{G}$ is a function $\sigma:V_1\rightarrow V$ such that $\sigma(v)$ is a child node of $v$ in the game tree of $\mathcal{G}$ for all $v\in V_1$.
\end{definition}

We say that a play $s\in S$ \emph{conforms to} a strategy $\sigma$ for the first player if, for any position $v$ in $s$ at which the first player is the next to move, i.e.~$v\in V_1$, we have that the successor of $v$ in $s$ is $\sigma(v)$;
in other words, the second player has a way to play the game that leads to the play $s$, when the first player follows $\sigma$.

\begin{definition}
A \emph{winning strategy} for the first player of $\mathcal{G}$ is a strategy $\sigma$ such that, for any play $s\in S$ conforming to the strategy $\sigma$ for the first player, we have that $s\in S_1$; i.e.~any play that conforms to such strategy is a win for the first player.
\end{definition}

Recall that the set of branches of the full countable tree is $\omega^\omega$; we can turn this set of sequences into a space by imposing on it the product topology induced by the discrete topology on $\omega$.

A base for such topology is given by the collection of position neighbourhoods $\mathcal{U}(v)$ for $v\in \omega^{<\omega}$, where $\mathcal{U}(v)=\{s\in\omega^\omega:v\subset s\}$, i.e.~$\mathcal{U}(v)$ is the collection of all branches which have the path from $v_0$ to $v$ as an initial segment.\footnote{Clearly, $\mathcal{U}(\emptyset)=\omega^\omega$.}

\medskip

We can similarly impose a topology on the set of branches $S$ of any tree $T=(V,E,v_0)$.
Let $\{\mathcal{U}(v)\}_{v\in V}$, the collection of neighbourhoods of finite paths in $T$, be the base for the topology on $S$.

Observe that this topology is equivalent to the subspace topology induced on $S$ from $\omega^\omega$ via any structure-preserving embedding $\widetilde{\iota}: T\rightarrow \omega^{<\omega}$.

In particular, if $T$ is the game tree of some game $\mathcal{G}$, then we can impose this topology on the set of plays $S$; in this case we can interpret the neighbourhood $\mathcal{U}(v)$ of a position $v$ to be the collection of plays that complete the partial play up to $v$.

Now, if the winning condition $S_1\subset S$ is open, then $S_1$ is equal to a union of position neighbourhoods $\bigcup_{i\in\mathcal{I}} \mathcal{U}(v_i)$ for some nodes $v_i\in V$; hence, for a given play $s\in S_1$, there is some node $v_j$ such that $s\in \mathcal{U}(v_j) \subset S_1$.

Moreover, any play $s'$ that contains $v_j$ is necessarily a win for the first player, as $s'\in \mathcal{U}(v_j) \subset S_1$; that is, the outcome of any play that at some point reaches $v_j$ is determined at that position and cannot be changed by any move made after $v_j$ by either player.
We can then say that any such play $s'$, including $s$, is ended in $v_j$, and is thus finite or \emph{essentially finite}.\label{essentially_fin}

\medskip

We can formally define an essentially finite play as a an infinite play whose outcome is the same of all the plays in the neighbourhood of one of its positions; that is, a play $s\in S$ is essentially finite if $s\in S_1$ (or, respectively, $S_2$, $(S_1\cup S_2)^c$) and there is some position $v$ in $s$ such that $\mathcal{U}(v)\subset S_1$ (or, $S_2$, $(S_1\cup S_2)^c$).
We can observe that there is no need to distinguish between such plays and finite plays, so that we will not.

In fact, consider a finite play $s$ of, say, usual Chess; the leaf node in $s$ is the ending position, which occurs after finitely many turns, according to the standard rules.
Then, we can extend $s$ to an infinite sequence of moves by allowing the players to make trivial moves in which they acknowledge the end of the game by alternately saying ``OK" to each other ad infinitum.

However, we can naturally define the ``actual" length of a (essentially) finite play $s$ as the minimal\footnote{It is immediate to see that if $v\in s$ and $s\subset \mathcal{U}(v)$, then also $s\subset \mathcal{U}(w)$ for all the descendants $w$ of $v$ in $s$.}
 $n$ for which there is a sequence of $n$ moves from the initial position to a node $v\in s$ and $s\subset \mathcal{U}(v)$.

\bigskip

We will focus on games which, even though infinite, enjoy the following finiteness property.

\begin{definition}\label{def_open}
A game is \emph{open} for one player if any win by that player takes place in finitely many moves; that is,\footnote{As remarked by \cite[\textsection 1]{hamkins17}.} if the winning condition for that player is open in the space of plays.
\end{definition}

We similarly define the following.

\begin{definition}
A game is \emph{closed} for some player if his winning condition is closed in the space of plays; that is, all the plays which are not wins for that player are determined to be so in finitely many moves.\footnote{However, it may not be determined in only finitely many moves whether such plays are draws or wins for the player's opponent.}

A game is \emph{clopen} for some player if it is open and closed for that player.
\end{definition}

\begin{remark}
If a game has no draws, then it is clopen for some player if and only if it is clopen for that player's opponent.

More in general, this implication holds when the set of draws is clopen in the space of plays.
\end{remark}


\medskip

In the majority of board games, players have essentially analogous goals, so that the direct infinite generalisations of such games are often open for both players,\footnote{A game open for both players is not necessarily clopen, as it may allow draws to occur with infinite plays, so that the game may not be closed for either player.} as we can see for instance in Chess, which was studied in detail by \cite{hamkins14}.

\begin{example}Infinite Chess.\label{infinite_chess}

This game is played on the infinite chessboard, which is a square tiling of the plane in bijection with $\mathbb{Z}\times\mathbb{Z}$, and features the usual Chess pieces, generalising the long-distance ones, i.e.~rook, bishop, and queen, to move for arbitrarily large, but finite, distances on the board.

The winner of a game of Infinite Chess is naturally defined as the first player that places the opponent's king in a checkmate position, which can only happen after finitely many moves.
It follows that all the plays that involve infinitely many turns are draws.

As a consequence of the definition of winning condition, Infinite Chess is an open game for both players, White and Black, given any starting position.

We can modify the winning conditions to make Infinite Chess not open for one player. Suppose we have an initial position in which Black has one king, while White has none; allow White to win if he satisfies the previous checkmate condition, and let Black win if the game continues for infinitely many turns,\footnote{Such an infinite play would be a tie according to the previous winning conditions.} so that his king is never in checkmate.
This game is clearly open for White, but not for Black.
\end{example}


This example highlights the importance of both the winning conditions and the initial position in the definition of a game.

We will see later another game open for only one player; the Angel and Devil game in Example \ref{angel_devil}.

\section{Game values}\label{sec_game_values}
The class of open games allows for the definition of ordinal game values within the framework of infinitary game theory.

We present the theory of transfinite game values in infinite games as developed by \cite{hamkins14}.

\medskip

We approach the concept of game value by borrowing another idea from Chess.

\begin{example}\label{mate_in_n}
The mate-in-$n$ problem.

Often presented as a puzzle for small values of $n$, a Chess position is defined to be \emph{mate-in-$n$} for White if there is a winning strategy for White such that all plays conforming with it are wins for White after at most $n$ moves\footnote{Clearly White could play sub-optimally, without winning in $n$ moves; this does not change the fact that the position itself is mate-in-$n$.} and, moreover, there is a strategy for Black such that all plays conforming with it comprise at least $n$ moves;\footnote{In both cases, only the moves by White are counted.}
we call this strategy for White a \emph{win-in-$n$} strategy.

In such a situation we can think of White as being at most $n$ moves away from winning; if $n$ is minimal, then the position is said to have game value $n$ for White.

For instance, observe that White has already won in a mate-in-0 position, and White can force a checkmate from a mate-in-1 position on his next move; in the latter case, White's move may depend on Black's move, if White plays second.
\end{example}

We can generalise this concept to the infinite case for the smallest infinite ordinal, $\omega$.

An Infinite Chess position is defined to have game value $\omega$ for White if White has a strategy to win in finitely many moves, but Black is next to move and his initial choice determines a finite lower bound for the number of turns that need to be played before Black's inevitable defeat.

In other words, Black makes a move from a position with game value $\omega$ to a mate-in-$n$ position, where he can choose $n\in\mathbb{N}$ to be arbitrarily large.

\medskip

Observe the example\footnote{\cite[Figure 5]{hamkins14}} of a position with game value $\omega$ for White in figure \ref{fig_chess_omega} which features infinitely many pawns, as indicated by the pattern.

In that position, Black is the next to move and the central Black rook is attacked by a pawn; the initial move by Black is moving the rook up the empty corridor for an arbitrarily long, but finite, distance.

White can then capture the Black rook with a pawn and systematically advance the pawns below the resulting empty square, until a White rook is let free to mate the Black king.
The Black pawn on the side prevents a stalemate.

\begin{figure}[h]
\centering
\includegraphics[width=4cm]{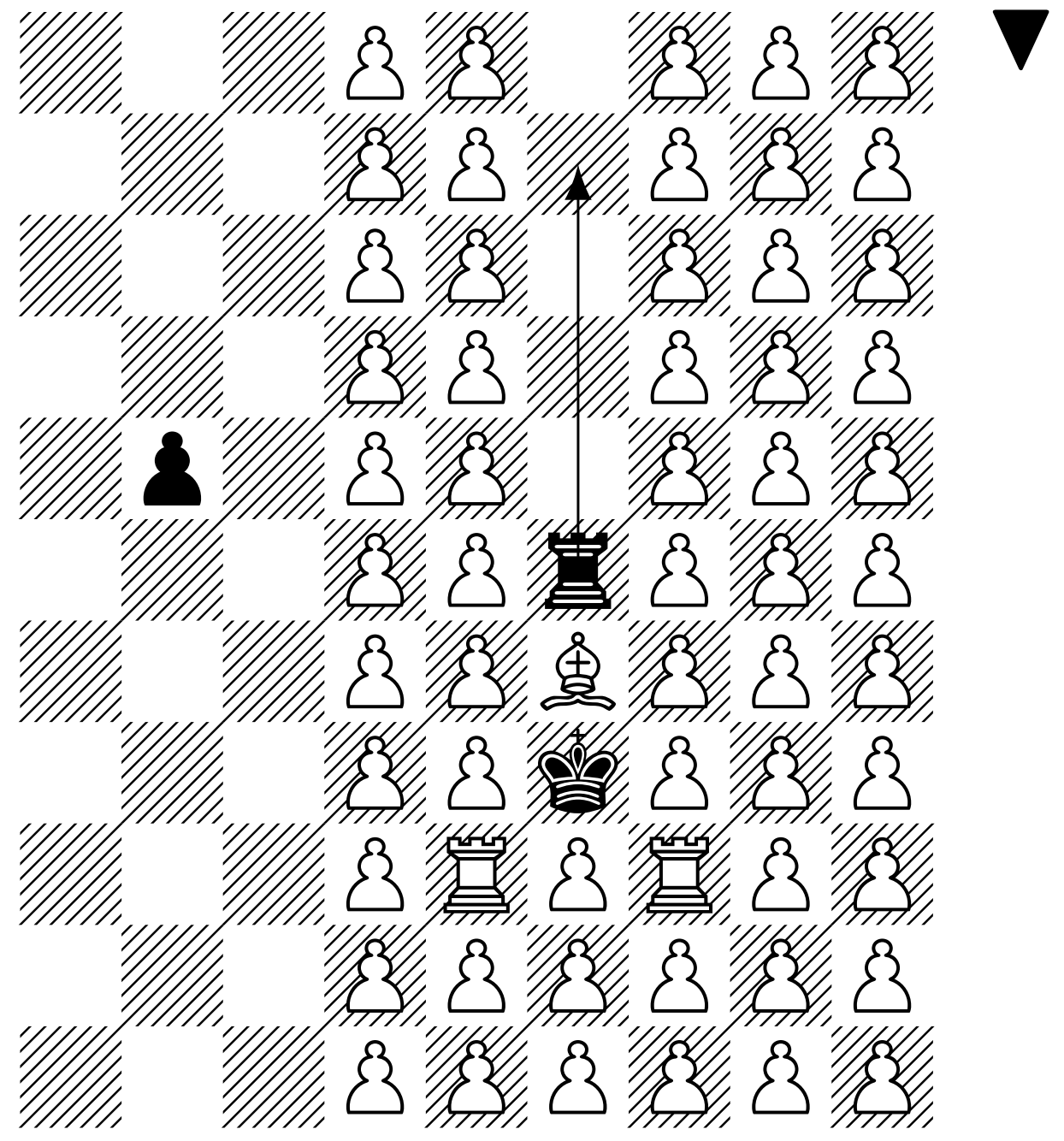}
\caption{Infinite Chess position with game value $\omega$ for White, Black to move.}
\label{fig_chess_omega}
\end{figure}

We have now given enough motivation to define the following.

\begin{definition}\label{def_game_value}
The \emph{game value} of a position $p$ in a game open for a player $O$, playing against $C$, is defined by recursion as follows;

if $O$ has already won in $p$, then $p$ has game value 0;

if $O$ is the next to move in $p$, and $\alpha$ is the smallest ordinal among the game values of the positions\footnote{Only the positions for which a game value can be defined.} that can be reached from $p$ via a legal move by $O$, then $p$ has game value $\alpha+1$;

if instead $C$ is the next to move in $p$, and all the positions that can be reached from $p$ via a legal move by $C$ have a defined game value, then the value of $p$ is the supremum of all these values.
\end{definition}

Note that this definition may not apply to all positions of a game; the positions with defined game value for some player are exactly the positions from which that player can force a win, which is reached in finitely many moves by openness of the game.

\medskip

Observe that Definition \ref{def_game_value} reflects the rational behaviour of the players in a game, from the perspective of player $O$.
The \emph{open player} $O$ aims to reduce the game value as quickly as possible, when it is defined, and to reach a position with defined value otherwise;
on the other hand, the \emph{closed player} $C$ has the opposite aim and tries to keep the game value undefined, if possible, or as high as possible, when the value is defined.

\medskip

Suppose that a position $p$ has ordinal value $\alpha>0$ for $O$, so that, by definition of game value, if $O$ is the next to move from $p$, then $O$ can reach a position with game value strictly smaller than $\alpha$;
if instead $C$ is the next to move from $p$, then $C$ can only reach positions with game value $\beta\le\alpha$, in particular $\beta<\alpha$ if $\alpha$ is a limit ordinal.

Thus, $O$ can play from $p$ so to strictly reduce the game value at each turn, while $C$ cannot reach a position with higher, or undefined, game value; we call this the \emph{value-reducing} strategy for the winning player.
By well-foundedness of the ordinals, it is clear that a winning player following such a value-reducing strategy does eventually reach a winning position of value 0 in finitely many turns, as there are no infinite descending sequences of ordinals.

Conversely, if $p$ has no defined game value for $O$, then $O$ cannot make a move to reach another position with defined game value, as otherwise $p$ would have had a game value in the first place.
Moreover, starting from such $p$, there is at least one position with undefined game value that $C$ can reach with one move, again by Definition \ref{def_game_value}; this gives rise to the \emph{value-maintaining} strategy, with which $C$ can keep the game value undefined at all positions reached over the play.

We conclude that the positions without a game value for $O$ are exactly the positions from which $C$ can force either a draw or a win, meaning that $C$ has a drawing or winning strategy for the game starting at $p$.

\medskip

Furthermore, if a game has no draws and is open for both players,\footnote{Such a game is also clopen.} then the positions with defined game value for one player have undefined value for the other, and vice versa.\footnote{This is true since we only consider games in which at most one player wins, as per the definition of winning condition in \ref{def_game}.\ref{winning_set}.}

\medskip

We will refer to the game value of a game to mean the game value of its initial position.

\bigskip

We can provide a visual intuition of what it means for a position $p$ to have game value for $O$ by fixing a winning value-reducing strategy for the open player $O$ and thus observing that the way the play unfolds depends entirely on the losing player $C$; by considering the ``partial game tree" that represents only the choices available to $C$, we obtain the structure of the following game.

\begin{example}\label{climb_tree}
The climbing-through-$T$ game.

Let $T$ be a tree; recall that $T$ is a rooted directed tree whose branching nodes have at most countable degree, and whose branches have at most order type $\omega$.

Let two players, Climber and Observer, play the following game.

Climber aims to climb up the tree, starting from the root; at each turn, if Climber is on a branching node, then he can move to one of its successor nodes, if instead Climber is standing on a leaf node, then he loses.

On the other hand, Observer does not directly influence the game; Observer simply acknowledges Climber's moves by saying ``OK", until Climber reaches a leaf node, making him win.

\medskip

This is clearly a game open only for Observer, as Climber wins exactly when he manages to climb up an infinite branch of $T$, thus playing for infinitely many turns.

In particular, Climber can win if and only if $T$ has an infinite branch, and Observer has a (trivial) winning strategy if and only if $T$ is well-founded, that is, if all the branches of $T$ have finite length.

\medskip

The simplicity of climbing-through-$T$ games enables us to assign values easily to such games; these will be a key tool in the construction of Infinite Draughts positions in Chapter \ref{chapter_draughts}.

\medskip

Observe that $\rightarrow_T$ is a strict partial order\footnote{A \emph{strict order} is an irreflexive and transitive relation, which is thus antisymmetric.} on the vertex set of $T$; we can thus apply the theory of strict partial orders on well-founded trees, as described in \cite[\textsection III.5]{kunen13}, to establish a link between tree rank and game value.

Define the \emph{rank}, or \emph{height}, of a node of a well-founded tree as ${\rm rank}(u)={\rm sup}\{{\rm rank}(v)+1: u\rightarrow_T v\}$; the rank of a well-founded tree is defined to be the rank of its root.

We can see that the leaves of the tree have rank 0, as they have no children. Nodes whose children are all leaves have rank 1; in general, the rank of a branching node is the supremum of the ordinal successors of the ranks of the nodes it is the parent of.

\medskip

We can now easily prove by transfinite induction the following.

\begin{proposition}\label{tree_value}
The rank of a node $v$ of a tree $T$ is exactly equal to the game value for Observer of the position of the climbing-through-$T$ game in which Climber is standing on $v$.
\end{proposition}

\begin{proof}
The nodes of $T$ with rank 0 are leaves, which are precisely the nodes on which Climber stands when Observer has won.

\medskip

Say $v$ is a node of $T$ with rank $\beta+1$; then, the children of $v$ have rank bounded by $\beta$, and one of them, say $w$, has exactly rank $\beta$.
When standing on $v$, Climber can only move to positions with defined game value bounded by $\beta$, by inductive hypothesis; by climbing one step up to $w$, Climber reaches a position of game value exactly $\beta$, again by hypothesis.

Hence, Climber was at a position with value $\beta+1$, when standing on $v$.

\medskip

Say $v$ is a node with a limit ordinal $\gamma$ as rank.
From $v$, Climber can only reach nodes of arbitrarily high rank $\alpha<\gamma$, which correspond to positions of value $\alpha<\gamma$ for Observer, by inductive hypothesis.

This completes the inductive proof by definition of game value.

\medskip

Moreover, note that $v$ has undefined rank exactly when it is contained in an infinite branch of $T$; in that case, the corresponding position of climbing-through-$T$ has undefined game value for Observer since Climber has a winning strategy, namely climbing from $v$ up such infinite branch of $T$.
\end{proof}

\medskip

We present the concrete example of a tree of rank $\omega+3$ in figure \ref{fig_tree},\footnote{Even though the height of nodes in a tree increases towards the root, we represent trees with the root down and the leaves up, because that is what ``they do in nature", as remarked by \cite[p.~202]{kunen13}.} where the nodes are labelled by their rank.

\begin{figure}[h]
\centering
\includegraphics[width=5cm]{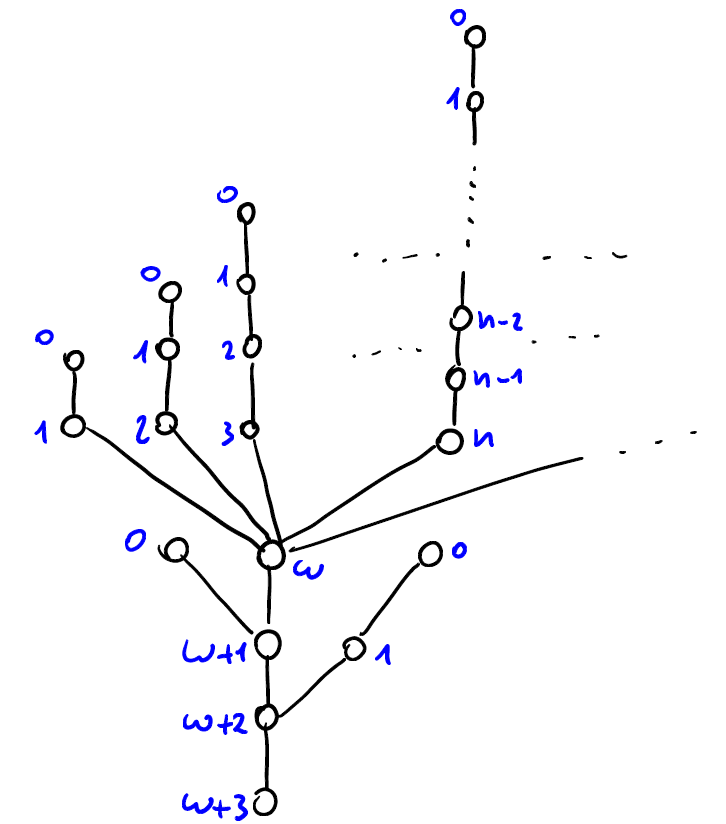}
\caption{A well-founded tree with nodes labelled by their ranks.}
\label{fig_tree}
\end{figure}

Note in particular the node with rank $\omega$; it is a branching node of countable degree whose children have arbitrarily high, but finite, rank $n\in\mathbb{N}$.

\medskip

We conclude this example justifying the presentation of $T$ as a ``partial game tree"; $T$ really is the tree obtained by the game tree of the climbing-through-$T$ game by joining each decision node of Observer to its preceding node, and deleting the edge between them.
\end{example}

\medskip

Setting the scene for the next section, we now show that the ordinals that can be realised as game values of positions in a game define the initial segment of some ordinal.

\begin{proposition}\label{init_segment_ord}
Let $p$ be a position in a game with game value $\alpha$ for some player; for any ordinal $\beta$ such that $\beta\le\alpha$, there is a finite sequence of legal moves that from $p$ reaches some position $q$ with game value $\beta$ for the same player.
\end{proposition}
\begin{proof}
Let $\mathcal{G}$ be a game played by $O$ and $C$, and let it be open for $O$. The position $p$ has value $\alpha$ for the open player $O$.

\medskip

We prove the result by transfinite induction on $\alpha$; the base case $\alpha=0$ is trivial.

\medskip

Let $\alpha=\delta+1$ be a successor ordinal and take $\beta\le\delta<\alpha$.

If $O$ is the next to move in $p$ then, by definition of game value, he can make a move to some position with game value $\delta$; then there is a sequence of moves that leads to a position of value $\beta\le\delta$ by inductive hypothesis.

If $C$ is the next to move, then $\delta+1$ is the supremum of the game values of all positions reachable by $C$; in particular, she can reach some position also of value $\delta+1$ in which $O$ is next to move. The result follows from the previous case.

\medskip

Let $\alpha=\gamma$ be a limit ordinal and take $\beta<\gamma$.

Recall that a limit ordinal can be achieved as game value of some position only if the losing player is the next to move.
Hence, $C$ is the next to move in $p$.

Again by definition, $\gamma$ is the supremum of the game values of all positions in the reach of $C$; in particular, she can reach positions of value arbitrarily high, but strictly bounded by $\gamma$.
Hence, $C$ can make a move to reach a position $q$ with game value $\delta<\gamma$ such that $\beta\le\delta$.
We can now apply the inductive hypothesis to $q$, which concludes the proof.
\end{proof}

\section{The omega one}
We can informally think of high-value positions as evidence of the ``strategic complexity" of a specific open game.
We formalise this idea in the following definition that we make for a \emph{class of games}, that is a collection of games constructed according to the same rules, but with different initial positions.

\begin{definition}
The \emph{omega one} of a class of games $\mathbb{G}$, all open for some player, is the supremum of the game values of all the games in $\mathbb{G}$, for that player. We denote it by $\omega_1^{\mathbb{G}}$.
\end{definition}

Recall that we only consider games in which plays are sequences of at most countably many moves and, at each turn, players can choose one of up to countably many moves.
In fact, we will only consider games played by moving or placing pieces on boards of countably infinite size.

Observe that, in such a game $\mathcal{G}$, the game value of all positions will be the supremum of at most countably many ordinals, and if all such ordinals are countable, then $\omega_1$ cannot arise as game value of any position.

The result below follows by transfinite induction.

\begin{proposition}\label{bound_omega_one_game}
For any open game $\mathcal{G}$, we have that the game value of $\mathcal{G}$ is strictly smaller than $\omega_1$.
\end{proposition}

\begin{corollary}\label{bound_omega_one_class}
For a class of games $\mathbb{G}$, all open for some designated player, we have that $\omega_1^{\mathbb{G}} \leq \omega_1$.
\end{corollary}

\begin{remark}
Proposition \ref{bound_omega_one_game} and Corollary \ref{bound_omega_one_class} crucially rely on game positions having at most countably many successors, which is an assumption that we make throughout this work, as highlighted in Remark \ref{rmk_countability}.

These results do not hold for classes of more general games, such as the one mentioned later in footnote \ref{hypergame}.
\end{remark}

We will see now that this bound is sharp for the class of climbing-through-$T$ games, which are open for Observer.

\begin{lemma}\label{lemma_climb_tree}
For any ordinal $\alpha<\omega_1$, there is a well-founded tree $T$ of rank $\alpha$.
\end{lemma}
\begin{proof}
We prove this by transfinite induction.

If $\alpha=0$, then $T$ is the trivial tree with only one node.

\medskip

If $\alpha=\beta+1$, then take a tree $T'$ of rank $\beta$ by inductive hypothesis and simply let $T$ be a tree obtained by adding an extra node $v$ to $T'$ and let $v$ have the root of $T'$ as only child; $v$ is the root of $T$, which thus has rank $\beta+1=\alpha$.

\medskip

If $\alpha$ is a limit ordinal, then let $T_\beta$ be a tree of rank $\beta$ for all $\beta<\alpha$.
Define $T=\bigoplus_{\beta<\alpha} T_\beta$ similarly to the previous construction; $T$ is obtained by letting an extra node $v$ have as children the roots of each $T_\beta$, for $\beta<\alpha$.
If $T$ has $v$ as root, then it is clearly as needed.

Note that this construction of $T$ satisfies our earlier definition of tree because $\alpha$ is a countable ordinal, and so the root node $v$ of $T$ has countable degree.
\end{proof}

\begin{remark}\label{rmk_value_tree}
A climbing-through-$T$ game has a defined game value for Observer if and only if $T$ is well-founded, otherwise the game would be a win for Climber, as we saw in the proof of Proposition \ref{tree_value}.
\end{remark}

\begin{proposition}\label{omega_one_climb}
The omega one of the class of open climbing-through-$T$ games is $\omega_1$.
\end{proposition}
\begin{proof}
By Lemma \ref{lemma_climb_tree} and Proposition \ref{tree_value}, there are climbing-through-$T$ games of arbitrarily high value $\alpha<\omega_1$; therefore, the omega one of the class of climbing-through-$T$ games is at least $\omega_1$.

If there were some climbing-through-$T'$ game of value $\omega_1$, then there would be some well-founded tree $T'$ of rank $\omega_1$ by Proposition \ref{tree_value}, so that the root of $T'$ would have uncountable degree; this is a contradiction because we assume trees to be sub-trees of the full countable tree.

Invoking the contrapositive of Proposition \ref{init_segment_ord}, we confirm the bound given by Corollary \ref{bound_omega_one_class} and complete the proof.
\end{proof}

\begin{remark}
Observe that no climbing-through-$T$ game can have value $\omega_1$ since the ranks of well-founded trees, which correspond to the possible initial positions, are countable.

Therefore, a necessary condition to have uncountable omega one for the class of climbing-through-$T$ games, as for any other game class,\footnote{\label{hypergame}We could consider a climbing-through-$T$ \emph{hypergame} in which Climber makes the first move and chooses a well-founded tree $T$ on which to play; it follows from our discussion that such a hypergame would have uncountable game value.} is to have uncountably many possible initial positions.
\end{remark}

\medskip

In order to stress the importance of infinite branching nodes in achieving infinite game values, we propose the following game defined by \cite{conway96}, which is an interesting example of game open for only one player.

\begin{example}\label{angel_devil}
Angel and Devil game.

Let two players, Angel and Devil, play on an infinite chessboard, which is in bijection with $\mathbb{Z}\times\mathbb{Z}$, as usual.

On his turn, the Devil can mark, i.e.~\emph{burn}, any square of the chessboard; such square becomes unavailable to the Angel.

At the initial position, the Angel stands on the origin (0,0) and, at each turn, moves by at most $p$ chess king's moves away from the previous square; that is, the Angel can move from $(x,y)$ to $(x',y')$ if $\lvert x - x'\rvert, \lvert y - y'\rvert\le p$ and $(x',y')$ is not burnt by the Devil; we say that the Angel has \emph{power} $p$.

The Devil wins if the Angel remains trapped, that is, at some turn the Angel has no legal moves;
conversely the Angel wins by remaining free and by having a legal move to make at all stages of the play.

\medskip

Observe how Conway's game, which is played on an infinite board, does not allow for infinite game values for the Devil because of the finite choice of moves of the Angel.

The game is clearly open only for the Devil, while the Angel wins exactly when both players make infinitely many moves.
Thus, assume that some position has infinite game value, where the Devil is winning; by Proposition \ref{init_segment_ord}, it is enough to take it to be $\omega$.

By definition of game value, the Angel, who is the losing player, has to make the first move. But the Angel can only choose one of finitely many moves; namely, the Angel can only choose among the at most $(2p+1)^2$ cells within the square allowed by its power $p$ which have not yet been burnt by the Devil.

But then, $\omega$ would be the supremum of at most $(2p+1)^2$ finite values, which is a contradiction.

The only way to obviate to this issue would be giving `infinite power' to the Angel, allowing it to move to an arbitrary cell of the infinite board, but that would lead the Angel to trivially win the game against a Devil that only burns finitely many cells at each turn;
such an Angel of infinite power would win even against a Devil who burns coinfinitely many cells at each turn, that is, if the Devil always leaves infinitely many cells available for the Angel to move on.

\medskip

We conclude that, generalising the choice of initial position, the class of games induced by the Angel and Devil game has omega one at most $\omega$.
\end{example}

\begin{remark}
Conway originally asked for a proof of the intuitive fact the Angel and Devil game starting from an empty board is a win for the Angel.

It was proved by Berlekamp\footnote{According to \cite[\textsection 1]{conway96}.} that an Angel of power 1, that is a chess king, can be trapped on any sufficiently large finite board; the question remained open for a decade, until \cite{andras07} proved that the Angel of power 2 wins against the Devil.
\end{remark}

\begin{proposition}
Let $\mathbb{G}$ be a class of games open for player $O$, and say that his opponent $C$ has only finitely many moves to choose from at each turn.\footnote{Observe that a finite bound for the number of possible moves for $C$ across all turns is not required.} Then, $\omega_1^\mathbb{G}\le \omega$.

Moreover, if $O$ has a winning strategy for some game $\mathcal{G}$ in $\mathbb{G}$, then $O$ has a win-in-$n_\mathcal{G}$ strategy\footnote{See the mate-in-$n$ problem in Example \ref{mate_in_n}.} for some $n_\mathcal{G} \in\mathbb{N}$.
\end{proposition}

\begin{proof}
The discussion in Example \ref{angel_devil} shows that no infinite ordinal can be realised as the value of a position of the Angel and Devil game, and so the omega one of its induced class of games is at most $\omega$.

That argument relies only on Proposition \ref{init_segment_ord} and on the fact that the Angel has finitely many moves to choose from when at a position of value $\omega$, so we can directly generalise it and obtain that $\omega_1^\mathbb{G}\le \omega$.

\medskip

Suppose now that $O$ plays according to some winning strategy $\sigma$ in a game $\mathcal{G}$ of $\mathbb{G}$. Let $p$ be the first position in which $C$ is the next to move.\footnote{$p$ is either the initial position, or the one reached by the first move of $O$.}

Note that $C$ can reach only finitely many positions from $p$, and all such positions have finite game value for $O$; let $m$ be the supremum of all such values.

It is clear that $m$ is finite and that $\sigma$ is a win-in-$n$ strategy for $n=m$ or $(m+1)$, depending on who moved first, so that the longest play of $\mathcal{G}$ ends after $n$ moves by $O$.
\end{proof}

We conclude by focussing our attention more narrowly than as stated at the beginning of this chapter; the games which are interesting for our analysis are open for some player and allow that player's opponent to have a strategy ensuring that she can choose among infinitely many possible moves at least once in any play.

\bigskip

We will prove in Chapter \ref{chapter_hex} that in Infinite Hex the freedom of infinite choice granted to the losing player is not enough to construct a proper obstruction to the winning player's main line of play.
The class of open games of Infinite Hex has the smallest possible omega one, that is $\omega_1^{\rm Hex}=\omega$; we can interpret this as saying that the complexity of the infinite generalisation of Hex is lost when we restrict our analysis to games open for some player.

In Chapter \ref{chapter_draughts} we will show a dual result; the class of open games of Infinite Draughts features positions of arbitrarily high game value, so that the losing player's obstructing power is put to good use, and hence $\omega_1^{\rm Draughts}=\omega_1$.
This means that we can directly appreciate the strategic complexity of the infinite version of Draughts in games open for some player.

\newpage
\chapter{Hex}\label{chapter_hex}
\section{Prelude: finite Hex}

The game of Hex was first designed by the inventor and polymath Piet Hein, who published it as ``Polygon" and made it become a popular board game in Denmark, starting in 1942.\footnote{\cite{hein42}}

Hex was then rediscovered by Nobel laureate John Nash when still a student at Princeton in 1948; he described it as a ``matter of connecting topology and game theory".\footnote{According to a private communication by \cite[\textsection 3]{hayward06}.}

From the common room of the Department of Mathematics at Princeton, the then-called game of ``Nash" became popular in the mathematical community, and eventually published as the mainstream board game ``Hex" in 1952.

\bigskip

Hex is played on a rhomboidal board of $n\times n$ hexagonal tiles by two players, Red and Blue, who alternate placing stones on the empty tiles of the board;
to each player is assigned a pair of opposite board sides and the player who joins his assigned sides with a connected chain of stones of his respective colour is the winner.

See in figure \ref{hex_empty} the initial position and in \ref{hex_full} the ending position of a play won by Red.

\begin{figure}[h]
\centering
\begin{subfigure}{4cm}
  \centering
  \includegraphics[width=.8\linewidth]{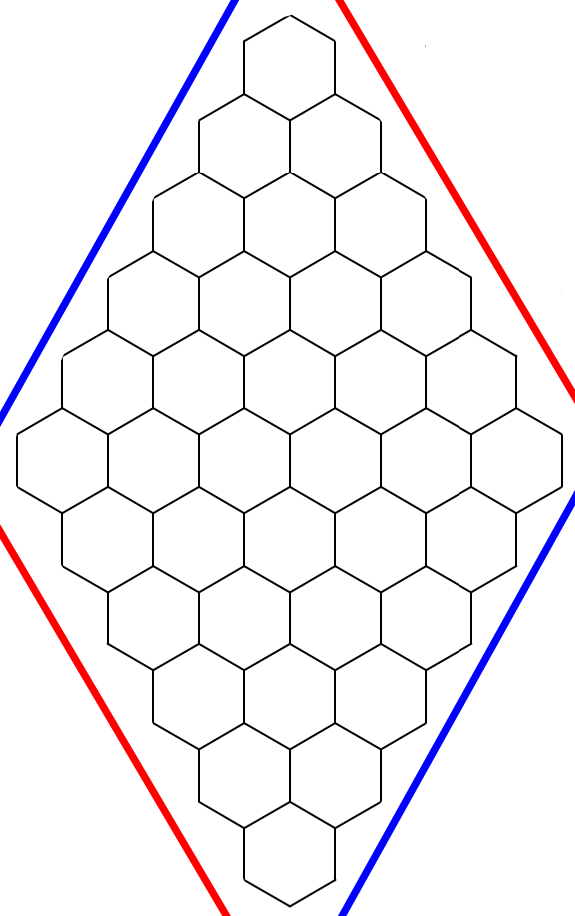}
  \caption{Initial position~~~~~~}
  \label{hex_empty}
\end{subfigure}
\begin{subfigure}{4cm}
  \centering
  \includegraphics[width=.8\linewidth]{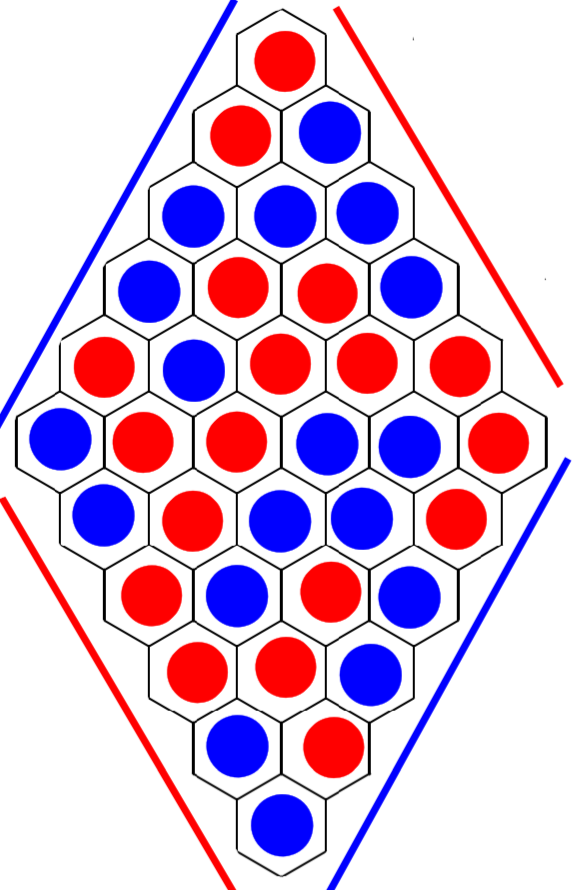}
  \caption{Ending position}
  \label{hex_full}
\end{subfigure}
\caption{The $6\times 6$ Hex board.}
\label{hex_board}
\end{figure}

\subsection{Fundamental results}\label{results_finite_hex}

\begin{remark}\label{rmk_boundary}
Observe that the following Theorems \ref{hex_theorem} and \ref{jordan_hex} rely on the topological properties of the boundary of the finite Hex board; namely, that the boundary is made of 4 contiguous segments alternatively assigned to the two players.
\end{remark}

\begin{theorem}[Hex Theorem]\label{hex_theorem}
At a finite Hex position in which all the tiles are marked by either Red or Blue, at least one of the players has won.
\end{theorem}

The Hex Theorem is the result that carries the highest mathematical significance of finite Hex; \cite{Gale79} showed it to be equivalent to the Brouwer Fixed Point Theorem, which in turn was shown to be equivalent to the Jordan Curve Theorem by \cite{maehara84} and \cite{adler16}.

In order to present the argument by \cite{Gale79}, we need to mention the following simple graph-theoretic result.

\begin{lemma}\label{lemma_graph}
A finite graph whose vertices have degree at most 2 is the union of disjoint subgraphs, each of which is either (i) an isolated vertex, (ii) a simple cycle, or (iii) a simple path.
\end{lemma}

\begin{proof}[Proof of Theorem \ref{hex_theorem}]
Let $\Gamma$ be the graph induced by the edges on the Hex board with the addition of 4 edges ending in the vertices $n, e, s, w$, as in figure \ref{gale_empty}.

Now, given a completely marked board, we present an algorithm to find a winning chain of stones.

Starting from $s$, we construct a tour along $\Gamma$ following the rule of always proceeding along an edge which is the common boundary of a cell with a Red stone and a cell with a Blue stone; observe that the edge starting from $s$ does have this property, by considering the quadrants outside of the board as assigned to the players, see figure \ref{gale_full} for an application of this touring rule starting from all the vertices of $\Gamma$.

\begin{figure}[h]
\centering
\begin{subfigure}{5cm}
  \centering
  \includegraphics[width=.8\linewidth]{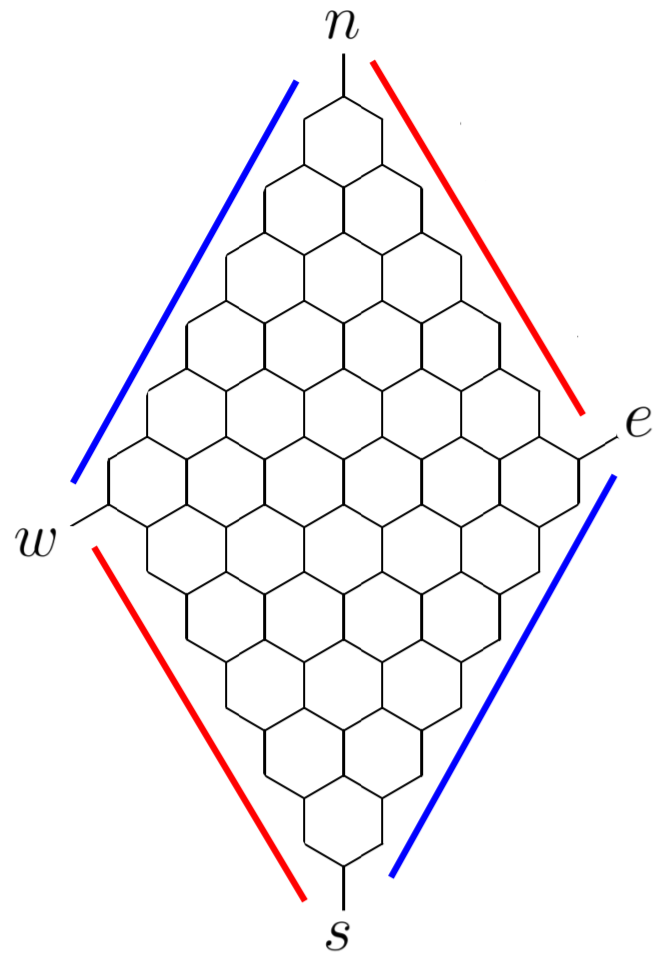}
  \caption{The edge graph $\Gamma$}
  \label{gale_empty}
\end{subfigure}
\begin{subfigure}{5cm}
  \centering
  \includegraphics[width=.8\linewidth]{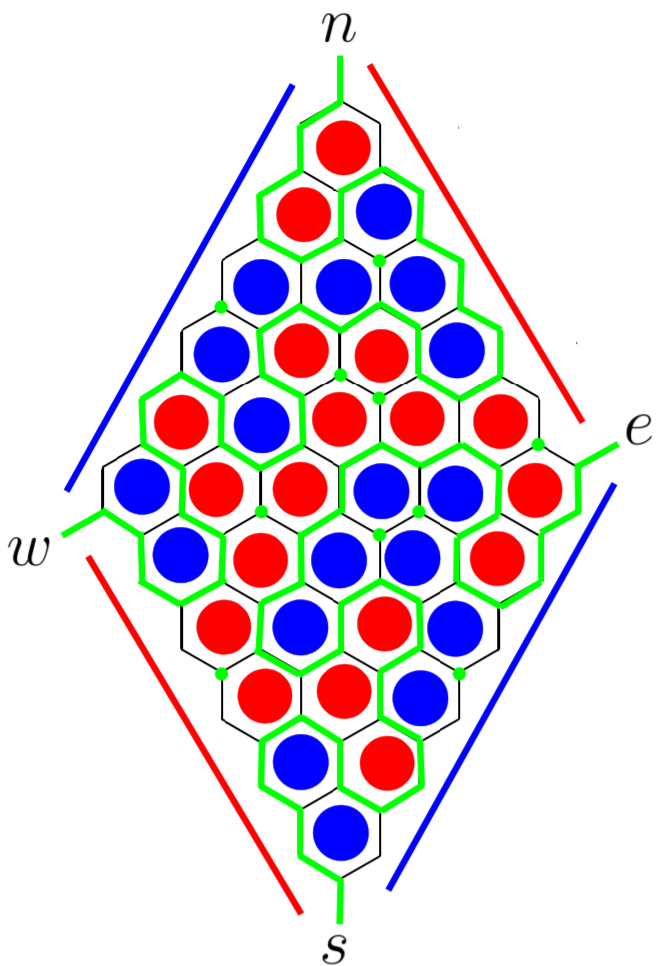}
  \caption{The tours on $\Gamma$}
  \label{gale_full}
\end{subfigure}
\caption{The argument for Theorem \ref{hex_theorem}.}
\label{gale_hex_board}
\end{figure}

Observe that this touring rule determines a unique path; suppose that the tour has proceeded along some edge $d$ and arrives at some vertex $v$.
Two of the three tiles incident to $v$ are those of which $d$ is the common boundary, hence one has a Red stone and the other a Blue stone; the third tile incident to $v$ may have a Red or a Blue stone, but in either case there is exactly one edge $d'$ which satisfies the rule described.

As a consequence of starting from $s$, we have that such a tour will never revisit any vertex of $\Gamma$; however, $\Gamma$ is finite and so the tour must terminate.

\medskip

It is clear from the above that such a tour must describe a subgraph of $\Gamma$ whose vertices have degree at most 2; hence, by Lemma \ref{lemma_graph}, the tour starting at $s$ can only be a simple path, thus ending at either $w, n,$ or $e$.
Observe that each of these vertices is incident to either the NW Blue quadrant or the NE Red quadrant.

If, say, the terminal vertex of the described tour is incident to the NE quadrant, then Red has a winning chain which is incident to the edge tour starting in $s$ and joins the NE side with the SW side of the board; the other case for Blue is analogous.
\end{proof}

\begin{remark}
Observe that this proof does not require any ``right-left" orientation of the tour.

Moreover, we do not even require that the fully marked board is the ending position of some legal play.
\end{remark}

\begin{theorem}\label{jordan_hex}
There is no finite Hex position in which both players have won.
\end{theorem}

This Theorem is usually proved via an application of the Jordan Curve Theorem; we propose another visual proof by \cite{schachner19}.

\begin{proof}[Proof of Theorem \ref{jordan_hex}]
Similarly to the previous proof, we add one distinguished vertex for each side of the board and an extra one.

Supposing for a contradiction that both Red and Blue win, we can use the two winning paths in the ending board position to construct the full graph with 5 vertices, as in figure \ref{K_5}.

However, such a graph cannot be planar; contradiction.

\begin{figure}[h]
\centering
\includegraphics[width=4cm]{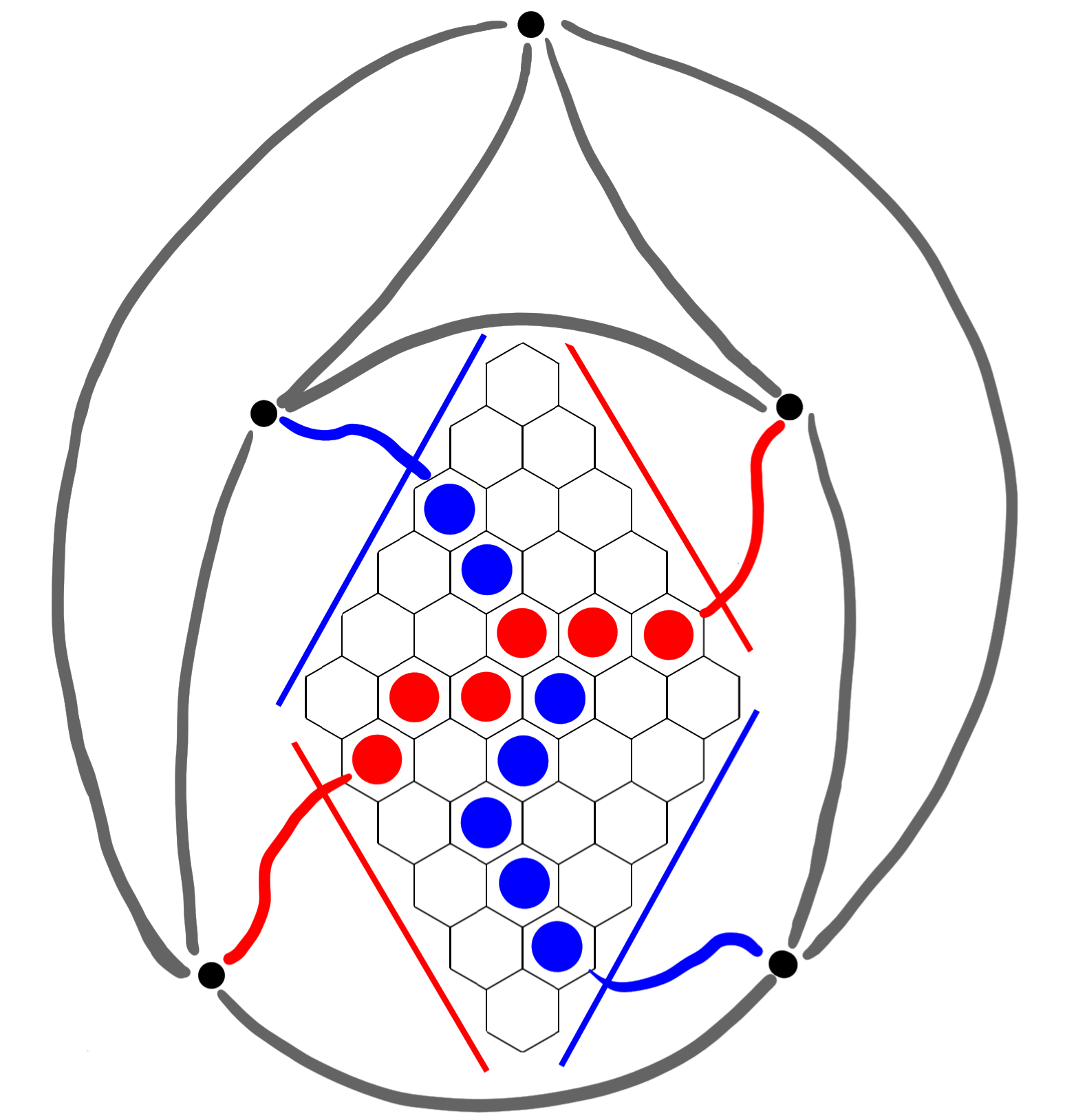}
\caption{Obtaining the graph $K_5$ from a Hex position.}
\label{K_5}
\end{figure}
\end{proof}

\begin{historical_remark}
According to \cite[pp.~7--8]{hayward19}, Piet Hein was aware of this alternative proof; he realised that the (then unproved) Four Colour Theorem implied that there is no planar full graph with 5 vertices. He mentioned this as the foundation for his intuition behind the design of Hex.
\end{historical_remark}

\medskip

From Theorems \ref{hex_theorem} and \ref{jordan_hex} we deduce the following;

\begin{corollary}[No-tie property]\label{no_tie}
All plays of finite Hex end with exactly one winner.\footnote{We can unambiguously say that plays end because they can last only for as many turns as there are tiles on the finite Hex board.}
\end{corollary}

\begin{proposition}[Strategy-stealing]\label{strategy_stealing}
In finite Hex, the second player does not have a winning strategy.
\end{proposition}

\begin{remark}\label{rmk_finite_hex}
We observe the following basic properties of finite Hex.
\renewcommand{\theenumi}{\roman{enumi}}
\begin{enumerate}
    \item\label{extra_stone} It is never disadvantageous for any player to have an extra stone on the board,\footnote{Note how in a \emph{dynamic game} in which pieces can change position like Chess, for instance, having an extra piece that obstructs the escape of one's king can facilitate a checkmate. Or observe that it is possible to construct dynamic game positions in which making more than one move in a turn can put a player at a disadvantage.} as first observed by the discoverer in \cite{nash52}.
    
    \item\label{symmetry_board} There exists a rigid motion of the board $B$ into its \emph{opposite} -- that is, $B$ can be rotated, translated, and possibly reflected, into the board consisting of $B$ with each boundary segment assigned to the opposite player.
\end{enumerate}
\end{remark}

\begin{proof}[Proof of Proposition \ref{strategy_stealing}]
Suppose for a contradiction that Blue, the second player, has a winning strategy $\sigma$.

Let $\sigma'$ be the \emph{opposite} strategy of $\sigma$, that is the corresponding strategy on the opposite board, in the sense of Remark \ref{rmk_finite_hex}.\ref{symmetry_board}.

Now, let Red, the first player, place a stone on an arbitrary tile $t$ and then play according to $\sigma'$ ignoring the extra tile $t$, so to pretend to play as second. If, at some turn, $\sigma'$ prescribes Red to play on $t$, then Red can place a stone on another arbitrary tile $t'$ and continue to play ignoring it.

Note that the strategy for Red is winning, exactly because $\sigma$ is a second-player winning strategy and because ignoring his initial move can not obstruct his win, as observed in Remark \ref{rmk_finite_hex}.\ref{extra_stone}.

But now both players can have a winning strategy; this is a contradiction by Corollary \ref{no_tie}.
\end{proof}

Observe how the symmetry property of Remark \ref{rmk_finite_hex}.\ref{symmetry_board} is crucial for the argument above; it would not be possible to let Red use Blue's strategy otherwise, as the board, and so the relevant winning conditions, would have ``looked different" to each of them.

\bigskip


\bigskip

\label{diss_asym_board}
A natural example of boards that satisfy condition \ref{rmk_finite_hex}.\ref{extra_stone}.~but fail to satisfy condition \ref{rmk_finite_hex}.\ref{symmetry_board}.~is that of asymmetric $m\times n$ Hex boards.

Aiming to balance the game of Hex in favour of the otherwise losing second player, it seems reasonable to use asymmetric boards; however, it turns out that the player playing along the short edge of the board has such a strong advantage to be able to win even when playing second.

In fact, it is possible to pair the tiles of any $(n+1)\times n$ Hex board so that the player joining the long edges wins by choosing, at each turn, the cell paired to the last cell taken by the opponent, as observed by \cite[\textsection 8]{gardner88}. When making the initial move, the player following this strategy can just make an arbitrary move as in the proof of Proposition \ref{strategy_stealing}.

This is clearer in the example of figure \ref{asymmetric_board}.

\begin{figure}[h]
\centering
\includegraphics[width=4cm]{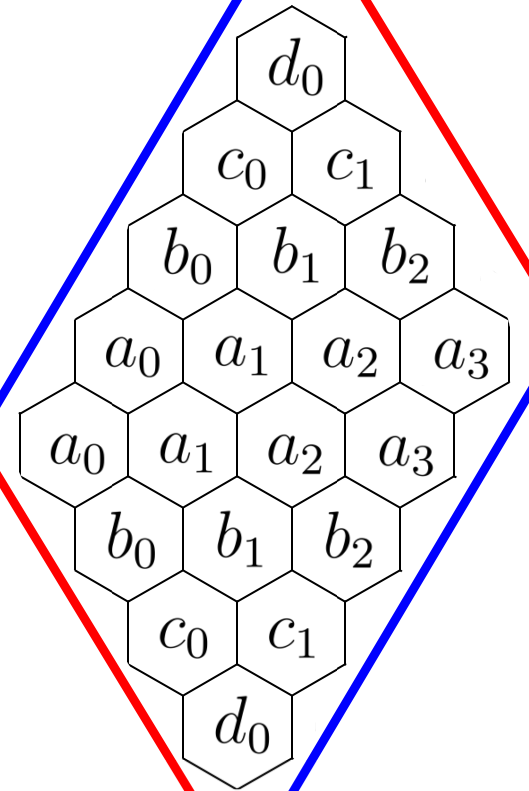}
\caption{The pairing strategy on the $5\times 4$ Hex board.}
\label{asymmetric_board}
\end{figure}

\medskip

We will see in section \ref{hex_empty_draw} that a direct generalisation of this strategy will provide a drawing strategy for the second player in the game of Infinite Hex starting from an empty board.


\section{Infinite Hex}

It is immediate to generalise the board of Hex to the infinite plane; Infinite Hex is a class of games in which the two players, still Red and Blue, alternate placing their respective stones on the hexagons that tile the infinite plane.

Note that the hexagonal tiles of the Infinite Hex board are in bijection with $\mathbb{Z}\times\mathbb{Z}$, so that the players, at each turn, can choose to mark one of at most countably many tiles; hence, the game tree of any game of Infinite Hex can always be embedded in the full countable tree $\omega^{<\omega}$, as we need.

In other words, any position of Infinite Hex, including the positions with infinitely many stones already placed, is the initial position of a well-defined game of Infinite Hex.

\medskip

We next discuss the definition of winning condition, which is less straightforward.

\subsection{Winning condition}\label{winning_condition}

Since the Infinite Hex board has no edges to join, it is not immediately clear how to define a winning condition; we would like to ``join the opposite infinite ends of the plane", so we will now formalise this intuition.

\medskip

Given a position $p$, we aim to determine whether some given player has won.
Moreover, we need to ensure that not both players win in a single play, so to comply with condition \ref{def_game}.\ref{winning_set}.~in the definition of game.

\medskip

We first need some definitions, which can be easily repeated for Blue; let $\widetilde{H}$ be the Infinite Hex board and fix some position $p$, that is an assignment of each cell of $\widetilde{H}$ to either or no player, together with a turn indicator.

An infinite Red path, or Red $\mathbb{Z}$\emph{-chain}, is a function $r:\mathbb{Z}\rightarrow \widetilde{H}$ such that all the cells in the image of $r$ are marked by Red at $p$ and $r(n)$, $r(n+1)$ are adjacent in $\widetilde{H}$ for all $n\in \mathbb{Z}$.

We can now approach the condition to determine whether an infinite Red path is \emph{winning}.\footnote{We could require a winning path $r$ to be \emph{geodesic}, so that the sub-path between any two tiles in $r$ is the shortest Red path between such tiles; however, it will be clear that if Red has a winning path, then Red also has a geodesic winning path.}

\medskip

Pick some cell $h_0\in\widetilde{H}$ and impose an orthogonal system of reference with origin in $h_0$ and axes directed from South to North and from West to East, as in figure \ref{fig_empty_axes}.

Observe that, with some flexibility in the identification of the South-North columns, we can use these axes to obtain unique coordinates for each tile in $\widetilde{H}$; this is best described by figure \ref{fig_axes_coords}.

\begin{figure}[h]
\centering
\begin{subfigure}{.49\textwidth}
  \centering
  \includegraphics[width=.9\linewidth]{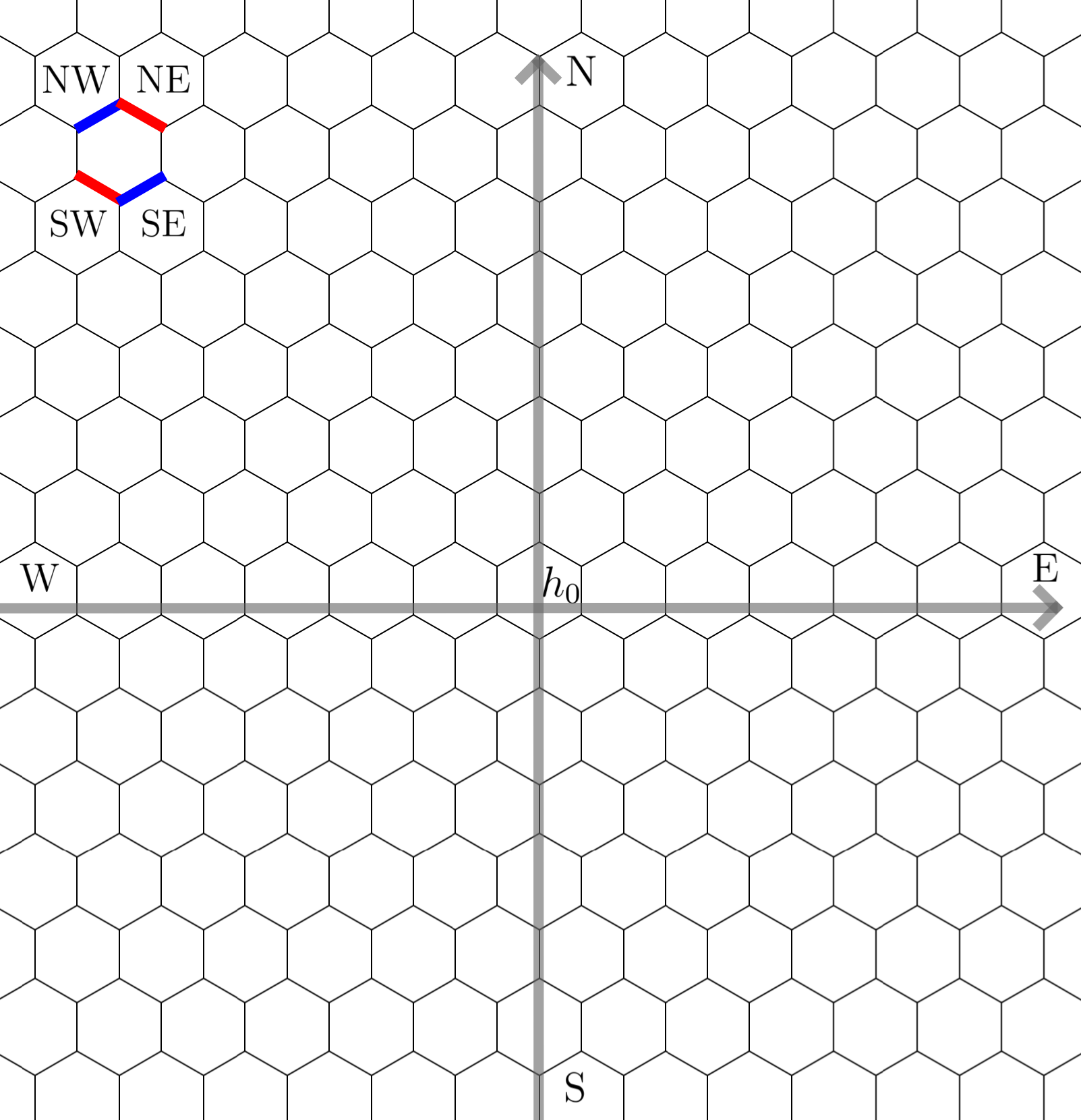}
  \caption{The origin $h_0$}
  \label{fig_empty_axes}
\end{subfigure}
\begin{subfigure}{.49\textwidth}
  \centering
  \includegraphics[width=.9\linewidth]{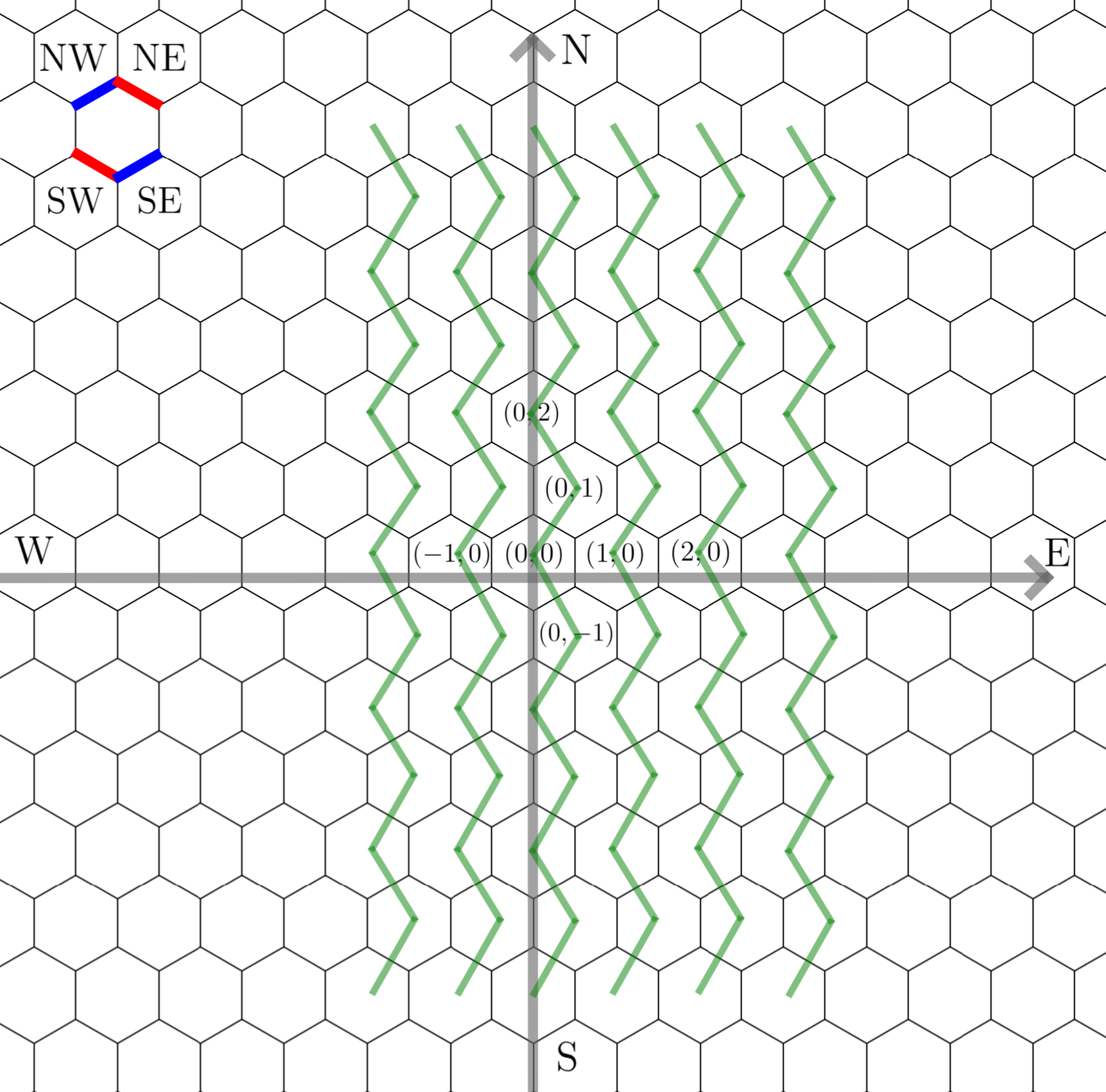}
  \caption{Columns in green}
  \label{fig_axes_coords}
\end{subfigure}
\caption{The reference system for the Infinite Hex board.}
\label{fig_ref_sys}
\end{figure}

It can be useful to keep in mind the sketch of a general Red winning path in figure \ref{sketch}, while reading the following definitions.

\begin{figure}[h]
\centering
\includegraphics[width=4cm]{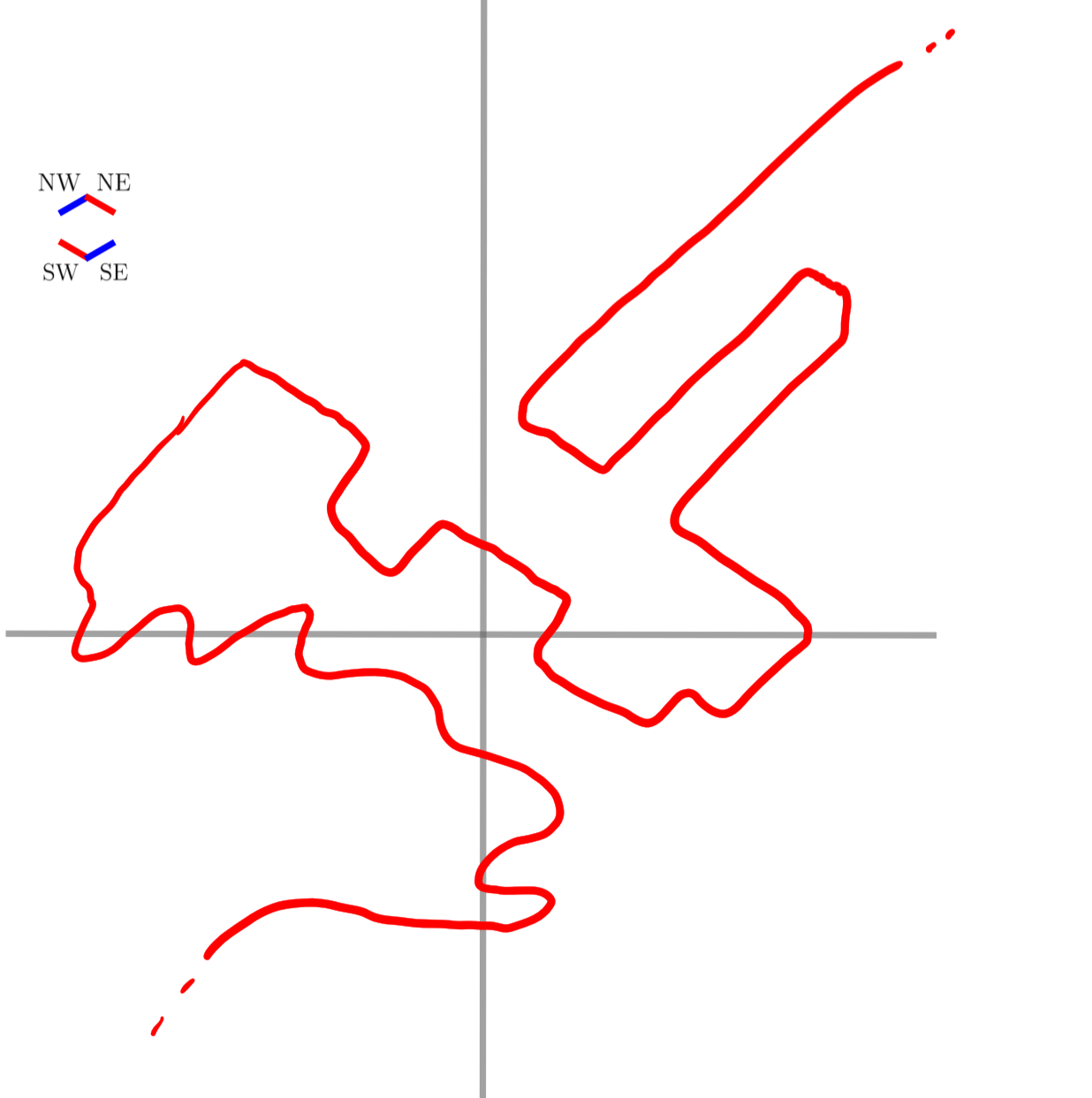}
\caption{Sketch of a general Red winning path.}
\label{sketch}
\end{figure}

Let $r$ be an infinite Red path.

\begin{definition}\label{winning_path_wrt}
Given a system of reference centered at $h_0$ as above, we say that $r$ is \emph{winning with respect to} $h_0$ if $r[\mathbb{Z}^+]$ is eventually contained in the NE-quadrant and $r[\mathbb{Z}^-]$ is eventually contained in the SW-quadrant of such reference system.

More precisely, we say that $r$ is winning with respect to $h_0$ if there is some $M\in\mathbb{N}$ such that if $m\ge M$, then the coordinates of $r(m)$ are both positive, while the coordinates of $r(-m)$ are both negative.
\end{definition}

\begin{definition}\label{winning_path}
We say that $r$ is a \emph{winning} Red path if it is winning with respect to every choice of centre $h$.
\end{definition}

The winning condition for a Blue path is analogous, substituting NW for NE and SE for SW in Definition \ref{winning_path_wrt}, as hinted by the key in figure \ref{fig_ref_sys}.

\medskip

Observe that the arbitrary choice of the origin for the reference system is necessary to avoid considering as wins paths which are bounded in one of the axes' directions; otherwise we would not be able to assign consistently an outcome to the position in figure \ref{fig_tans}, in which the Red and Blue paths are unbounded towards North and South, but bounded in the WE-direction.

In fact, each path is winning for the relevant player with respect to several cells on the paths themselves, while neither path is winning with respect to the starred tile, for instance.

\begin{figure}[h]
\centering
\includegraphics[width=8cm]{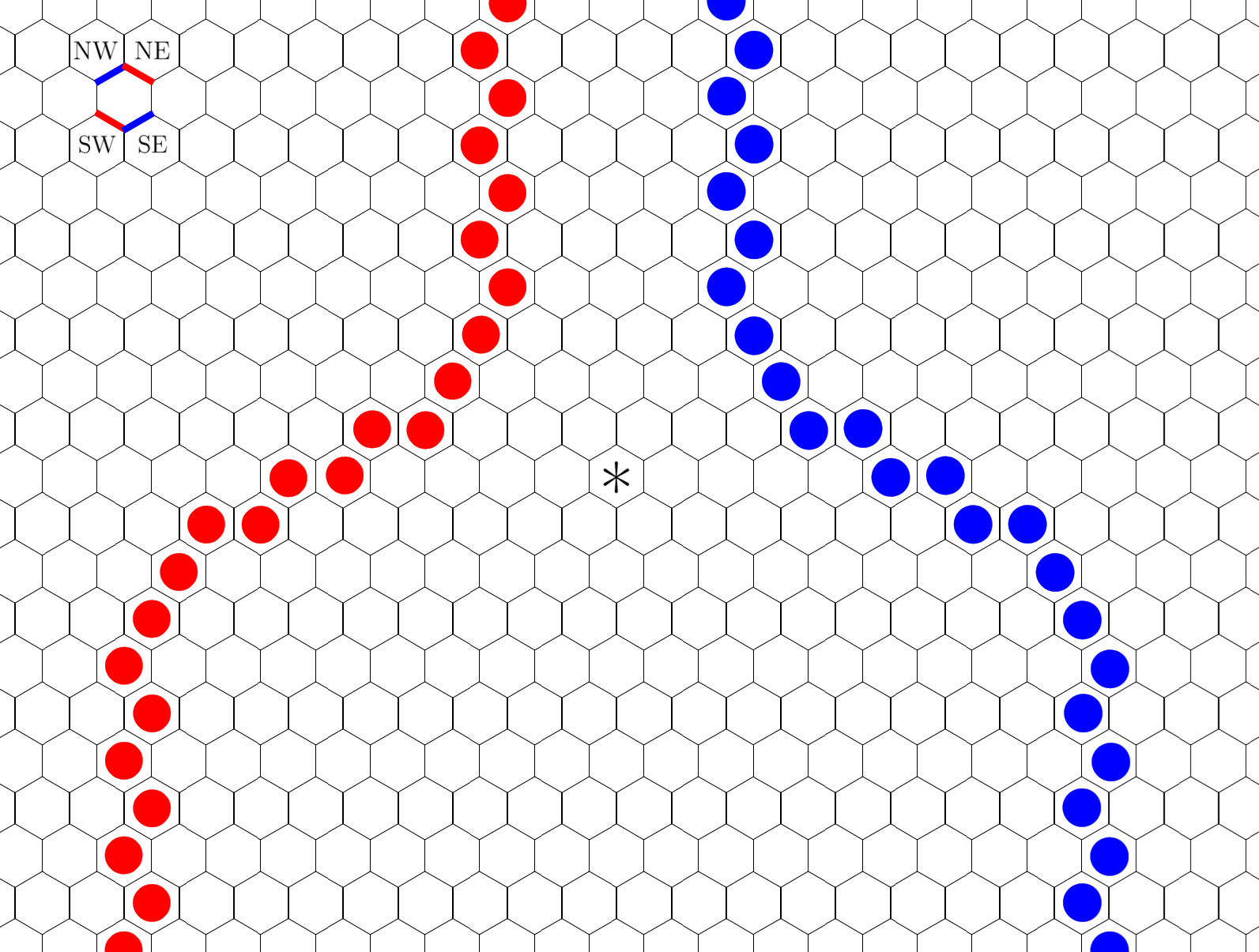}
\caption{Infinite paths bounded in the WE-direction.}
\label{fig_tans}
\end{figure}

\begin{remark}
By insisting in Definition \ref{winning_path_wrt} that the coordinates of $r(m)$ have to be strictly positive, and of $r(-m)$ strictly negative, then we can define a path to be winning, equivalently to Definition \ref{winning_path}, if it is winning with respect to each of its tiles; such a path cannot then be bounded in either of the axes' directions.
\end{remark}

\bigskip

We now show that the condition above is well-defined in the sense of defining disjoint winning conditions for the two players, as per Definition \ref{def_game}.\ref{winning_set}.

\begin{proposition}\label{not_both_win_inf_hex}
Not both players can win in a position of Infinite Hex.
\end{proposition}
\begin{proof}
Suppose for a contradiction that $p$ is a position of Infinite Hex such that both Red and Blue have winning paths, respectively, $r$ and $b$.

\medskip

Pick some origin $h_0\in \widetilde{H}$ and let $x,y:\widetilde{H}\rightarrow \mathbb{Z}$ be the projections on the WE- and SN-axes centered at $h_0$, so that the coordinates of a tile $h\in\widetilde{H}$ are $(x(h),y(h))$.

Note that both paths $r$ and $b$ are winning, for the relevant players, with respect to $h_0$; hence there is some $M\in\mathbb{N}$ such that for $m\ge M$ we have that\footnote{$M$ is the maximum of the two bounds given by Definition \ref{winning_path} for $r$ and $b$.}
\[x(r(m)), y(r(m)) >0;~~ x(r(-m)), y(r(-m)) <0;
\]
\[x(b(m))<0,~ y(b(m)) >0;~~ x(b(-m))>0,~ y(b(-m)) <0.
\]

Consider now the finite paths $r([-M,M])$ and $b([-M,M])$; we claim that they cannot be both connected.

Let $H$ be the smallest finite Hex board centered at $h_0$ and with the usual orientation, such that it contains both $r([-M,M])$ and $b([-M,M])$.

Now, let $r',b'$ be the maximally connected finite sub-paths of $r,b$ that extend, respectively, $r([-M,M]),b([-M,M])$ in $H$; we claim that they are both winning in $H$, according to the usual finite Hex winning conditions, leading to a contradiction.

\medskip

Observe that, by choice of $M$, the ends of the path $r([-M,M])$, i.e.~$r(-M)$ and $r(M)$, are included in, respectively, the NE- and SW-quadrants of the reference system centered at $h_0$.

As $r'$ extends $r([-M,M])$ without getting out of the relevant quadrants, then $r'$ is a Red winning path in the game of finite Hex played on $H$.

An analogous argument can be made to show that $b'$ is a Blue winning path in the game of finite Hex on $H$, which leads to the contradiction needed, by Theorem \ref{jordan_hex}.
\end{proof}

\medskip

We conclude this section by presenting two ``extreme" positions that motivated the choice of the definition of winning condition given and allowed to discard several other possible definitions.
These positions are draws according to our definition, but can easily be considered wins by either or both players with weaker conditions.

\medskip

In particular, observe how the Red path in figure \ref{fig_double_prongs} is unbounded in the NE-SW-direction, as desired, but also goes through the East-West pink band infinitely many times.

\begin{figure}[h]
\centering
\includegraphics[width=9cm]{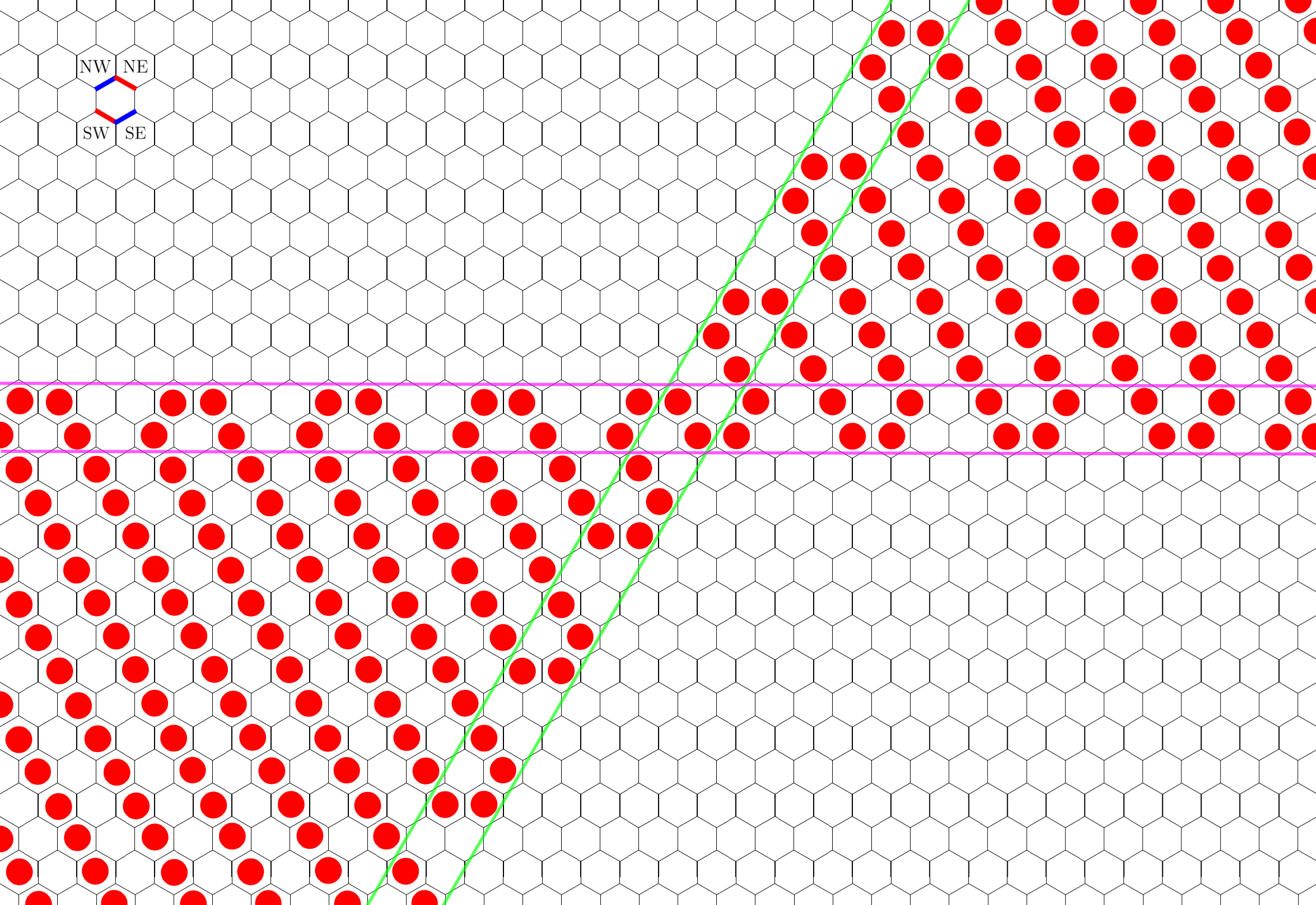}
\caption{Infinite Red path unbounded towards East, West, NE, and SW.}
\label{fig_double_prongs}
\end{figure}

Moreover, the Red and Blue paths in figure \ref{fig_double_spiral} are unbounded in all directions; any coherent winning condition that identifies one of them as winning has to do the same with the other.

\begin{figure}[H]
\centering
\includegraphics[width=9cm]{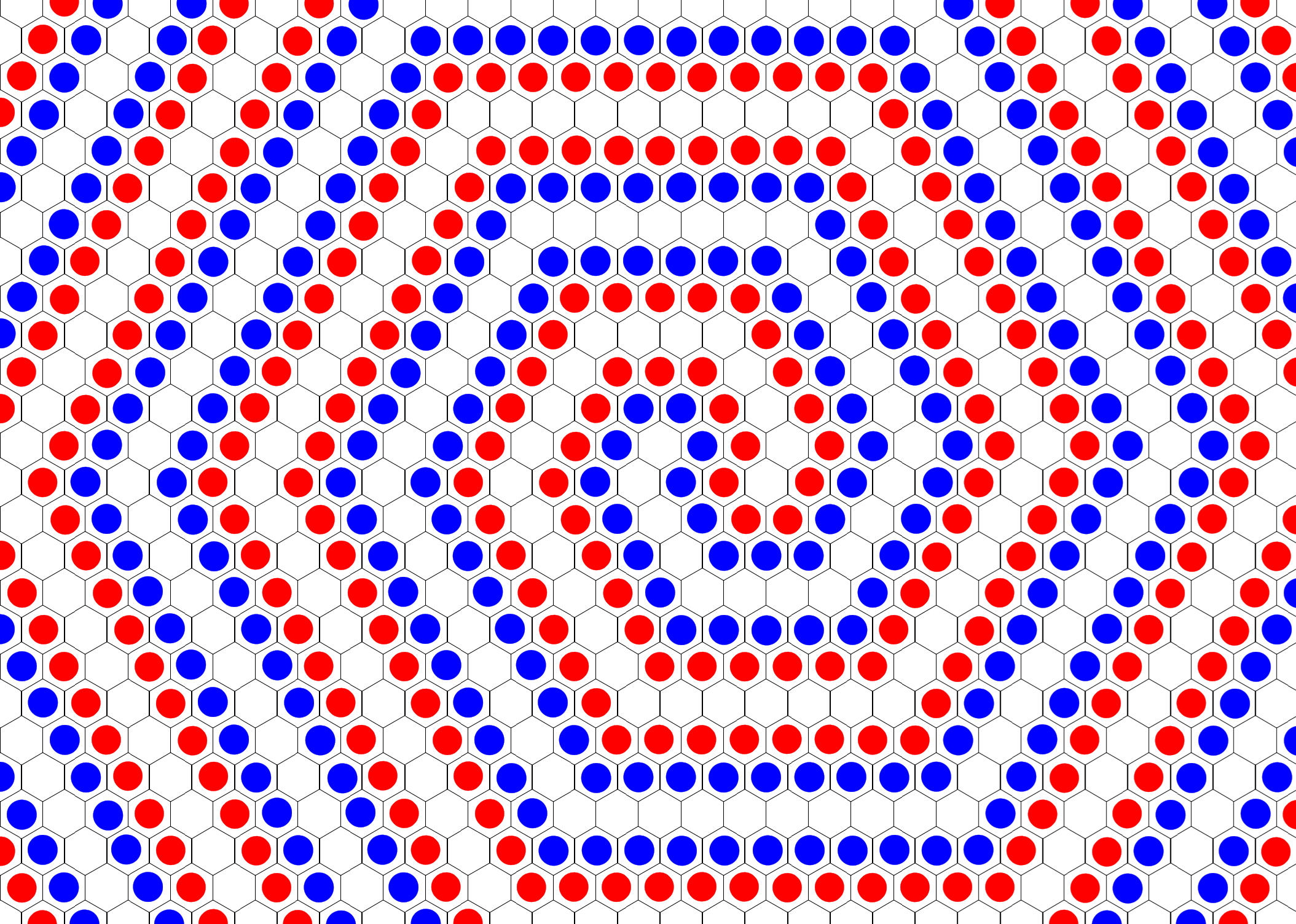}
\caption{Two spiralling Red and Blue paths.}
\label{fig_double_spiral}
\end{figure}

\subsection{Normal Infinite Hex is a draw}\label{hex_empty_draw}

In section \ref{winning_condition} we provided a definition of winning positions, rather than plays; this is because the theory of transfinite game values developed in section \ref{sec_game_values} only applies to open games, and so we are mostly interested in applying such theory to the sub-class of games of Infinite Hex which are open for some designated player.

\medskip

Nevertheless, we can consistently extend Definition \ref{winning_path} to all plays of Infinite Hex by dropping the finiteness condition; a player wins if he has a strategy with which he can eventually construct\footnote{In at most countably many moves, as plays still have order type at most $\omega$.} an infinite path which is winning according to Definition \ref{winning_path}, however his opponent plays.

\medskip

Equipped with this new definition, we can ask the question: who wins from the initially empty position?
We are particularly interested in clarifying whether the advantage of the first player in finite Hex carries on to Normal Infinite Hex, the game of Infinite Hex with the empty board as initial position.

We will show that this is not the case and that both players can force at least a draw from the empty board. However, we later conjecture in section \ref{biased_inf_hex} that the first player may still win if given some advantage, possibly of only one tile.

\medskip

 Let $\theta$ be a pairing of the tiles of the Infinite Hex board as depicted in figure \ref{mirroring_empty}; cells in the same pair have the same label.
 Note that $\theta$ pairs each cell with exactly one other cell.

Let $\mu$ be the second-player strategy, called \emph{mirroring strategy}, induced by $\theta$ as follows;
whenever Red, the first player, places a stone on a tile $h$, then $\mu$ prescribes that Blue places a stone on the the tile paired with $h$, according to $\theta$.

\begin{figure}[h]
\centering
\includegraphics[width=8cm]{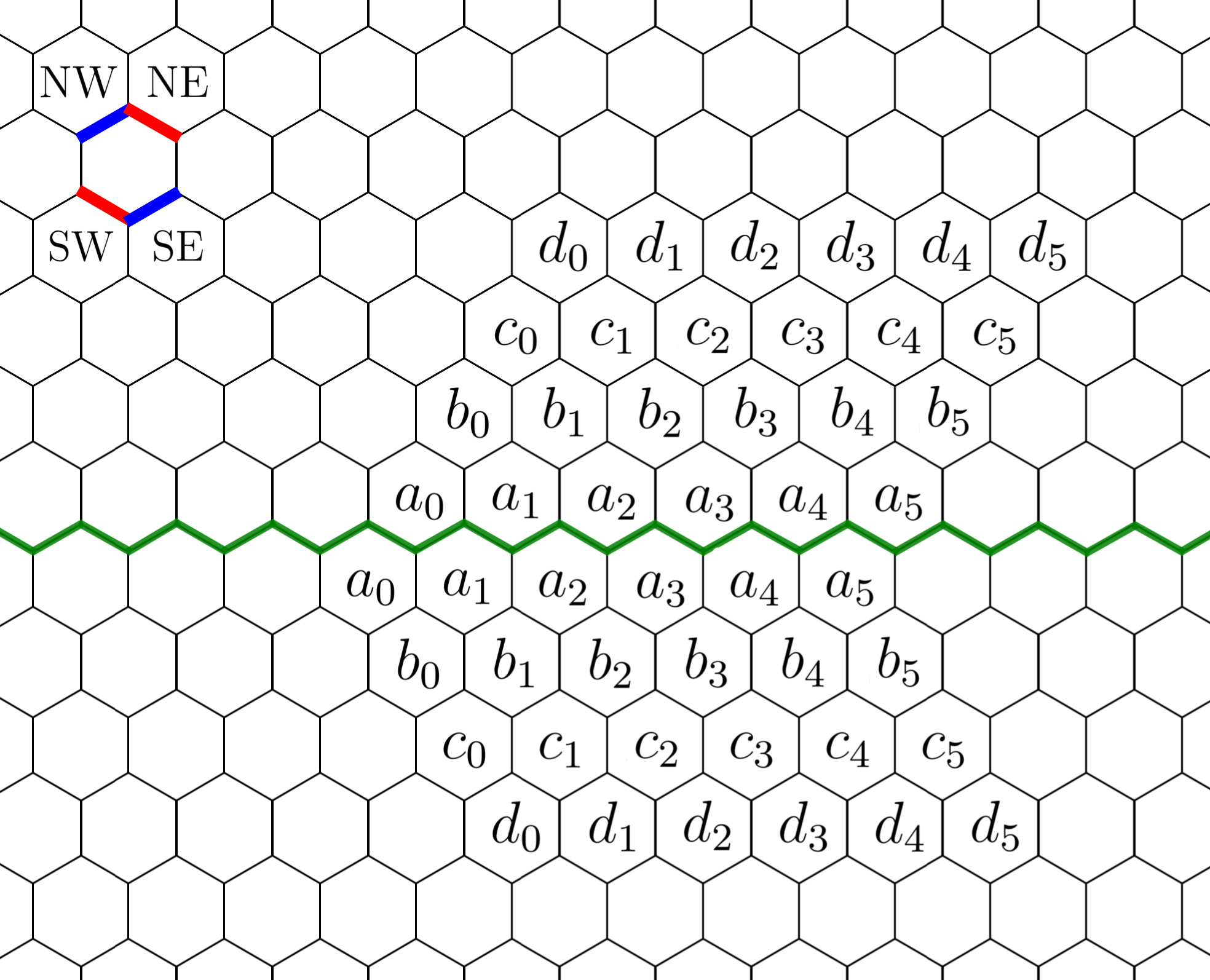}
\caption{The pairing $\theta$ for $\mu$, with the mirroring axis in green.}
\label{mirroring_empty}
\end{figure}

Observe that $\mu$ is a well-defined strategy for Normal Infinite Hex. If Blue follows it since her first move, then Red will not be able to mark over the course of the play both tiles of any pair given by $\theta$; hence, at each turn, the tile indicated by $\mu$ to Blue will always be empty.
Moreover, note that, following $\mu$, Blue always places a stone in the opposite half-plane to where Red played immediately before, with respect to the East-West mirroring axis.

\begin{remark}
Similarly to the Strategy-stealing argument of Proposition \ref{strategy_stealing}, we can adapt $\mu$ to be a first-player strategy.

Let $\theta'$ be the pairing given by reflecting the pairing $\theta$ across some North-South axis.\footnote{Observe that $\theta'$ is essentially unique, given $\theta$.}
Let $\mu'$ be the strategy that prescribes Red to make an arbitrary initial move and then follow the (second-player) mirroring strategy induced by $\theta'$, possibly making arbitrary moves when the mirroring strategy prescribed to place a stone on a tile already marked by Red; this is because the mirroring strategy ignores the initial move and the subsequent arbitrary moves by Red.

Since the reflection of all possible winning paths for Red across a North-South axis gives exactly all possible winning paths for Blue, it is then clear that showing that $\mu$ is a drawing strategy for Blue implies that $\mu'$ is a drawing strategy for Red.
\end{remark}

\begin{remark}
The pairing in figure \ref{mirroring_empty} is called mirroring because it is not a symmetry and does not represent a rigid motion of the board. In particular, we obtain the Northern half-plane as the reflection across the mirroring axis of the Southern half-plane, followed by a translation by half a cell towards East; this translation is in effect ``tearing" the plane.
\end{remark}

\begin{remark}\label{mirror_condition}
As already observed, when Blue plays according to $\mu$, Red can mark at most one tile in each pair given by $\theta$.

It follows that if Red constructs some winning path $r$, then $r$ contains at most one tile in each pair given by $\theta$, so that Blue can mark all the cells of $\theta(r)$, the \emph{mirroring} of $r$ with respect to $\theta$; we call this the \emph{mirroring condition} on $r$.
\end{remark}

Before proving the main result of this section, we make two observations.

Firstly, any Red path that crosses the mirroring axis of $\mu$ can only do so with two adjacent cells in the NW-SE-direction, that is the ``wrong" direction for Red. This will be key for the argument, and is clear by inspection of figure \ref{fig_mirroring_1_int}.

\begin{figure}[h]
\centering
\includegraphics[width=6cm]{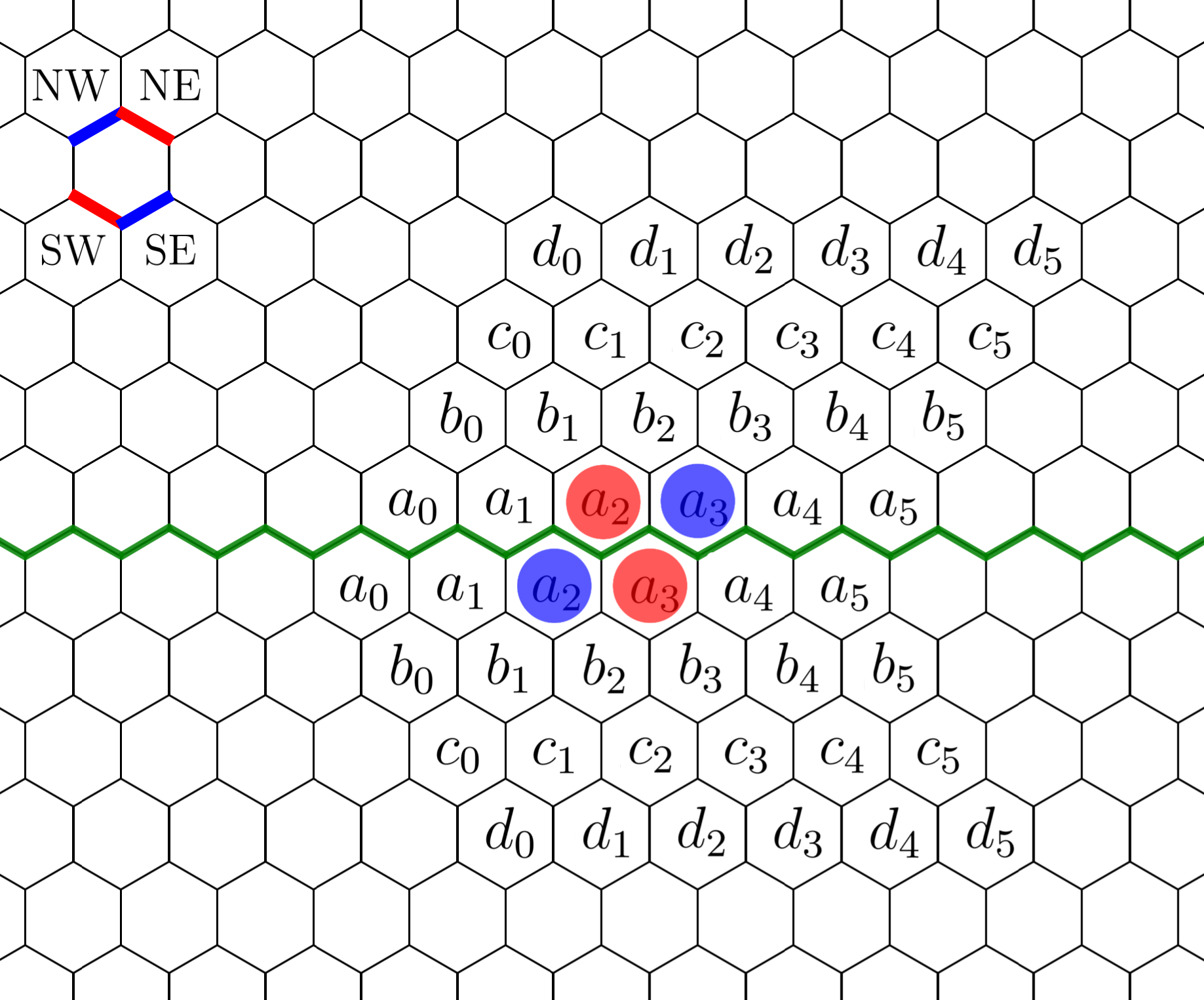}
\caption{Single Red intersection of the mirroring axis.}
\label{fig_mirroring_1_int}
\end{figure}

Secondly, $\mu$ allows Blue to win, as long as Red plays appropriately bad.
Observe that the position in figure \ref{fig_mirroring_blue_win} is winning for Blue and satisfies the mirroring condition of Remark \ref{mirror_condition} for the Red path.

\begin{figure}[h]
\centering
\includegraphics[width=7cm]{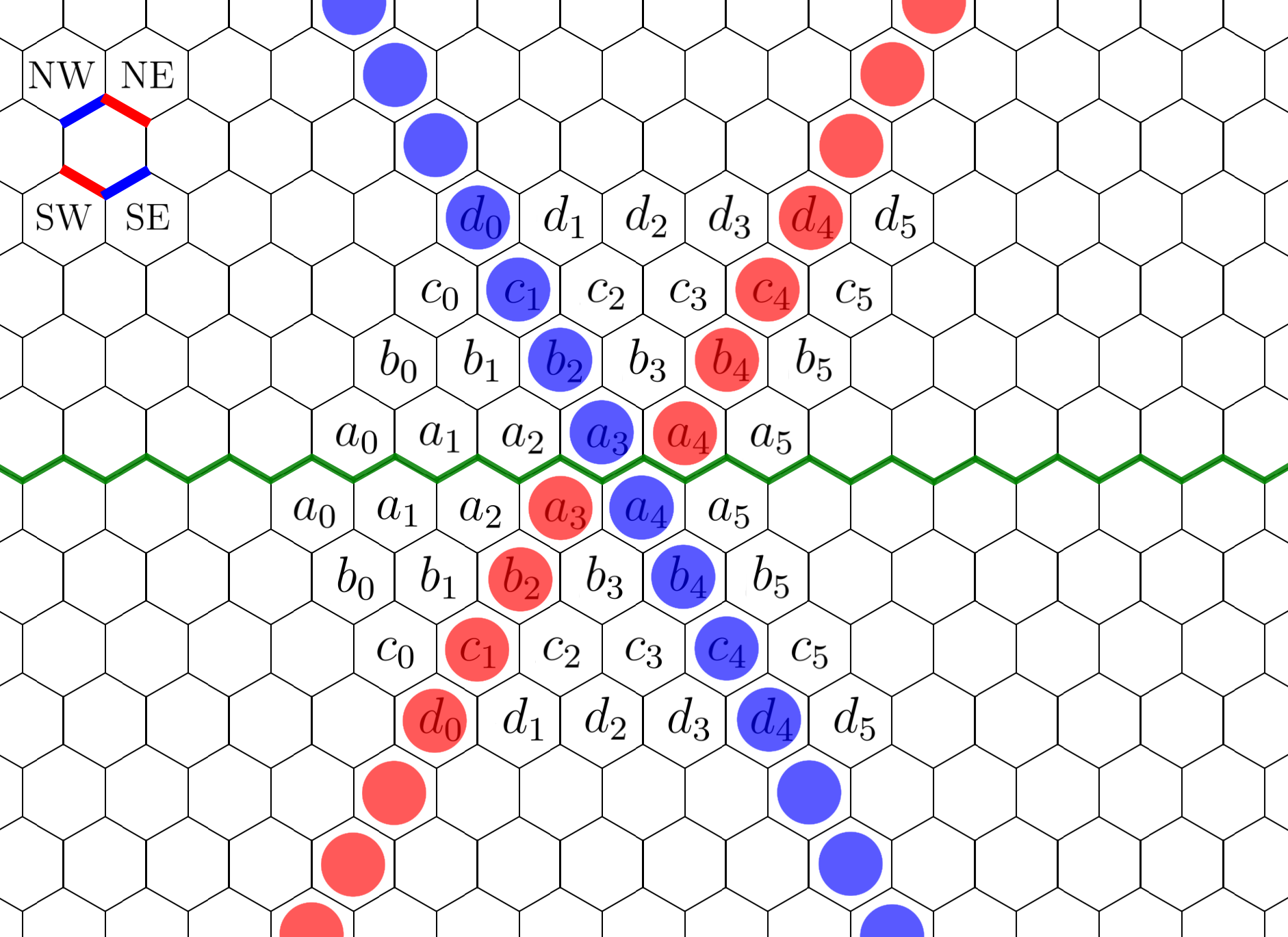}
\caption{Blue wins following the mirroring strategy $\mu$.}
\label{fig_mirroring_blue_win}
\end{figure}

\begin{proposition}\label{hex_drawing_strategy}
In Normal Infinite Hex, $\mu$ is a drawing strategy for the second player.
\end{proposition}
\begin{proof}
Suppose for a  contradiction that Red has a winning strategy $\sigma$.

Letting Red play according to $\sigma$ and Blue according to $\mu$, Red will be able to construct a winning path\footnote{Such path is only constructed after infinitely many turns, as described at the beginning of this section.} $r$ which satisfies the mirroring condition of Remark \ref{mirror_condition}, so that $r$ does not contain both tiles of any of the pairs of $\theta$.

\medskip

Now, we claim that $r$ cannot be a winning path.
Observe that $r$ has to cross the East-West mirroring axis at least once.
So, suppose that $r$ crosses the mirroring axis exactly once. We then expect a single Red intersection as in figure \ref{fig_mirroring_1_int}.

Since $r$ is winning for Red, then the semi-infinite sub-path of $r$ in the Northern half-plane will eventually go towards North-East; by definition of $\mu$, Blue will correspondingly construct the mirrored sub-path in the Southern half-plane that eventually goes towards South-East, as in figure \ref{one_intersec}.

However, we also have that the sub-path of $r$ in the Southern half-plane must eventually go towards South-West, but if that is the case, then we can draw a finite Red path\footnote{This extra Red finite path is drawn directly on the board for the sake of the geometric argument; it is not part of the play.} as in figure \ref{one_intersec2} so that we can identify a simple closed curve determined by adjacent Red stones, such curve is a Jordan curve which is intersected by a \emph{connected} Blue path; this is a contradiction by the Jordan Curve Theorem.

\begin{figure}[h]
\centering
\begin{subfigure}{.49\textwidth}
  \centering
  \includegraphics[width=.82\linewidth]{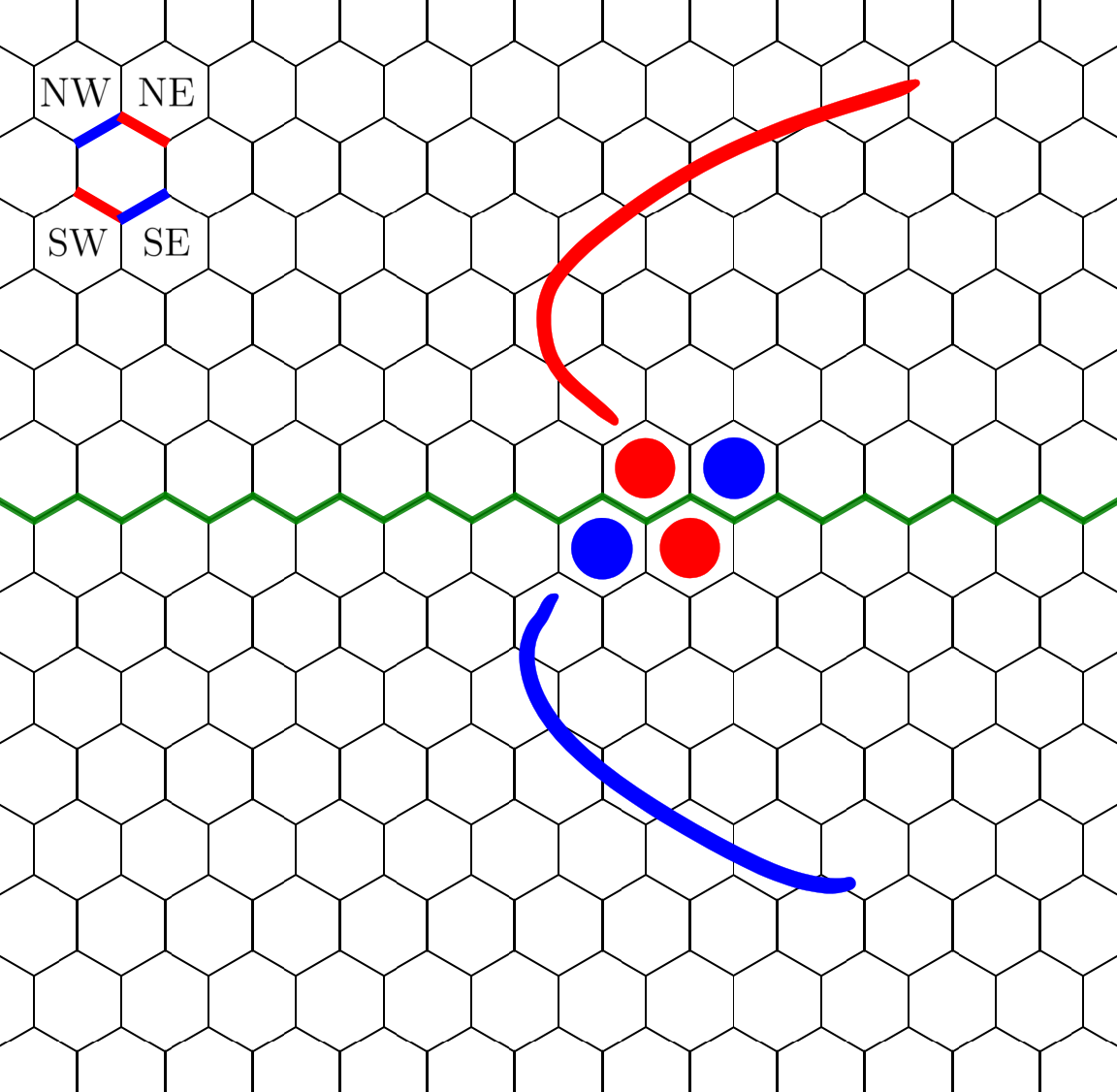}
  \caption{Mirroring an half-infinite Red path}
  \label{one_intersec}
\end{subfigure}
\begin{subfigure}{.49\textwidth}
  \centering
  \includegraphics[width=.9\linewidth]{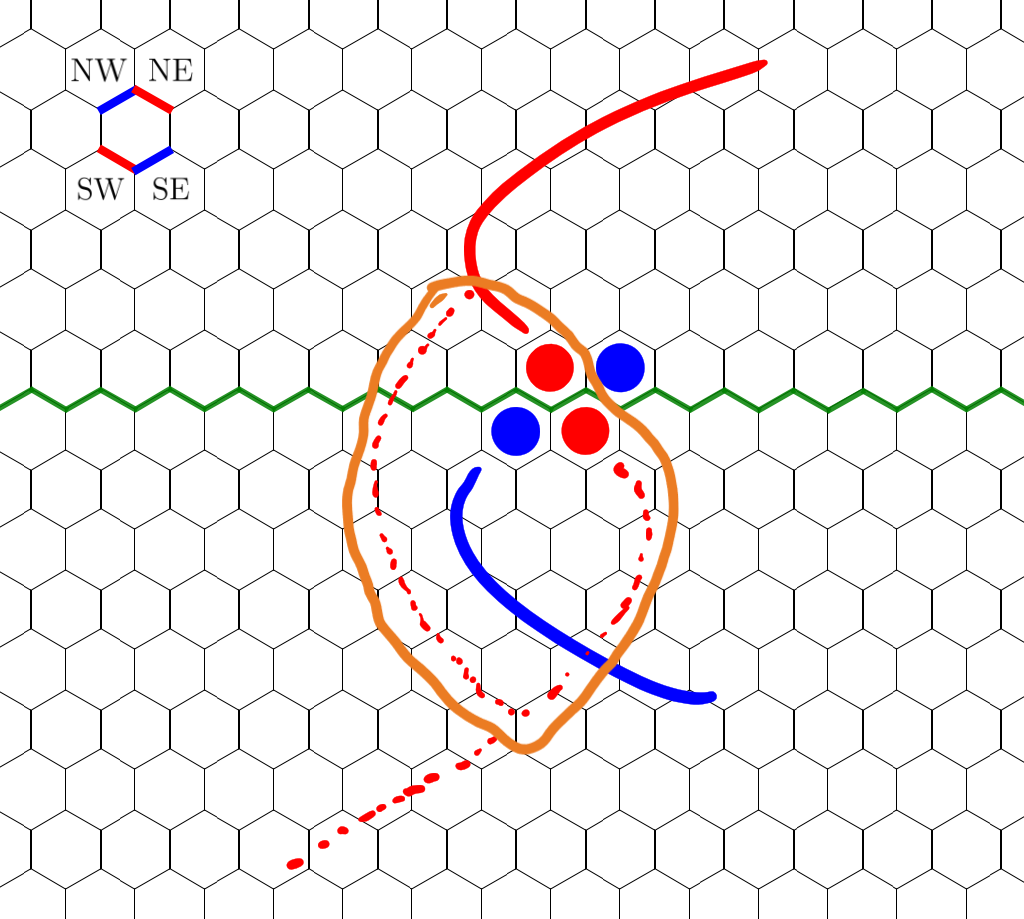}
  \caption{The Jordan curve in orange}
  \label{one_intersec2}
\end{subfigure}
\caption{Application of the Jordan Curve Theorem.}
\label{one_intersec0}
\end{figure}


Suppose now that $r$ crosses the mirroring axis more than once; $r$ cannot cross it exactly twice because it is a winning path, and in particular a connected path, thus assume that $r$ crosses the mirroring axis exactly thrice.

We expect to see three Red intersections of the mirroring axis as in figure \ref{3intersec}; these were discussed before in figure \ref{fig_mirroring_1_int}.

\begin{figure}[H]
\centering
\begin{subfigure}{.49\textwidth}
  \centering
  \includegraphics[height=4.2cm]{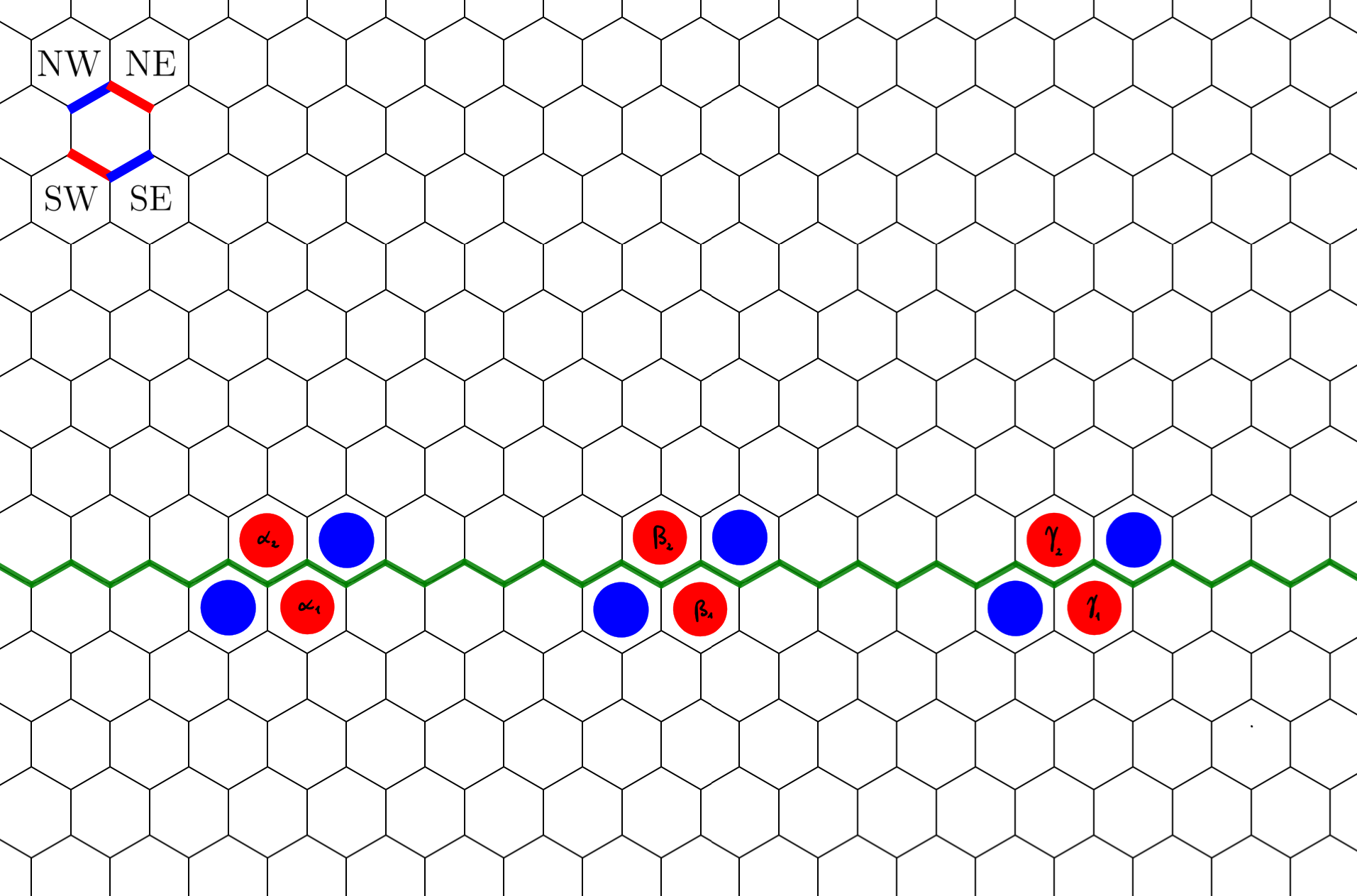}
  \caption{Three intersections}
  \label{3intersec}
\end{subfigure}
\begin{subfigure}{.49\textwidth}
  \centering
  \includegraphics[height=4.2cm]{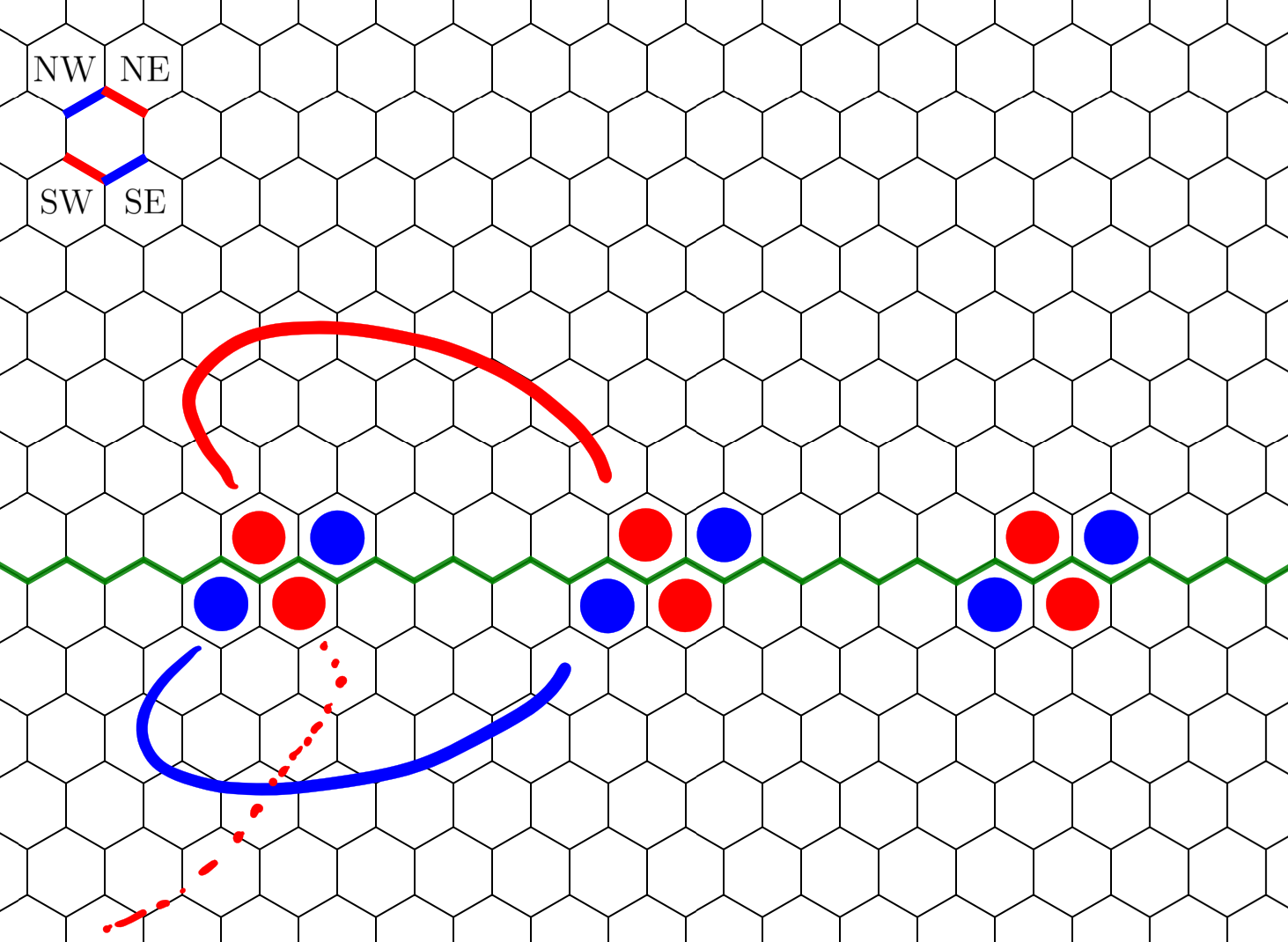}
  \caption{Contradiction}
  \label{3intersec2}
\end{subfigure}
\caption{A Red path cannot cross the mirroring axis three times.}
\label{3_intersec0}
\end{figure}


Observe that the only way to construct such a path $r$ is to join with Red finite paths the ends labelled $\alpha_2$ with $\beta_2$, and $\beta_1$ with $\gamma_1$, and joining appropriate semi-infinite Red paths to the ends labelled $\alpha_1$ and $\gamma_2$.

Such a construction is not possible due to the mirroring condition of Remark \ref{mirror_condition}; observe in figure \ref{3intersec2} how joining two ends in the same half-plane prevents one end in the opposite half-plane to be extended indefinitely.

These observations clearly hold by applying the Jordan Curve Theorem similarly to the above.

\medskip

Observe that the argument for $r$ intersecting the mirroring axis three times can be easily adapted to any position in which $r$ intersects the mirroring axis a larger (odd) number of times.

Since $r$ can only intersect the mirroring axis finitely many times, as it is a winning path, we then conclude that $r$ cannot intersect the mirroring axis more than thrice, either. The proof is complete.
\end{proof}

\begin{remark}
It is possible to prove Proposition \ref{hex_drawing_strategy} via a more involved argument by contradiction which reduces a Red winning path to a winning path for the disadvantaged player of a game of finite Hex on an asymmetric board.

In this sense, we can think of Infinite Hex as generalising the asymmetric variant of finite Hex, rather than the standard finite version.
\end{remark}

\section{Positional and stone-placing games}\label{sec_positional_games}

\subsection{Positional games}\label{subsec_positional_games}

Many popular board games belong to the class of positional games, which share the property of extra moves being always advantageous as in Hex,\footnote{See Remark \ref{rmk_finite_hex}.\ref{extra_stone}.} and whose simplest representative is the game of Tic-tac-toe.

Even though universally popular, we mention the rules of this game; in Tic-tac-toe two players, Naughts and Crosses, alternatively place their marks on a $3\times 3$ square board. The first player to mark three squares on a straight line wins, i.e.~the first player to mark all 3 elements of any of 8 distinguished subsets of the board wins.

Note that the sets of board cells that allow the players to win are the same for both players.

\medskip

We define such ``Tic-tac-toe-like games" as in \cite{hefetz14}; in order to additionally allow for infinite games, we define the following structure.

\begin{definition}
A \emph{generalised hypergraph} is a pair of sets $(B,\mathcal{F})$, such that $B$ is the collection of \emph{vertices}, and $\mathcal{F}\subset\mathcal{P}(B)\setminus\emptyset$ is the collection of \emph{hyperedges}.
\end{definition}

Note that from now on we will not distinguish between generalised hypergraphs and \emph{hypergraphs}, which are usually defined stressing the finiteness of both the vertex and hyperedge sets.

\medskip

\begin{definition}\label{def_pos_game}A \emph{positional game} played on a hypergraph $(B,\mathcal{F})$ is a game such that the players alternatively mark previously unchosen elements of the vertex set $B$, which represents the set of cells of the game board that are unmarked at the initial position.

The first player to mark all the vertices of some hyperedge in $\mathcal{F}$ wins; here $\mathcal{F}\subset \mathcal{P}(B)\setminus \emptyset$ is called the collection of \emph{winning sets}.
\end{definition}

Recall from Remark \ref{rmk_position} that a \emph{position} in a positional game is an assignment of some of the board cells to the players, together with the history of moves and a turn indicator.

\begin{remark}
Observe that, in Definition \ref{def_pos_game}, specifying that only the first player to satisfy the winning condition is the winner ensures openness of the game for both players and is necessary to avoid ambiguity.

See for instance the Tic-tac-toe board position in figure \ref{tris}; both players seem to be winning, but assuming that Crosses is the first player of the game, then we can say that, for any history of legal moves leading to that board position, the resulting position is a win for Crosses.
\end{remark}

\begin{figure}[h]
\centering
\includegraphics[width=3.5cm]{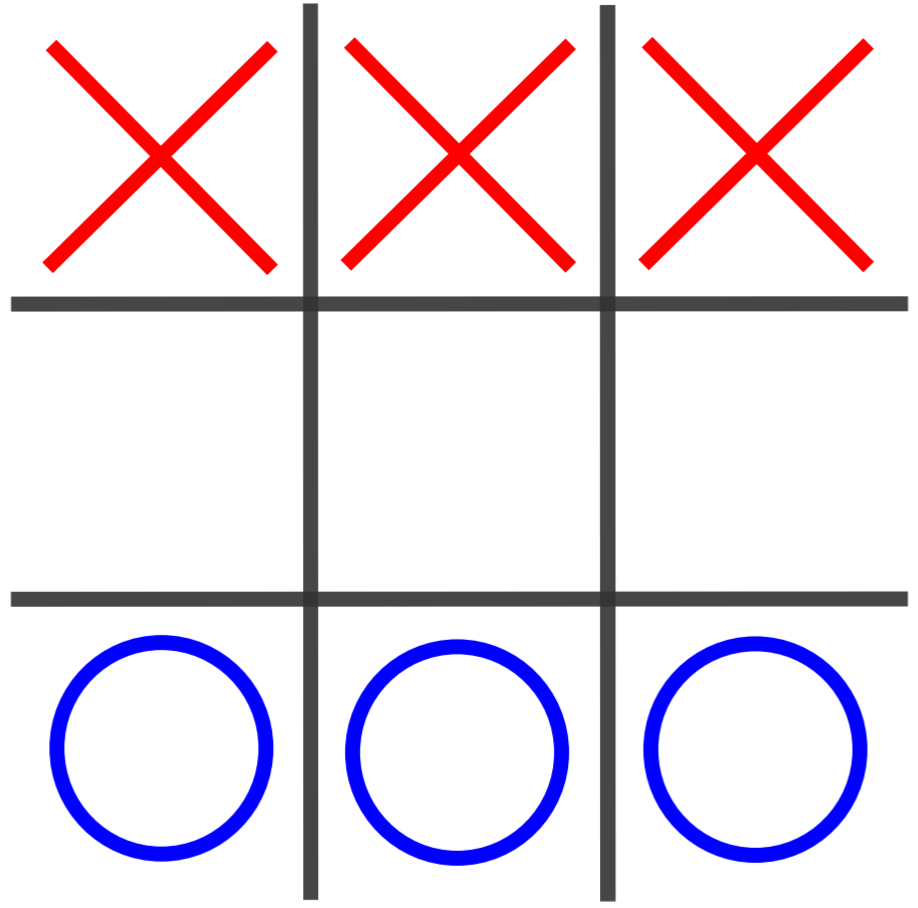}
\caption{An ambiguous Tic-tac-toe board position.}
\label{tris}
\end{figure}

\medskip

\begin{remark}[Strategy-stealing]
Observe that Nash's Strategy-stealing argument\footnote{See proof of Proposition \ref{strategy_stealing}.} can be applied to show that any positional game played on $(B,\mathcal{F})$ is either a draw or a first-player win.

In fact, if the second player has a winning strategy, then the first player can use it after having made an arbitrary initial move, leading to a contradiction; this holds because both players have exactly the same winning conditions, which are represented by $\mathcal{F}$, and because in all positional games it is always advantageous to make an extra move, i.e.~to choose one extra vertex.
\end{remark}

\medskip

Thus, the best outcome for the second player in a positional game is achieving a draw. For this reason, we do not lose generality when considering the class of games defined below.

\begin{definition}A \emph{Maker-Breaker game} played on a hypergraph $(B,\mathcal{F})$ is a game in which Maker wins by choosing all the vertices of some winning set in $\mathcal{F}$, while Breaker wins by choosing at least one vertex of each winning set in $\mathcal{F}$.\footnote{A different definition could let Breaker win whenever Maker does not win; such a no-tie variant 
is not equivalent to our definition in the setting of infinite games.
Not all vertices in $B$ may be marked over some plays and neither player may satisfy their original winning condition over such plays, so that we would consider such plays ties with our definition and wins for Breaker with the no-tie variant.}
\end{definition}

\begin{remark}
Clearly, Breaker's goal is precisely to prevent Maker from winning; it follows that a play of any Maker-Breaker game cannot allow both players to win.
\end{remark}

\label{duality}
Moreover, we can see that there is a duality between Maker and Breaker that preserves strategies, but not game openness.

Namely, take a Maker-Breaker game $\mathcal{G}$ open for Maker and played on some hypergraph $(B,\mathcal{F})$; observe that such a game is equivalent\footnote{Equivalent from the point of view of the closed player.} to the Maker-Breaker game $\mathcal{G}'$ not necessarily open for Breaker and played on the hypergraph $(B,\mathcal{F}')$, but with the players' roles reversed, where
\[\mathcal{F}'=\{F'\in \mathcal{P}(B):\lvert F'\cap F\rvert \ge 1 {\rm ~for~all~} F\in \mathcal{F}\}.\]
Observe that Breaker aims to mark at least one vertex in each of the winning sets in $\mathcal{F}$, so that $\mathcal{F}'$ does represent the correct winning condition in the game with the roles reversed.

\bigskip

\begin{remark}\label{mb_finite_hex}
Recalling the no-tie property of finite Hex,\footnote{See Corollary \ref{no_tie}.} we observe that a finite Hex player who prevents the opponent from winning is equivalent to the player himself winning. In other words, the best offensive moves in finite Hex are exactly the best defensive moves, and vice versa.

Hence, we see that finite Hex is a Maker-Breaker game played on the hypergraph $(H,\mathcal{C})$ such that the vertex set $H$ is in bijection with the empty finite Hex board and $\mathcal{C}$ is the collection of all the winning paths for Red, who corresponds to Maker.
\end{remark}

\medskip

What we mentioned in section \ref{results_finite_hex} for asymmetric Hex boards is that no winning component for the player who is joining the short edges can avoid containing two cells paired with respect to the opponent's strategy.

A weaker graph-theoretic property that follows from this pairing is that of \emph{proper 2-colourability} of $(H_{asym},\mathcal{C}_{disadv})$, where $H_{asym}$ represents the relevant empty asymmetric board, and $\mathcal{C}_{disadv}$ is the collection of winning paths for the disadvantaged player, who is losing; that is, $H_{asym}$ can be coloured with two colours, say black and white, such that no collection of tiles that represents a winning path in $\mathcal{C}_{disadv}$ is monochromatic.

Indeed, the cell pairing given by the mirroring strategy described in figure \ref{asymmetric_board} is a particular case of such a \emph{proper 2-colouring}, which is obtained by colouring with distinct colours the cells in each pair.

\medskip

We prove the following result in appendix \ref{app_2_col_mb}.

\begin{proposition}\label{mb_2_colouring}
Take some Maker-Breaker game $\mathcal{G}$ played on a hypergraph $(B,\mathcal{F})$ such that $\mathcal{G}$ is open for Breaker, and Maker is the first player to move.
If Breaker has a winning strategy, so that he can play to take at least one vertex in each of Maker's winning sets in $\mathcal{F}$, then $(B,\mathcal{F})$ is properly 2-colourable.
\end{proposition}

\subsection{Biased Infinite Hex may not be a draw}\label{biased_inf_hex}

Informally, our intuition of Infinite Hex is that, since the infinite board is ``too large", the advantage of placing the first stone is ``too diluted", and so an initial position with finitely many cells already marked by Red\footnote{Or even infinitely many, if placed ``badly" enough.} should not give a consistent advantage to Red, compared to the initially empty board, so that the game should result in a draw, as we saw in section \ref{hex_empty_draw}.

However, adaptations of the mirroring strategy $\mu$ do not seem to be sufficient even in the case of just one extra stone for Red.\footnote{Of course, games of Infinite Hex biased in favour of the second player present essentially the same problem.}

\medskip

It seems reasonable to claim that no pairing strategy can be formulated.
A possible proof could model Biased Infinite Hex as a Maker-Breaker game played on $(\widetilde{H}_{adv},\mathcal{C}_1)$, where $\widetilde{H}_{adv}$ is the infinite board without the cells given to Red as advantage, and $\mathcal{C}_1$ represents the winning condition for Red, the first player, and then assume for a contradiction that Blue has a drawing strategy, which would imply the proper 2-colourability of $(\widetilde{H},\mathcal{C}_1)$ by Proposition \ref{mb_2_colouring}; it is not clear how to get a contradiction from here.

\medskip

The non-openness of the games of Biased Infinite Hex makes such problem less tractable.

We then formulate the following.

\begin{conjecture}
A game of Infinite Hex with exactly one tile of the board assigned to Red at the initial position, and none to Blue, is a win for Red.
\end{conjecture}

As an apparently easier problem, we also propose the following.

\begin{conjecture}
A game of Infinite Hex with initial position as in figure \ref{hor_hex} is a win for Red.
\end{conjecture}

\begin{figure}[h]
\centering
\includegraphics[width=6cm]{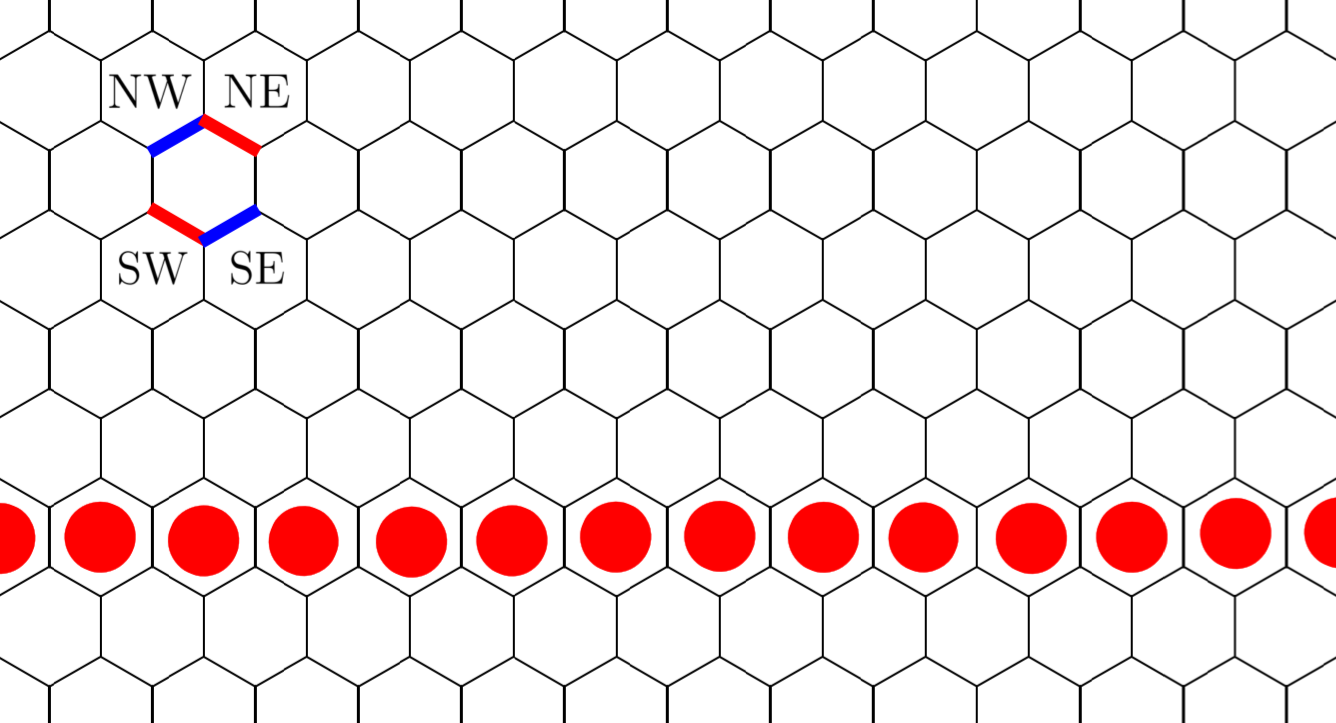}
\caption{Infinite East-West Red path.}
\label{hor_hex}
\end{figure}

It is not difficult to see that from such a position Red can construct arbitrarily long, but finite, ``prongs" that shoot out of the East-West infinite path; however, this is short of a definitive answer.



\subsection{Stone-placing games}\label{sec_stone_pl}

We will now prove that no infinite game values can be achieved in the following wider class of vertex-colouring games, that retain 
the property of extra moves being always advantageous for the player making them, which we first observed for finite Hex in Remark \ref{rmk_finite_hex}.\ref{extra_stone}.

\begin{definition}\label{def_stone_placing}
A \emph{stone-placing game} played on $(B,\mathcal{F},\mathcal{S})$, where both $(B,\mathcal{F})$ and $(B,\mathcal{S})$ are hypergraphs, is a game in which the two players alternatively choose previously unchosen elements of the common vertex set $B$; at the initial position, all the vertices of the board $B$ are unmarked.

The hyperedge sets $\mathcal{F},\mathcal{S}\subset\mathcal{P}(B) \setminus\emptyset$ represent the winning conditions of, respectively, the first and the second player; that is, the winner is the player who chooses all the vertices of some winning set specified by his respective winning condition first.

Thus, we impose that both $\mathcal{F}$ and $\mathcal{S}$ are closed under taking supersets, without loss of generality.
\end{definition}

\begin{remark}
Note that Definition \ref{def_stone_placing} does indeed generalise the other games defined previously in the current section \ref{sec_positional_games}.

Say that $\mathcal{G}$ is a stone-placing game played on $(B,\mathcal{F},\mathcal{S})$; if $\mathcal{F}=\mathcal{S}$, then $\mathcal{G}$ is a positional game.\footnote{This clarifies why positional games are also known as \emph{Maker-Maker games}.}

We can see that if $\mathcal{S}= \mathcal{F}'$, where
\[\mathcal{F}'=\{F'\in \mathcal{P}(B):\lvert F'\cap F\rvert\ge 1 {\rm ~for~all~} F\in \mathcal{F}\},\]
then $\mathcal{G}$ is a Maker-Breaker game in which Maker plays first; an analogous condition allows $\mathcal{G}$ to be a Maker-Breaker game with Breaker playing first.
\end{remark}

\begin{remark}\label{rmk_open_st_pl}
If $\mathcal{G}$ is a stone-placing game played on $(B,\mathcal{F},\mathcal{S})$ open for the first player, then
every win by the first player is essentially finite, and so for any $f\in\mathcal{F}$ there is some finite $f'\in\mathcal{F}$ such that $f'\subset f$.
Thus, we can say that $\mathcal{F}$ has a basis of finite sets; that is, there is $\{f_i\}_{i\in \mathcal{I}} \subset\mathcal{F}$ such that $f_i$ is finite for each $i\in \mathcal{I}$, which is not necessarily finite, and, for any $f\in \mathcal{F}$, there is $f_i \subset f$ for some $i\in \mathcal{I}$.

To express this more explicitly, say that a play $s$ is a win for the first player because he manages to mark all the vertices in some infinite winning set $f\in \mathcal{F}$, then there is a finite $f_i\subset f$, for some $i\in\mathcal{I}$, which is completely marked by the first player over $s$; since each vertex is marked at some finite stage of the play,\footnote{This is a consequence of plays having order type at most $\omega$.} then the whole of $f_i$ is marked by the first player after finitely many turns and $s$ is an essentially finite play, as required by the openness of $\mathcal{G}$.

Hence, if $\mathcal{G}$ is open for the first player, then we can regard the elements of the basis for $\mathcal{F}$ as the ``meaningful" winning sets, and we can thus assume that there are no ``redundant" infinite elements in $\mathcal{F}$, without essentially changing $\mathcal{G}$.
In particular, if all the elements of $\mathcal{F}$ are finite, then $\mathcal{G}$ is open for the first player.

If $\mathcal{G}$ is open for the second player, then $\mathcal{S}$ satisfies analogous conditions.

\medskip

Analogously, we see that if a Maker-Breaker game $\mathcal{G}$ played on $(B,\mathcal{F})$ is open for Maker, then $\mathcal{F}$ has a basis of finite sets.

Dually, if $\mathcal{G}$ is open for Breaker, then $\mathcal{F}'$ has a basis of finite sets, where $\mathcal{F}'$ is defined as above.\footnote{Recall the duality between Maker and Breaker discussed at page \pageref{duality}.}

In particular, if both $\mathcal{F}$ and its elements are finite, then $\mathcal{G}$ is open for both players.
\end{remark}

\medskip

\begin{remark}
We can generalise the mirroring strategy described in section \ref{hex_empty_draw} for Normal Infinite Hex, and at page \pageref{diss_asym_board} for finite asymmetric Hex boards, to give a condition that enables the second player\footnote{Note that an analogous result holds for the first player, who additionally makes an arbitrary initial move, as usual. We state the result for the second player because she is a priori at a disadvantage.} to force a draw.

We propose another way to see that argument; the pairing $\theta$ for the the mirroring strategy $\mu$ determines an involution\footnote{Given a pairing of the tiles in the Infinite Hex board, we need the Axiom of Choice to define such an involution.} of the infinite board $\theta:\widetilde{H} \rightarrow\widetilde{H}$;
note that $\theta$ has no fixed point, as each cell is paired with a different cell.

Take a winning path for Red and take its image under $\theta$; this will not necessarily be a winning path for Blue, exactly because the original tile pairing $\theta$ does not determine a rigid movement of the board; however, we have shown in section \ref{hex_empty_draw} that it will intersect Red's path.

Observe that the intersection consists only of paired cells; hence, if the intersection is finite, then it has even size.
\end{remark}

\begin{proposition}[Mirroring]\label{stone_pl_mirroring}
Let $\mathcal{G}$ be a stone-placing game played on $(B,\mathcal{F},\mathcal{S})$. Suppose that there is a fixed-point free involution $\theta$ of the board, that is a bijective function $\theta:B\rightarrow B$ such that $\theta^2(b)=b$ and $\theta(b)\neq b$ for all $b\in B$, and say that, for each $f\in\mathcal{F}$, we have $f\cap \theta[f] \neq \emptyset$.

Then, the second player in $\mathcal{G}$ can force a draw.
\end{proposition}

\begin{proof}
Let the second player play as follows; at each turn, when the first player marks some vertex $v\in B$, then the second player chooses $\theta(v)\neq v$ immediately after. Note that if $\theta(v)$ was already chosen by the first player, then $\theta(\theta(v))=v$ would have been previously taken by the second player, leading to a contradiction; thus, this strategy is well-defined.

Say that $v\in f\cap \theta[f]$. Since $v\in f$, then $\theta(v)\in \theta[f]$; moreover, since $v\in \theta[f]$, then $\theta(v)\in \theta^2[f] =f$.
Thus, $v$ and $\theta(v)$ are distinct elements of $f\cap \theta[f]$; we deduce that $f\cap \theta[f]$, if finite, has even size.

Following this strategy, the second player prevents the first player from choosing all the vertices in some hyperedge $f\in \mathcal{F}$; since $\lvert f\cap \theta[f]\rvert \ge 2$, then the second player will always be able to counteract the moves of the first player, so that the first player will not be able to mark all the vertices in $f\cap \theta[f]$.

Therefore, the second player can force a draw.
\end{proof}

\begin{remark}
Observe that a board involution $\theta$ as in Proposition \ref{stone_pl_mirroring} specifies a vertex pairing that can be used to describe a proper 2-colouring of $(B,\mathcal{F})$, as we had seen in Proposition \ref{mb_2_colouring}.
\end{remark}

\begin{remark}
Note that the Strategy-Stealing argument\footnote{See Proposition \ref{strategy_stealing}.} does not generally apply to this class of games, in fact, there is not always an automorphism of the vertices in $B$ such that each winning set in $\mathcal{S}$ is sent to some winning set in $\mathcal{F}$.

However, we can generalise the symmetry condition of the finite Hex board highlighted in Remark \ref{rmk_finite_hex}.\ref{symmetry_board} to construct the Strategy-stealing argument for stone-placing games which are not open; we provide an argument in appendix \ref{app_str_steal_st_pl}.
\end{remark}

\subsection{Impossibility of infinite game values}

\subsubsection{Finite game values in Infinite Hex}\label{finite_hex_values}

Firstly, we show that all finite ordinals can be realised as the game values of some Infinite Hex positions through the use of \emph{virtual links}.\footnote{These were first called \emph{angle-positions} by \cite{hein42}, and then received this name by Claude Berge in 1977, according to \cite[\textsection 7]{hayward06}; they are now part of the standard terminology used by Hex players \cite{hexwiki}.}

The position in figure \ref{fig_bridge_0} is a \emph{bridge}, as defined by \cite[\textsection 12.3]{browne00}, which ensures that Red can join the two marked cells by choosing either of the two cells separating them. Blue can threaten this construction by occupying one of these two cells, and Red can simply respond by picking the remaining one, thus ``realising" the virtual connection.

\begin{figure}[h]
\centering
\begin{subfigure}{.32\textwidth}
  \centering
  \includegraphics[width=.9\linewidth]{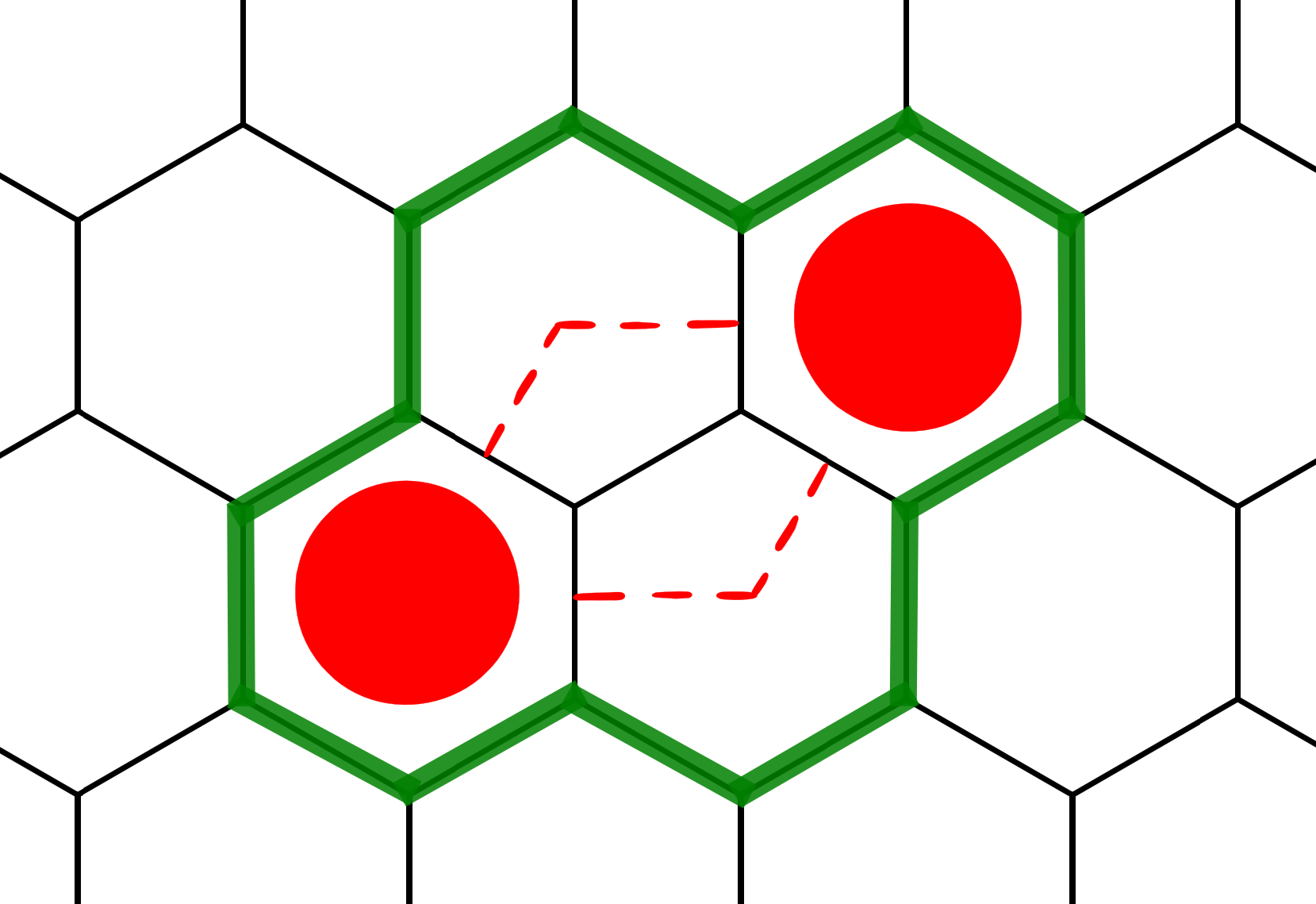}
  \caption{Red bridge}
  \label{fig_bridge_0}
\end{subfigure}
\begin{subfigure}{.32\textwidth}
  \centering
  \includegraphics[width=.9\linewidth]{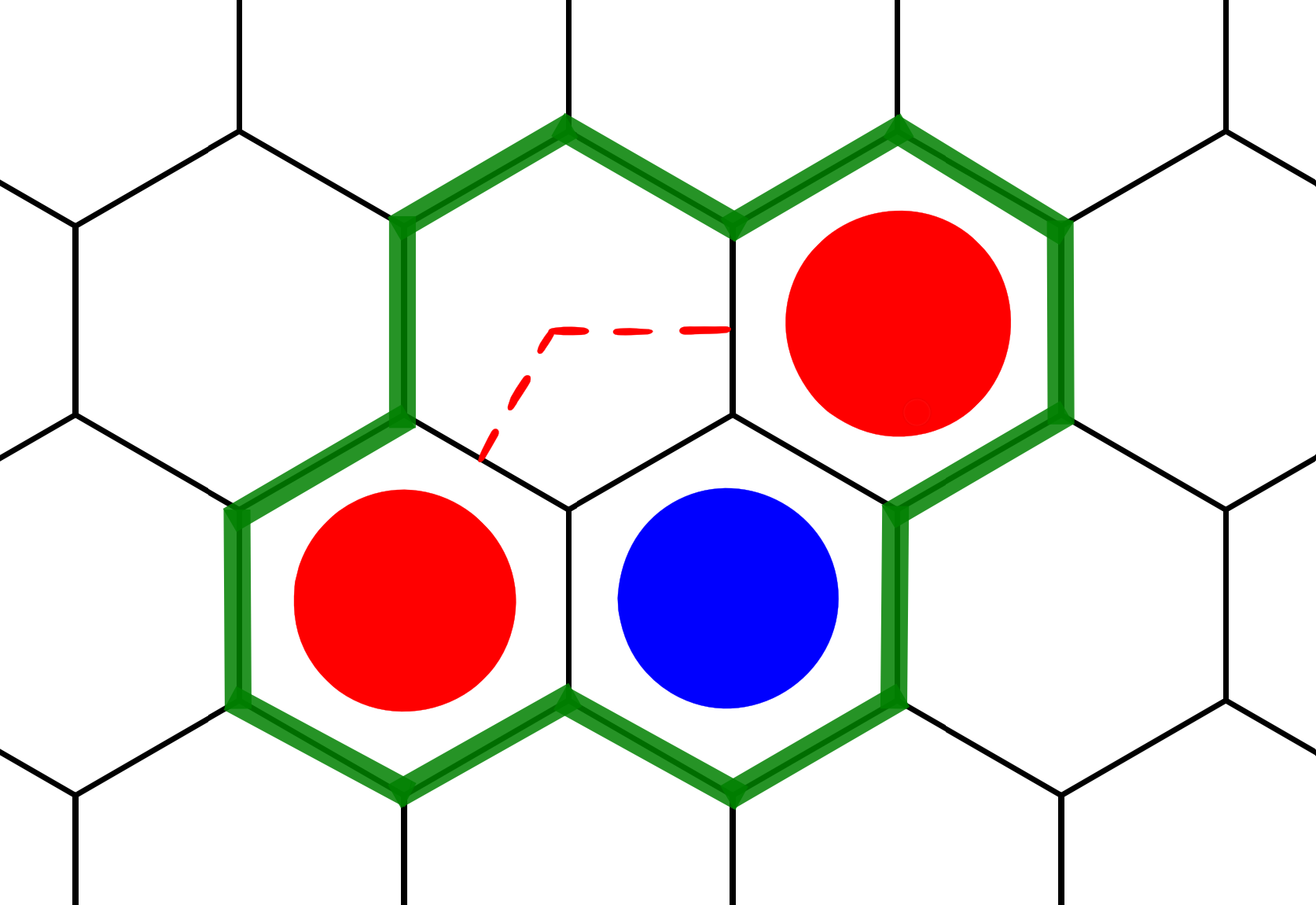}
  \caption{Threat by Blue}
  \label{fig_bridge_1}
\end{subfigure}
\begin{subfigure}{.32\textwidth}
  \centering
  \includegraphics[width=.9\linewidth]{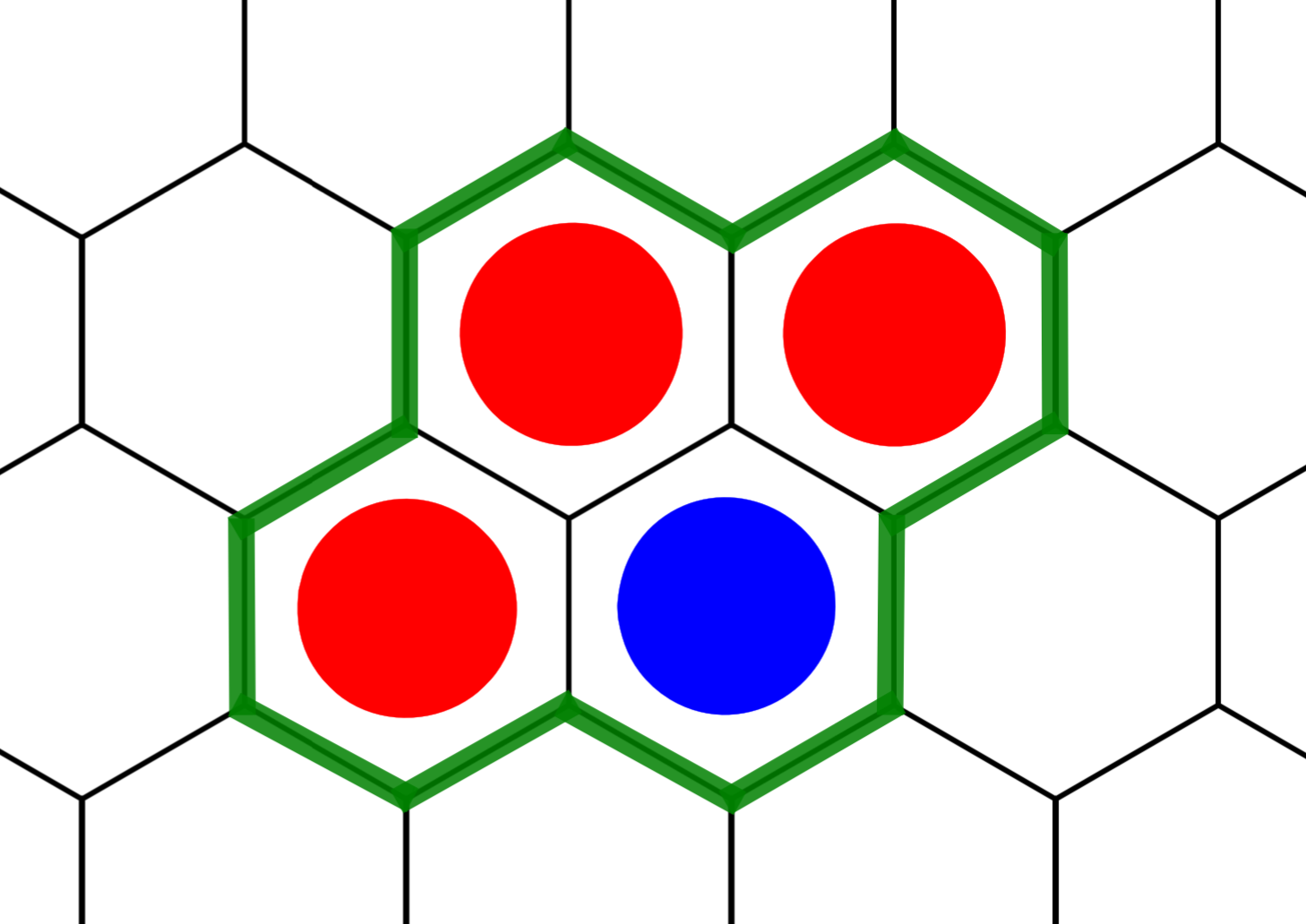}
  \caption{Response by Red}
  \label{fig_bridge_2}
\end{subfigure}
\caption{The realisation of a virtual link by Red.}
\label{fig_bridge}
\end{figure}


Joining two half-infinite lines by finitely many interlinked bridges allows to construct positions of any finite game value.

We present in figure \ref{fig_interlinked_bridges} a position with value 5 for Red, two half-infinite linear components joined by 5 bridges, in which the virtual connections are highlighted.

\begin{figure}[h]
\centering
\includegraphics[width=8cm]{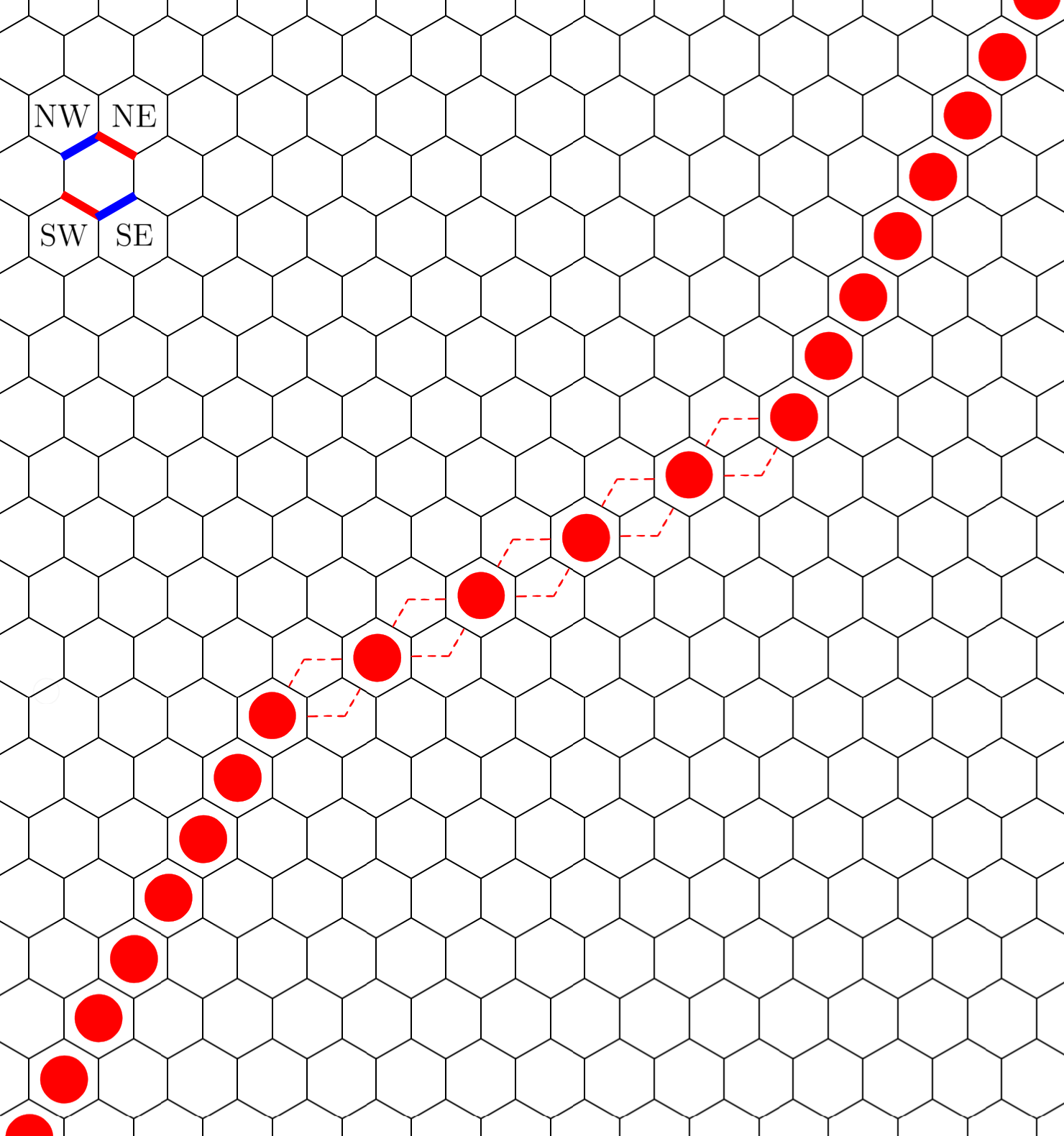}
\caption{Infinite Hex position with value 5 for Red.}
\label{fig_interlinked_bridges}
\end{figure}

Note that this position has infinitely many pieces already placed.
More in general, since the winning player places only finitely many stones when starting from a position with defined game value, then it follows from the definition of winning condition that no position with defined game value can have only finitely many stones already placed.

\subsubsection{Infinite game values}

It can be reasonable to expect that some well-constructed open games of Infinite Hex can achieve an infinite game value; however, we will show that this is not possible by proving that Hex, as all stone-placing games, is an \emph{essentially local} game.

\begin{definition}
A \emph{dead region} for some player in a stone-placing game, which is played on $(B,\mathcal{F},\mathcal{S})$ and has defined game value $\alpha$ for the player's opponent,\footnote{So that the game is clearly open for the player's opponent.} is some subset of the board $R\subset B$ such that, if all the vertices in $R$ are assigned to the (losing) player, then the game with the resulting initial position, which has the same turn indicator, still has value $\alpha$ for the (winning) player's opponent.\footnote{This definition is inspired by the concept of dead cells, which are usually defined in a more restrictive sense as cells that cannot influence the outcome of the game from a given position; see for instance \cite[3.1]{hayward06}.}
\end{definition}

In other words, modifying the initial position of the game by assigning the cells of a dead region as advantage to the losing player does not substantially alter the game; even though the strategy of the winning player may change, the game value for such player does not vary.

\medskip

\begin{definition}
A stone-placing game, which has defined game value for some player, is \emph{essentially local} for that player if there is a cofinite dead region of the board for that player's opponent.
\end{definition}

\begin{remark}
An analogous definition can be made for essentially local positional games.
Also, it is obvious that all finite games are essentially local.
\end{remark}

For instance, see that all the empty tiles in figure \ref{fig_interlinked_bridges}, except for the 10 empty tiles which allow Red to realise the virtual links, belong to a dead region for Blue, hence that is the initial position of an essentially local game.

\bigskip

We aim to prove the following.

\begin{theorem}\label{no_inf_hex}
Any game of Infinite Hex with defined value for some player is essentially local for that player; thus, it has finite game value.
\end{theorem}

Observe that a game $\mathcal{H}$ of Infinite Hex can be represented as a stone-placing game.

Suppose that Red is the first player to move in $\mathcal{H}$ and say that $(\widetilde{H},\widetilde{\mathcal{R}})$ and $(\widetilde{H},\widetilde{\mathcal{B}})$ are hypergraphs in which the vertex set $\widetilde{H}$ is in bijection with the board cells unchosen at the initial position of $\mathcal{H}$, and the hyperedges in $\widetilde{\mathcal{R}},\widetilde{\mathcal{B}} \subset\mathcal{P}(\widetilde{H})$ represent all the possible winning paths for, respectively, Red and Blue.
Note that the vertices of all hyperedges in $\widetilde{\mathcal{R}}$ and $\widetilde{\mathcal{B}}$ are unmarked at the initial position of $\mathcal{H}$.

We can see that $\mathcal{H}$ can be interpreted as a stone-placing game played on $(\widetilde{H},\widetilde{\mathcal{R}},\widetilde{\mathcal{B}})$, where the first player is Red and the second player is Blue.

\bigskip

It is now clear that Theorem \ref{no_inf_hex} is a consequence of the following more general result.

\begin{theorem}\label{st_pl_thm_ess_loc}
Any stone-placing game with defined value for some player is essentially local for that player; thus, it has finite game value.
\end{theorem}

\medskip

In order to highlight the symmetry between the first and second player, in the rest of this section stone-placing games will be played on $(B,\mathcal{O},\mathcal{C})$ by the players Open and Closed, who have winning conditions $\mathcal{O}$ and $\mathcal{C}$, respectively.

\begin{lemma}\label{mb_lemma}
Given any stone-placing game $\mathcal{G}$ played on $(B,\mathcal{O},\mathcal{C})$ which has finite game value $n$, we have that $\mathcal{G}$ is essentially local for Open.

That is, there is a finite subset of the board $D\subset B$ such that Open has a strategy that allows him to win in at most $n$ turns by choosing vertices only in $D$, and $B\setminus D$ is a dead region for Closed.
\end{lemma}
\begin{proof}
We prove this by induction on $n$.

\medskip

If $n=0$, then there is nothing to prove, as Open has already won at  the initial position and does not need to make any move. Trivially, $D=\emptyset$.

\medskip

Suppose that $n=1$; we consider this case for clarity of argument.
Let $\sigma$ be a winning strategy for Open realising this game value.

If Open is the next to move, then he can simply win by choosing a single vertex $v\in B$, as prescribed by $\sigma$; in this case, $D=\{v\}$ is sufficient.

If instead Closed is the next to move, then suppose that he chooses some vertex $w$, without winning.
Then Open can again make a winning move by choosing some $v'\neq w$. If instead Closed decides to initially choose $v'$, then Open can still win by choosing the vertex $v''\neq v'$ prescribed by $\sigma$ in that position.
In either case, $D=\{v',v'',w\}$ is as needed.

\medskip

Suppose now that the Lemma is true for $n=N$. Let $\mathcal{G}$ have game value $N+1$ for Open, and let $\tau$ be a value-reducing strategy that realises this game value.

\medskip

If Open is the next to move from the initial position, then, by definition of game value,\footnote{Definition \ref{def_game_value}.} he can choose a vertex $v\in B$ so to reach a position with game value $N$.
Then, by inductive hypothesis, there is a finite subset $D'\subset B$ such that Open can win in at most $N$ moves by choosing vertices only from $D'$. In this case, $D=D'\cup \{v\}$ is as needed.

\medskip

If instead Closed is the next to move from the initial position, then suppose that he chooses some $w_0\in B$.
Following $\tau$, Open can mark some other $v_0\in B$ to reach a position with game value at most $N$; by inductive hypothesis, there is a finite sub-board $D_0'\subset B$ such that Open wins by choosing vertices only from it. Let $D_0=D_0'\cup\{v_0,w_0\}$.


Consider again the initial position of $\mathcal{G}$; suppose that there are $k$ unchosen vertices\footnote{Note that both $v_0$ and $w_0$ are still unchosen at the initial position.} in $D_0'\cup\{v_0\}$ at the start, say $\{w_i\}_{i=1}^k$.
Now, take some $1\le i\le k$ and suppose that Closed decides to initially choose $w_i$, to which Open replies by choosing some $v_i$ determined by $\tau$, reaching a position of game value at most $N$; by inductive hypothesis, there is a finite sub-board $D_i'\subset B$ such that Open wins by playing only on its vertices.
Note that, in this case, the line of play of Open is only influenced by the moves made on the sub-board $D_i=D_i'\cup\{v_i,w_i\}$.

\medskip

We claim that $D=\bigcup_{i=0}^k D_i$ is a finite sub-board that satisfies the condition of the Lemma at the initial position of $\mathcal{G}$.

Suppose that Closed initially chooses an available vertex in $D_0$, that is, he chooses $w_i$ for some $0\le i\le k$.
Open can reply by choosing the corresponding vertex $v_i$, and subsequently win in at most $N$ more moves by playing on $D_i$; Open thus ends the game having chosen vertices only in $D_i\subset D$, by construction.

Suppose now that Closed makes an initial move outside of $D_0$, choosing some vertex $w'\not\in D_0$ (but possibly in $D$). Open can pretend to see $w'$ as still unchosen and he can play as if Closed had made the \emph{phantom move} of choosing $w_0$.
Hence Open can win by playing only on $D_0\subset D$, as before.
Observe that in this case the initial move by Closed does not obstruct Open's play since Open wins by marking all the vertices of some $o\in\mathcal{O}$ such that $o\subset D_0$, and that certainly does not depend on $w'\not\in D_0$.

\medskip

This concludes the proof of the Lemma.
\end{proof}

\medskip

\begin{lemma}\label{no_inf_mb}
No stone-placing game open for some player has infinite game value for that player.
\end{lemma}

\begin{proof}
Let $\mathcal{G}$ be a stone-placing game played on $(B,\mathcal{O},\mathcal{C})$, which is open for Open.
Suppose for a contradiction that $\mathcal{G}$ has infinite game value for Open; by Proposition \ref{init_segment_ord}, it is enough to assume that $\mathcal{G}$ has game value $\omega$ for Open.

By definition of game value, Closed is the first to move in $\mathcal{G}$ and can only reach positions with finite game value from the initial position.

So, suppose that Closed initially chooses some $w\in B$, reaching a position with finite value $N$ for Open. Now, by Lemma \ref{mb_lemma}, there is a finite sub-board $D'\subset B$ such that Open can win in at most $N$ turns by marking vertices only in $D'$.
Let $\sigma$ be a value-reducing strategy for Open that realises this in $\mathcal{G}$; note that the moves prescribed by $\sigma$ for Open only depend on the moves that take place on $D=D'\cup \{w\}$.

\medskip

Suppose now that, from the initial position, Closed makes a different move $w'\not\in D$.

Then, Open can ignore such move by Closed, and imagine instead that Closed had made the phantom move of choosing $w$, and then play according to $\sigma$ from such position imagined by Open.

Similarly to the situation in the proof of Lemma \ref{mb_lemma}, Open can safely ignore Closed's move in $w'$ because $\sigma$ allows him to win by playing only on $D$, and so by marking all the vertices of some $o\in\mathcal{O}$ such that $o\subset D$; it is clear that Closed's moves outside of $D$ cannot obstruct this goal.

It is important to note that, by playing outside of $D$, Closed cannot aim to win, as he can only reach positions with defined game value for Open.

\medskip

We have shown that, for an arbitrary initial move by Closed, Open can force a win in at most $N$ turns; thus, the position achieved by the initial move of Closed has game value bounded by $N$. But the initial position had game value $\omega$; this is a contradiction.
\end{proof}

\medskip

\begin{proof}[Proof of Theorem \ref{st_pl_thm_ess_loc}]
A stone-placing game with value defined for some player can only have finite game value, by Lemma \ref{no_inf_mb}; such game is thus essentially local for that player by Lemma \ref{mb_lemma}.
\end{proof}

\begin{remark}
Observe that the arguments above can be adapted to generalise the results to $n$-player stone-placing games, once the definition of game value is extended to $n$-player games.
\end{remark}

\bigskip

By definition of omega one, we have the following immediate consequence of Theorem \ref{st_pl_thm_ess_loc}.

\begin{theorem}\label{stone_pl_omega_one_thm}
The omega one of any class of stone-placing games open for some player is at most $\omega$.
\end{theorem}

\bigskip

We see that also Theorem \ref{no_inf_hex} is now proved; following our previous discussion of Infinite Hex positions with finite game value, we can conclude by stating the result which answers the main question that motivated the present research.

\begin{theorem}
The omega one of the class of open games of Infinite Hex is the smallest infinite ordinal, that is,
\[\omega_1^{\rm Hex}=\omega.\]
\end{theorem}

\section{Other generalisations of Hex}
It is not easy to intuitively capture the complexity of finite board games, so we now present explicitly some measurement of complexity for the finite games we mention in this work.

We mention the following estimates\footnote{From \cite{herik02}, and \cite{uiterwijk09} for TwixT.} of the state-space and game-tree complexities of the different board games we are considering.

\begin{center}
\begin{tabular}{lcc}
\toprule
Game	&	State-space complexity	&	Game-tree complexity		\\
\midrule
Checkers $(8\times 8)$	&	$10^{21}$		&	$10^{31}$		\\[2pt]
Draughts $(10\times 10)$	&	$10^{30}$		&	$10^{54}$		\\[2pt]
Chess	&	$10^{46}$		&	$~10^{123}$		\\[2pt]
Hex $(11\times 11)$	&	$10^{57}$		&	$10^{98}$		\\[2pt]
TwixT $(24\times 24)$	&	$>>10^{140}~~~~$		&	$~10^{159}$		\\[2pt]
\bottomrule
\end{tabular}
\end{center}

\subsubsection{The Shannon game}

The Shannon game is played on a graph $\Gamma$ with two distinguished vertices.

Rephrasing the definition of \cite[\textsection 2]{hayward06}, we can define this game as a Maker-Breaker game played on $(B,\mathcal{F})$, such that $B$ is the set of non-distinguished vertices of $\Gamma$ and $\mathcal{F}$ is the collection of paths that join the distinguished vertices of $\Gamma$.

Usually, Maker is called Short, and Breaker is called Cut.

\medskip

It should come as no surprise that Hex can be represented as a Shannon game.


\subsubsection{TwixT}
TwixT is a connection game played on a $24\times 24$ square lattice. The two players place stones on the vertices of such lattice and join two of their stones with an edge when they are opposite vertices of a $2\times 1$ rectangle.

The peculiarity of this game is that edges cannot intersect, so that the order in which moves are made have a great importance in determining the outcome of a game.

Moreover, even though TwixT looks rather different from Hex, it is actually a generalization of it. Thanks to the constructions provided by \cite{bonnet16}, we can simulate any Hex game on an appropriately filled TwixT board.

Observe that TwixT can be described as a particular edge-colouring game on the Knight's graph, i.e.~the graph obtained by joining the squares of a chessboard separated by a Chess knight's move.

However, TwixT cannot be directly studied as a stone-placing game, and so we cannot apply our results to an infinite version of TwixT.


\chapter{Draughts}\label{chapter_draughts}

Also known as Checkers, Draughts is a board game that has a long tradition and many variants, all with the same core rules.

Some pieces, called \emph{pawns},\footnote{Pawns are usually called \emph{pieces}, which is a word that we will use to refer to both pawns and kings.} are initially placed on the white squares\footnote{We follow the common custom of displaying Draughts positions as played on the white squares of the board for clarity of presentation; we state the rules accordingly.} of a chessboard, or \emph{draughtboard}, usually of size $8\times 8$, which we orient as usual.
At the initial position, the pieces of Black are placed on the Southern side of the board and the pieces of White on the Northern side, as in figure \ref{standard_draughts}.

\begin{figure}[h]
\centering
\includegraphics[width=4cm]{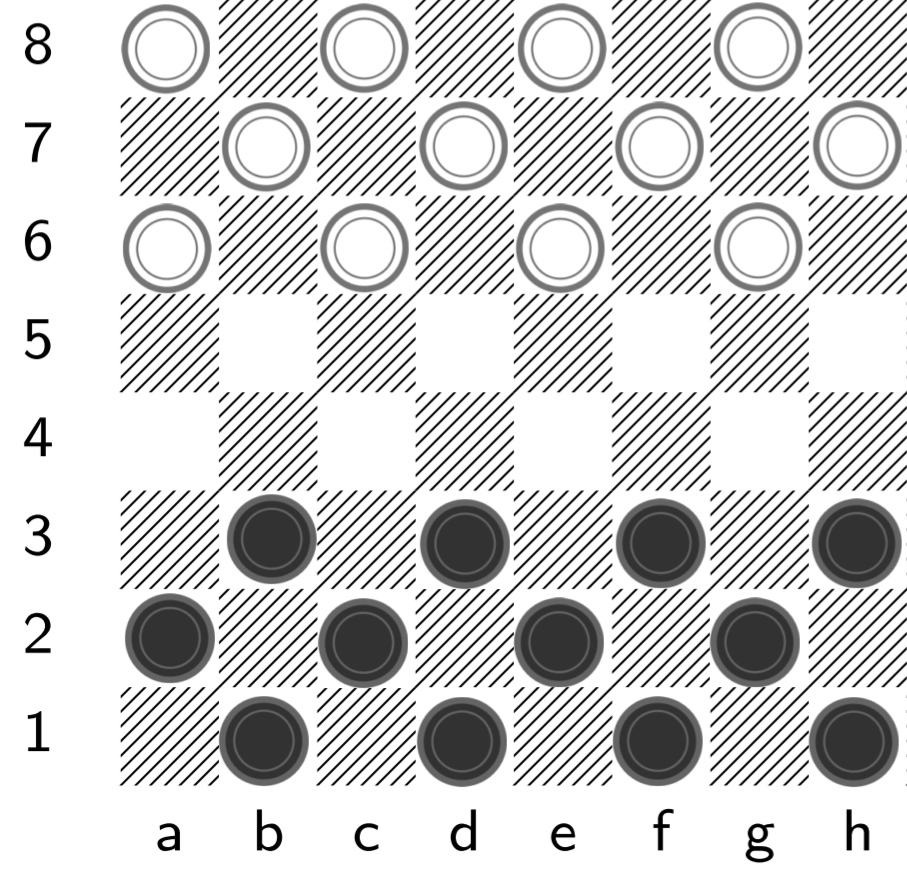}
\caption{Initial position of finite Draughts.}
\label{standard_draughts}
\end{figure}

\medskip

At each turn, a player can move one of his pieces to one of the empty squares diagonally adjacent to it.

If a player's piece is adjacent to an opponent's piece, which in turn is adjacent to an empty square aligned with such two pieces, then the player can make a \emph{simple jump} move
letting his piece ``jump" to the empty square, removing the ``jumped" opponent's piece from the board; see figure \ref{simple_jump}.

\begin{figure}[h]
\centering
\begin{subfigure}{.25\textwidth}
  \centering
  \includegraphics[width=.7\linewidth]{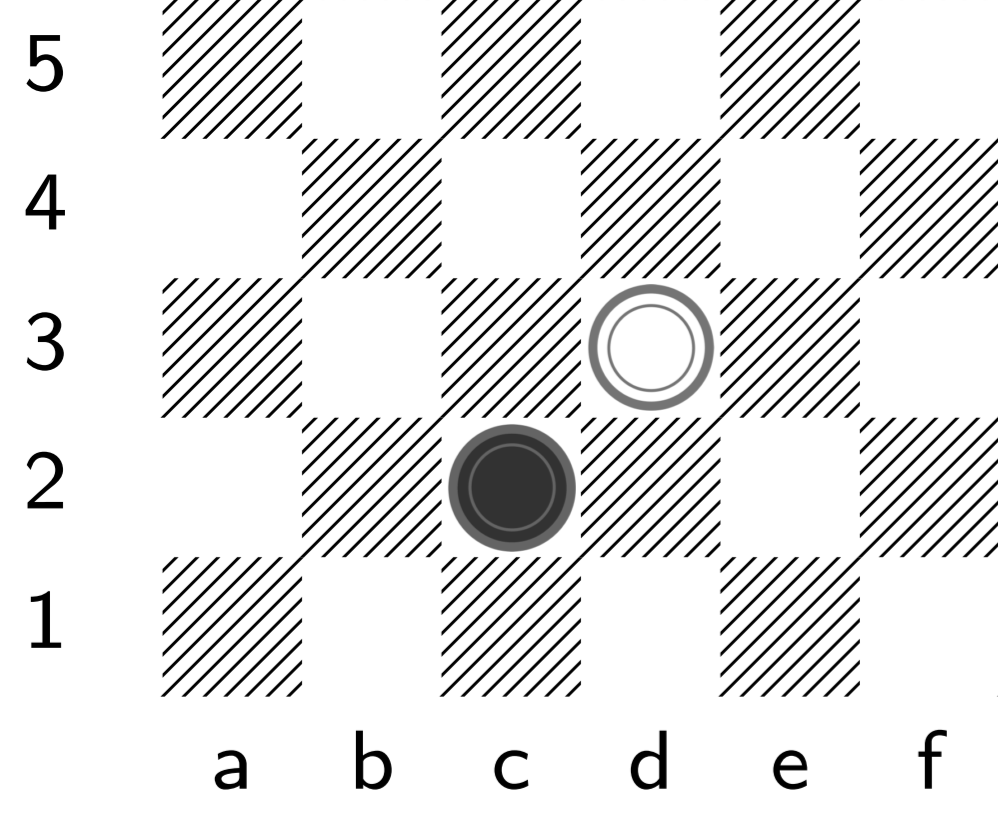}
  \caption{Before}
  \label{jump_before}
\end{subfigure}
\begin{subfigure}{.25\textwidth}
  \centering
  \includegraphics[width=.7\linewidth]{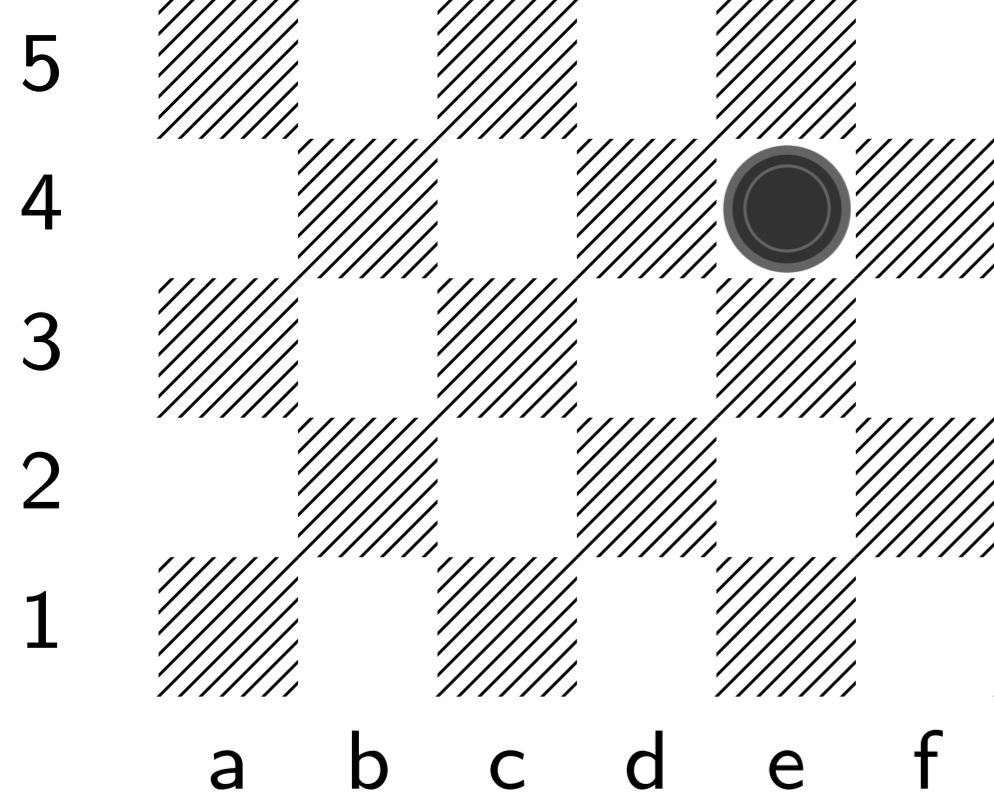}
  \caption{After}
  \label{jump_after}
\end{subfigure}
\caption{A simple jump by Black.}
\label{simple_jump}
\end{figure}

Moreover, if the player's piece is, immediately after this jump, in a similar position, then the player can extend his move by making another jump; the resulting move is called an \emph{iterated jump}.
It is important to note that an iteration of jumps constitutes a single move, rather than multiple moves.

We refer to these two last types of move as \emph{jump} moves.

\medskip

In particular, Black pawns are only able to make such moves going strictly towards North; when a Black pawn reaches the Northernmost row of the chessboard, it is promoted to \emph{king}, so that it can then make the moves described above in all directions.
An analogous rule applies to White pieces.

\medskip

The first player that is not able to make a move, either because he has no pieces left on the board or because he has no available legal moves, loses.

\medskip

It is required in Draughts competitions that a player who can make a jump move at some turn must do so at that turn.
Moreover, after a jump move, the piece just moved cannot be in a position from which another jump move could be possible.

In other words, if a player makes a piece jump, then he has to make it keep jumping, choosing which way to jump, until it reaches a position from which no more jumps are possible.

We can formalise this as follows.

\begin{game_rule}[Forced Jump]\label{forced_jump}
Each player must make jump moves, when they are legal.
\end{game_rule}

\begin{game_rule}[Forced Iteration]\label{forced_iteration}
No player can make a jump move which can be extended as an iterated jump move.
\end{game_rule}

Observe that rule \ref{forced_iteration} allows players to make only maximal jump moves;
in particular, it allows players to choose among distinct maximal jump moves.
See for instance the position in figure \ref{dr_mult_jump}, where Black can make a simple jump to the square labelled with a dagger, or make either of the iterated jumps to the starred squares; however, Black cannot make a simple jump to the square labelled with ``$\circledast$" by rule \ref{forced_iteration}.

\begin{figure}[h]
\centering
\includegraphics[width=4.5cm]{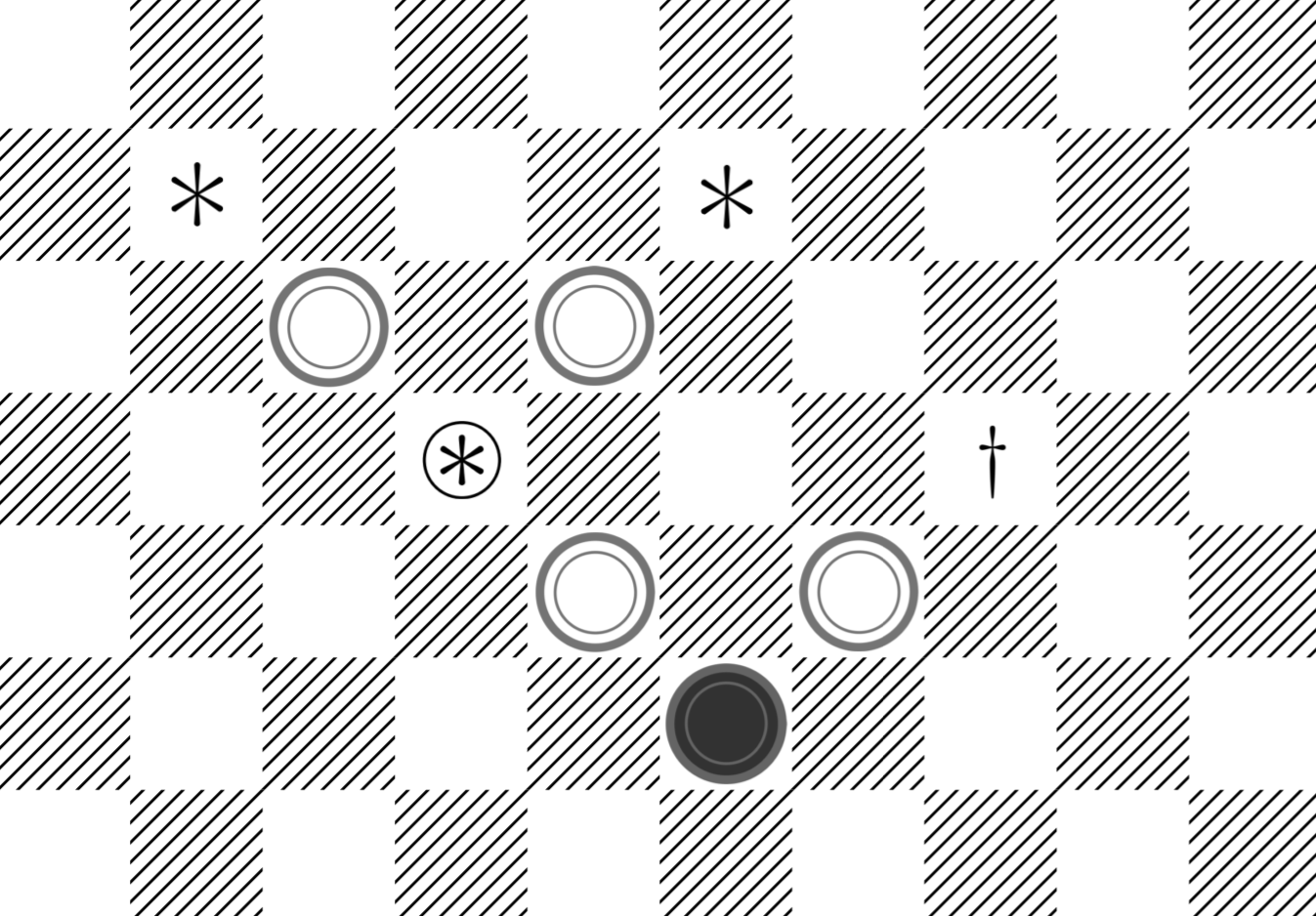}
\caption{Distinct maximal jumps for Black.}
\label{dr_mult_jump}
\end{figure}

\begin{remark}
Rule \ref{forced_iteration} is justified by the fact that if a player makes one of his pieces jump without completing an available jump iteration, then that piece will be threatened, and likely captured at the next turn,\footnote{Or certainly captured, in case that is the only jump move available to the opponent.} by an opponent's piece, as in the case of Black stopping on the square labelled ``$\circledast$" in figure \ref{dr_mult_jump}.\footnote{Moreover, it is usually felt as an unfair advantage the possibility of a king to end a jump at the ``back" of an opponent's pawn, by which the king cannot be threatened.}
\end{remark}

\section{Infinite Draughts}

Analogously to Infinite Chess,\footnote{See Example \ref{infinite_chess}.} Infinite Draughts is a game played on the infinite chessboard; in particular on its white squares, which are also in bijection with $\mathbb{Z}\times\mathbb{Z}$.

We can give a North-South orientation to the board so to maintain the difference between pawns and kings; however, all the pieces in the positions that follow are kings, unless otherwise specified.

Even though the rules of the game are simple and immediately generalise to the infinite case, we encounter a difficulty with the forced-iteration rule \ref{forced_iteration}; observe the position in figure \ref{dr_inf_jump}.

\begin{figure}[h]
\centering
\includegraphics[width=7cm]{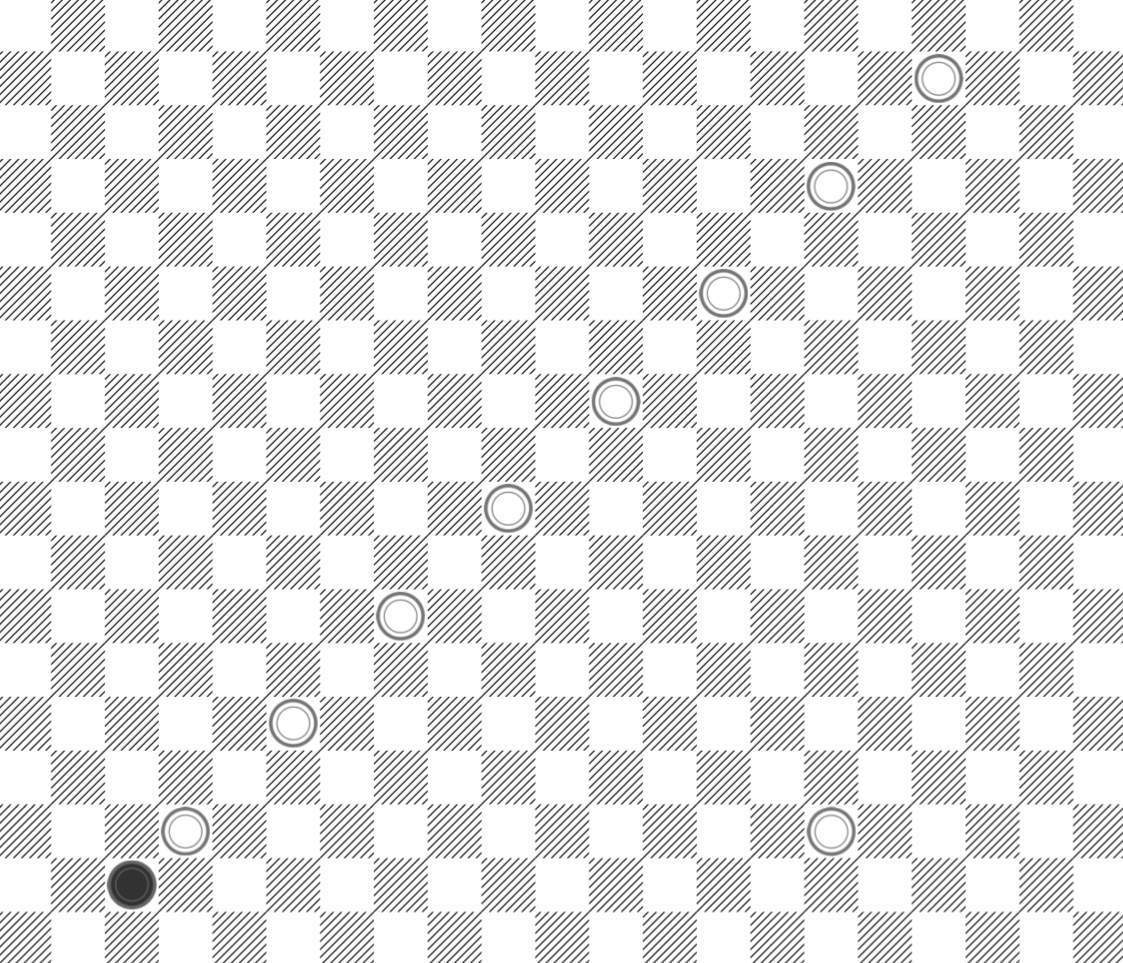}
\caption{Black piece at the bottom of an infinite White ladder, Black to move.}
\label{dr_inf_jump}
\end{figure}

\medskip

In fact, assuming only rule \ref{forced_jump}, if Black makes a jump move up the White infinite \emph{ladder} in figure \ref{dr_inf_jump} and captures only finitely many White pieces, then White wins on the next move, as Black would stop his only piece in a square already threatened by White.
Thus, if we applied strictly the forced-iteration rule \ref{forced_iteration}, we would say that Black must make an infinitely iterated jump, capturing infinitely many White pieces in a single move.

Note that, on the next turn, White can move his only surviving piece at the bottom right, but, following that, Black has no pieces left on the board and thus loses;
in fact, having made an infinitely iterated jump, the Black king is not on any of the squares of the infinite chessboard, thus it is not on the board anymore.\footnote{In other words, the Black king  effectively went to infinity.}

\medskip

It is clear that, for Black, making an infinitely iterated jump in the position of figure \ref{dr_inf_jump} is equivalent to admitting defeat;
as we will consider positions in which Black also loses if he decides to make an infinitely iterated jump and we will want Black to delay his defeat as much as possible, we generalise the forced-iteration rule of Draughts as follows, in order to simplify the exposition in the next section \ref{trees_in_draughts}.

\begin{game_rule}[Finite Iteration]\label{fin_iteration}
If a player can make an iterated jump move, then he can only make a finite iteration of jumps.
\end{game_rule}

\begin{remark}
Observe that rule \ref{fin_iteration} is still coherent with rules \ref{forced_jump} and \ref{forced_iteration} also when considering the position of figure \ref{dr_inf_jump}.
In that position, making a jump move up the White ladder would not be a legal move for Black, who can thus move his only king in one of the other directions.
\end{remark}

\subsection{Embedding trees in the draughtboard}\label{trees_in_draughts}

In this section, we will consider games of Infinite Draughts assuming rules \ref{forced_jump} and \ref{fin_iteration}, so that players must make jump moves when they can, and they can choose to iterate such jumps for at most finitely many times.

This will allow for simpler constructions in the rest of this section; we will consider different sets of rules in section \ref{draughts_further_results}.

\medskip

Observe that, similarly to Infinite Hex, games of Infinite Draughts are not in general open, as it is possible to construct positions from which a player can win over an infinite play, but not at any finite stage of such play.

We therefore focus on the class of games of Infinite Draughts which are open for some designated player; we will prove the following.

\begin{theorem}\label{dr_omega_one}
Every countable ordinal arises as the value of an open game of Infinite Draughts. Hence, 
\[\omega_1^{\rm Draughts} =\omega_1.\]
\end{theorem}

\medskip

We will prove this Theorem by constructing an argument closely related to the proof that the omega one of 3-dimensional Infinite Chess is $\omega_1$, as shown by \cite[\textsection 4]{hamkins14}.

We have already met a class of games with uncountable omega one; recall the class of climbing-through-$T$ games and Proposition \ref{omega_one_climb}.

\medskip

We will now describe a way to construct, given any tree $T_0$, a game of Infinite Draughts $\mathcal{D}_{T_0}$ equivalent to the game of climbing-through-$T_0$, so that a closed-player strategy for one game determines a closed-player strategy for the other game.

First, we ensure that the infinite draughtboard does provide enough space for our construction.

Keeping in mind the bijection between the white squares of the infinite chessboard and $\mathbb{Z}\times \mathbb{Z}$ as specified by the system of reference presented in figure \ref{axes}, we prove the following.

\begin{figure}[h]
\centering
\includegraphics[width=5cm]{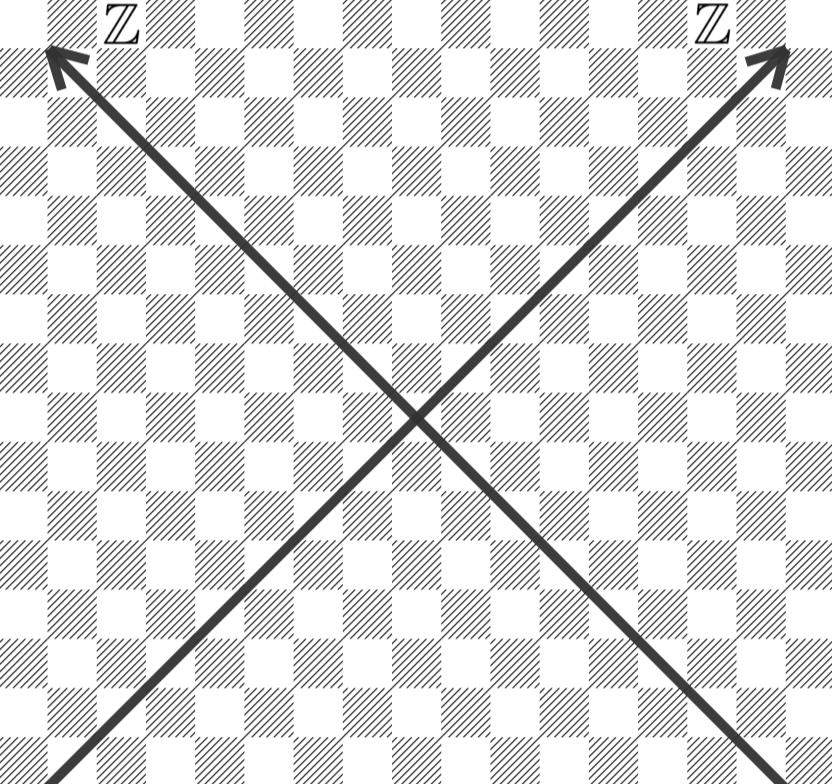}
\caption{Orthogonal axes on the Infinite draughtboard.}
\label{axes}
\end{figure}

\begin{proposition}\label{full_binary}
The full binary tree $2^{<\omega}$ can be realised in $\mathbb{Z}^+\times \mathbb{Z}^+$.
\end{proposition}
\begin{proof}

Recall that the full binary tree is the tree whose nodes all have exactly two children.

We now describe an algorithm to construct recursively the full binary tree.

For $n=0$, we start the recursion with a single vertex.

For $n=1$, we stack two copies of the same vertex, in red in figure \ref{recursion2}, at a sufficient distance in order to join them to the original vertex.

\begin{figure}[H]
\centering
\includegraphics[width=5cm]{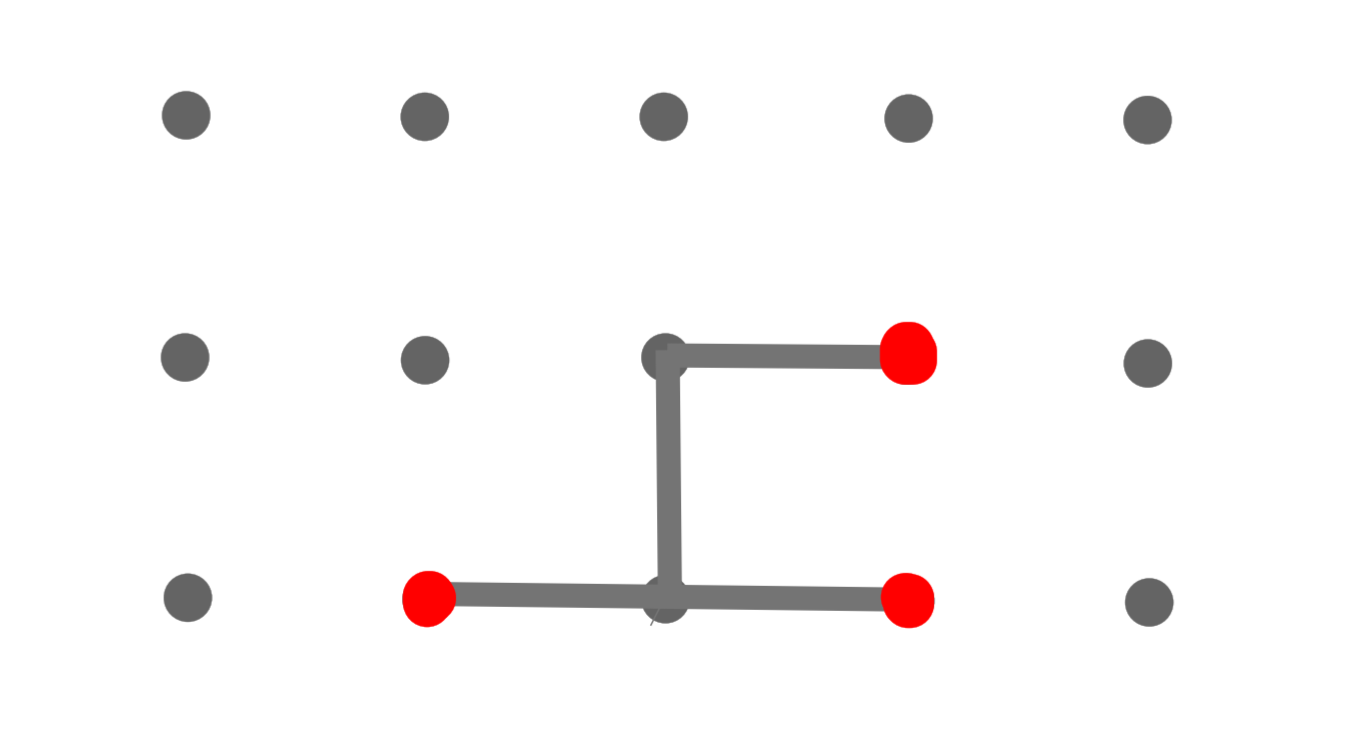}
\caption{First recursive step.}
\label{recursion2}
\end{figure}

We then proceed similarly, given a subtree $T_n$ of $2^{<\omega}$, we construct $T_{n+1}$ by attaching a copy of $T_n$ to each of the leaves of the original $T_n$ we stack two copies of it to the right of the previous subtree at a sufficient distance from the original subtree and join each root of the new copies of $T_n$ to some leaf of the original $T_n$; this is always possible by stacking the new copies of $T_n$ far enough in the positive $x$-direction; this is clear in the second recursive step pictured in figure \ref{recursion3}.

We can see that the trees $T_n$ so constructed are binary and have $2^{2^{n}}$ leaves.

\begin{figure}[H]
\centering
\includegraphics[width=5cm]{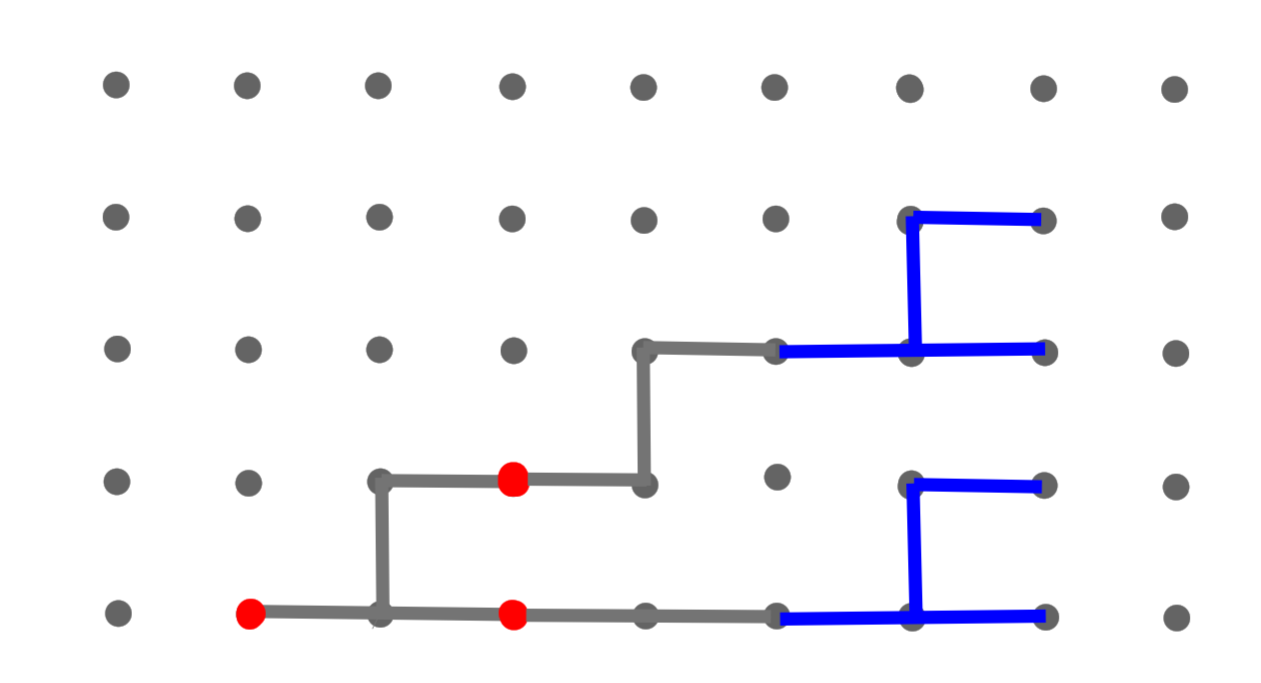}
\caption{Second recursive step.}
\label{recursion3}
\end{figure}
\end{proof}

\begin{remark}
We have just shown that $2^{<\omega}$ can be realised in a quadrant of the integer lattice $\mathbb{Z}\times \mathbb{Z}$; however, at each step of the previous recursion, the new finite tree copies can be placed arbitrarily far in the positive direction of the $x$-axis, so that the whole tree $2^{<\omega}$ can be contained in a cone centered at the origin with an arbitrarily small non-zero angle at the centre.
\end{remark}

We now define an arrangement of pieces on the draughtboard which will be key in the next argument.

\begin{definition}\label{king_tree}
A White \emph{king tree} $T_W$ is an arrangement of White kings on the infinite draughtboard, such that;
\renewcommand{\theenumi}{\roman{enumi}}
\begin{enumerate}
    \item if a Black king $B$ is placed on a distinguished square, the \emph{root node} of $T_W$,  diagonally adjacent to exactly one White king of the king tree, then each king of $T_W$ could be captured via a single jump move made by $B$;
    this ensures that the kings of $T_W$ are suitably alternated to empty squares and $T_W$ is \emph{connected};
    
    \item the \emph{nodes} of $T_W$ are the empty squares of the board on which $B$ could end a, possibly not maximal, jump move; such jump move is unique for each node of $T_W$;
    this ensures that $T_W$ has no \emph{loops}, and so really looks like a tree;
    
    \item there is no empty square of the board diagonally adjacent to 4 kings of $T_W$; so that the \emph{nodes} of $T_W$ have degree at most 3, and so we can say that $T_W$ is a binary tree.
\end{enumerate}
\end{definition}

Following the usual graph-theoretic terminology, we say that a node of a White king tree $T_W$ is a \emph{leaf} if it is the square reached by a Black king that makes a maximal jump move starting from the root of $T_W$.

\begin{theorem}\label{dr_thm}
For any well-founded tree $T$ such that the game of climbing-through-$T$ has value $\alpha$ for Observer, there is an open game $\mathcal{D}_T$ of Infinite Draughts which has value $\alpha$ for White; moreover, any strategy for Climber in the climbing-through-$T$ game determines a strategy for Black in $\mathcal{D}_T$, and vice versa.\footnote{Such strategy correspondence could be made unique by imposing arbitrary conditions on them; it is not an issue in our discussion.}

The initial position of $\mathcal{D}_T$ comprises of a White king tree $T_W$, a Black king on the root of $T_W$, and White kings on all the leaves of $T_W$;
the Black king is placed on the Southernmost non-empty square of the board and Black is the first to move.
\end{theorem}

\begin{proof}
We will prove by transfinite induction on the rank $\alpha$ of the tree $T$ that there exists a game $\mathcal{D}_T$ of Infinite Draughts as described.

\medskip

If $\alpha = 0$, then White has already won and the empty board can be the trivial initial position of $\mathcal{D}_T$; Black loses by being the first player.

\medskip

If $\alpha = 1$, then the  children of the root node $v_0$ of $T$ are all leaves.
If $v_0$ has only one child, then we can choose the position in figure \ref{dr_value_1} as the initial position of $\mathcal{D}_T$; Black must make an initial jump move that ends on the square labelled with a dagger, and his only piece is then captured by White in the following turn.

Observe that figure \ref{dr_value_1} also shows what the leaves of the White king tree embedded in the draughtboard look like.

\begin{figure}[H]
\centering
\includegraphics[width=3.5cm]{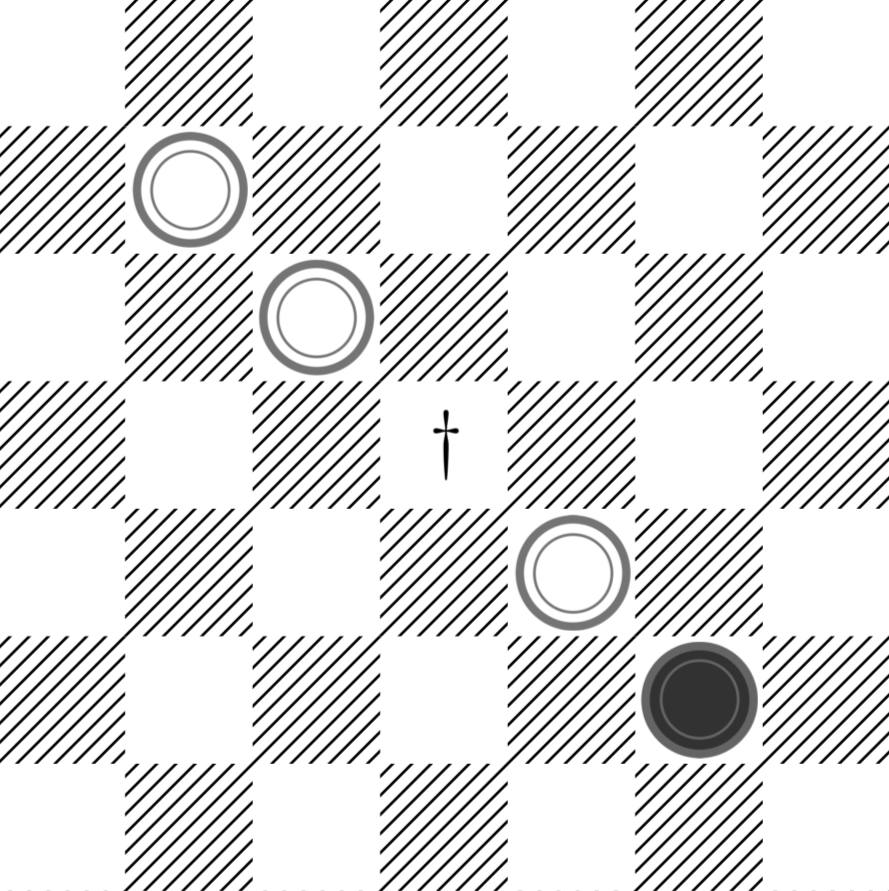}
\caption{Position with value 1 for White, Black to move.}
\label{dr_value_1}
\end{figure}

Similarly, if the root node of $T$ has any other finite number of children, we can iterate this construction; see figure \ref{dr_3_branch} for the initial position of $\mathcal{D}_T$ when $v_0$ has exactly 3 children.

\begin{figure}[h]
\centering
\includegraphics[width=8cm]{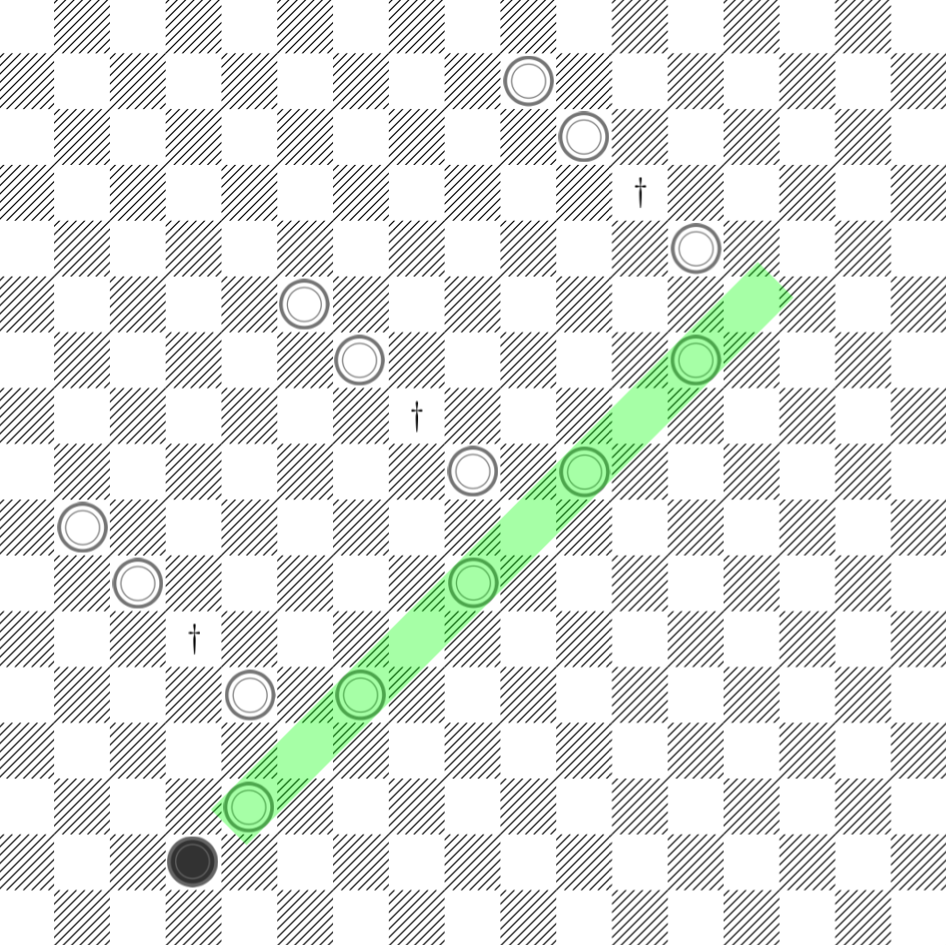}
\caption{Black to move and choose one of 3 jump moves; value 1 for White.}
\label{dr_3_branch}
\end{figure}

Furthermore, if $v_0$ has (countably) infinitely many children, then we can let the initial position of $\mathcal{D}_T$ be as in figure \ref{dr_inf_branch}.

\begin{figure}[h]
\centering
\includegraphics[width=10cm]{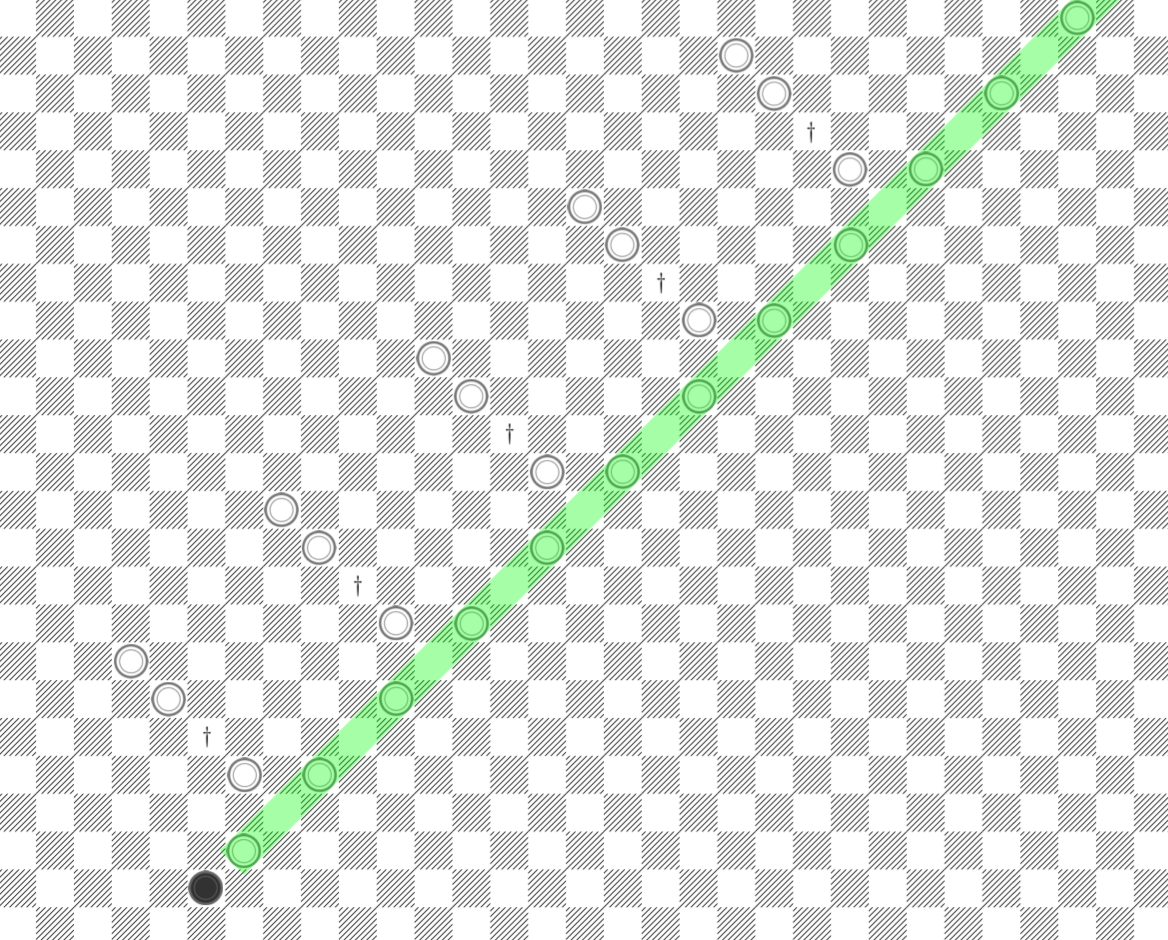}
\caption{Position with value 1 for White; Black to move.}
\label{dr_inf_branch}
\end{figure}

Observe that, in the last 3 figures presented, the Black king can only make legal maximal jump moves ending on the squares labelled with daggers.

In particular, in figure \ref{dr_inf_branch}, Black must choose one of such squares; in fact, rule \ref{fin_iteration} prevents the Black king from making an infinitely iterated jump and thus remaining on the White ladder highlighted in green.

We can see that, given a strategy for Climber, that is a choice of a child node of $v_0$, we can determine a strategy for Black in $\mathcal{D}_T$, that is a choice of jump move to the corresponding square labelled with a dagger.

\medskip

Here, the fundamental observation is that we can see the initial position of $\mathcal{D}_T$, in which the Black king stands at the root node of the king tree, as corresponding to the initial position of the climbing-through-$T$ game, in which Climber stands on the root of $T$;
moreover, we can see the Black king reaching a square labelled with a dagger as corresponding to the Climber moving up to the relevant node.

Hence, this construction allows us to determine the strategy correspondence of the Theorem.

\medskip

Now, suppose that for all trees of rank $\beta<\alpha$ there are games of Infinite Draughts with value $\beta$ as described.\footnote{Recall that the rank of $T$ equals the game value of climbing-through-$T$ by Proposition \ref{tree_value}.}

First, suppose for simplicity that the root $v$ of $T$ has a single child $v'$;
if the maximal sub-tree $T'$ of $T$ with root $v'$ has rank $\delta$,\footnote{$T'$ is maximal in the sense that its vertex set is the same as the one of $T$, except for $v$.}
then $T$ has rank $\delta+1= \alpha$.

By induction, there exists some Infinite Draughts game $\mathcal{D}_{T'}$ of value $\delta$ as described by the Theorem;
let $\widetilde{T}'_W$ be the arrangement of White kings at the initial position of $\mathcal{D}_{T'}$, so that $\widetilde{T}'_W$ consists of a White king tree $T'_W$ with White kings on the leaf nodes.

Let the initial position of $\mathcal{D}_{T}$ be as in figure \ref{dr_ind_step}, in which we embed two copies of $\widetilde{T}'_W$, so that each distinct copy of the king tree $T'_W$ has one of the two starred squares as a root node.

Observe that this initial position of $\mathcal{D}_{T}$ is constructible thanks to Proposition \ref{full_binary}.
The full binary tree can be embedded in a quadrant of the integer lattice and the king tree $T'_W$ is binary, so the position of figure \ref{dr_ind_step} can be constructed so that the Black king stands on the Southernmost non-empty square of the draughtboard.
Observe that the embedded copies of $T'_W$ may have more White kings than the original arrangement because Proposition \ref{full_binary} ensures constructibility, provided flexibility in the length of edges;
this is not an issue because it is never advantageous for the Black king to end jump moves on nodes of the king tree which are not binary,\footnote{According to our definition of king tree, all the empty squares between white kings are \emph{nodes}; however, the not properly branching nodes cannot be resting squares.}
such ``stretched" copies of $T'_W$ will be equivalent to the original $T'_W$ with respect to the possible strategies available to Black.

In other words, Proposition \ref{full_binary} allows us to embed in the draughtboard king trees with the same graph-theoretic features of the original $T'_W$.

\begin{figure}[h]
\centering
\includegraphics[width=8cm]{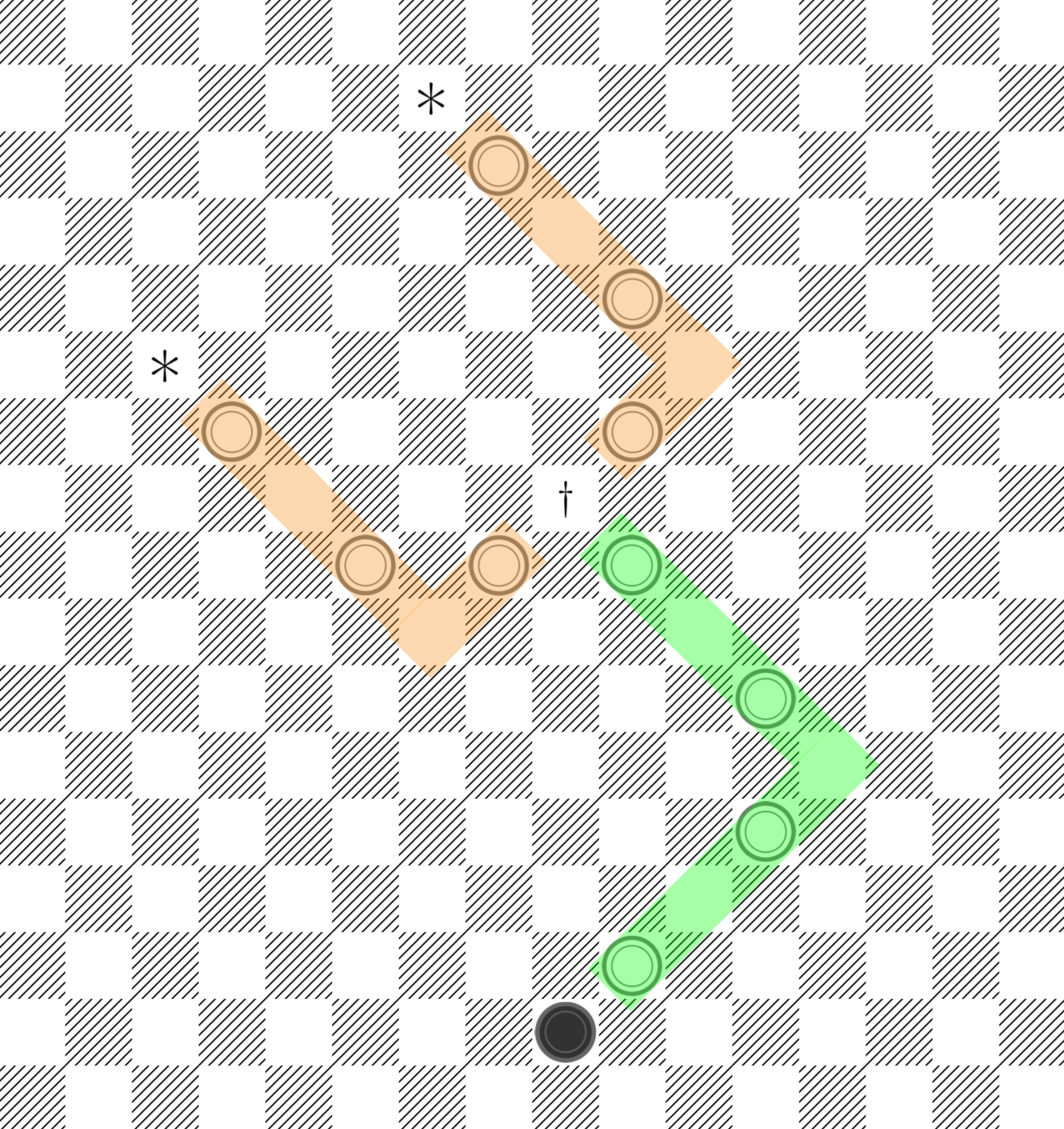}
\caption{Inductive step.}
\label{dr_ind_step}
\end{figure}
\bigskip

We can now describe the main line of play of Black, who is the first to move in $\mathcal{D}_{T}$.

The initial move by black is a jump move to the tile labelled with a dagger; we call it a \emph{resting} square because the Black king can stop there without being captured by White at the next turn.\footnote{Observe that the Black king reaches the resting square by jumping over a White piece, hence it ends the move being adjacent to exactly 2 White pieces.}

Then, we can see that the Black king is standing on the root of two White king trees, highlighted in orange, which each contain a copy of $T'_W$; at the next turn, White can only move one piece from either of such king trees.\footnote{Or one of the kings on the leaves of such king trees.}

At the next turn, Black can let his king make a jump move up the king tree whose pieces were not moved by White, and then play as in $\mathcal{D}_{T'}$.

Since Black can then play in $\mathcal{D}_{T'}$ so to realise the game value $\delta$, then the strategy just described realises the value $\delta+1= \alpha$.
In fact, if the Black king does not stop on the described resting square, then it could stop on a non-resting square, leading to Black's defeat on the next turn, or move past the first resting square and stop on some resting square higher up in one of the embedded copies of $\mathcal{D}_{T'}$; however, recalling the correspondence with the climbing-through-$T$ game, such a move would correspond to the Climber climbing up more than one edge per turn, causing a decrease in game value larger than necessary.
Moreover, observe that White's moves cannot really obstruct Black in the same way that the Observer does not obstruct the Climber by saying ``OK".

Hence, the main line of play for Black describes a well-defined strategy.

\bigskip

We have now all that we need to complete the proof; if the root of the tree $T$ has finitely many children, then we can iterate the previous construction so that, for each child $v$ of the root of $T$, letting $T_v$ be the maximal sub-tree of $T$ rooted at $v$, we can have, for each $v$, a distinct resting square which is the root of two White king trees, both realising the White king tree of $\mathcal{D}_{T_v}$.

It should be clear that we can inductively construct the initial position of $\mathcal{D}_{T}$ when the root of $T$ is infinitely branching, as in figure \ref{dr_inf_DT}; a position for $\mathcal{D}_{T}$ when the root of $T$ has finitely many children can be obtained by truncation of the green branch.

\begin{figure}[h]
\centering
\includegraphics[width=10cm]{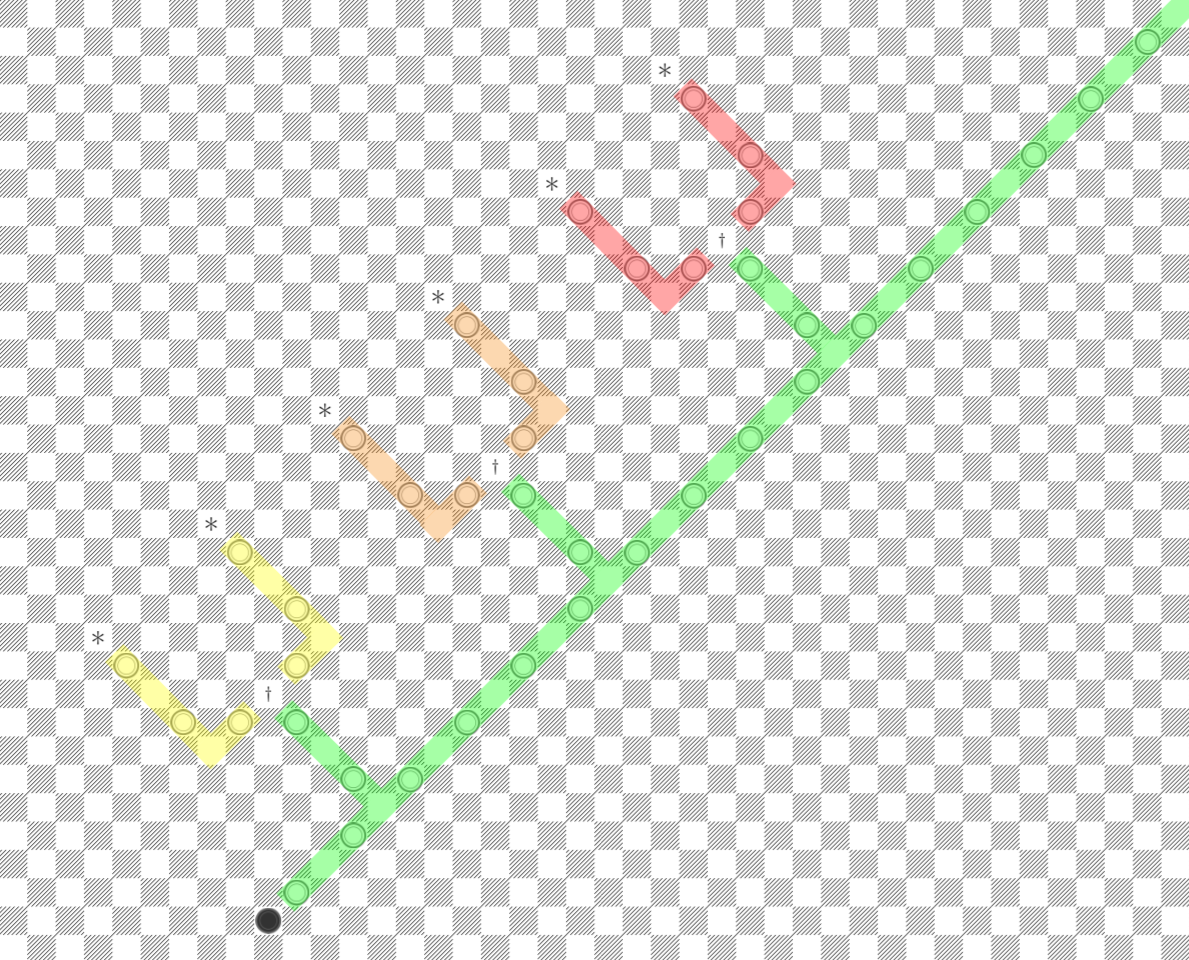}
\caption{$\mathcal{D}_{T}$ when the root of $T$ is infinitely branching.}
\label{dr_inf_DT}
\end{figure}

Observe that the discussion following figure \ref{dr_ind_step} regarding the correspondence between the resting squares of the initial position of $\mathcal{D}_{T}$ and the nodes of $T$ in climbing-through-$T$ easily extends to the finite and infinite branching positions of figures \ref{dr_ind_step}-\ref{dr_inf_DT},
with the only difference that White may respond to Black's moves by moving pieces which Black does not intend to capture.\footnote{In other words, White could move pieces in irrelevant far branches of the king tree.}

\bigskip

The proof is now complete.
\end{proof}

Observe that Theorem \ref{dr_omega_one} is a direct consequence of Theorem \ref{dr_thm} and Proposition \ref{omega_one_climb};
we can appreciate that Theorem \ref{dr_thm} is a stronger result because it does not just prove the existence of Infinite Draughts games of arbitrarily high (countable) value, but also constructs them explicitly.

\subsection{Further results}\label{draughts_further_results}
\subsubsection{Computable positions}

Note that we can define a position of Infinite Draughts to be a function that assigns to each white square of the chessboard a value to mark it as free or occupied by a piece, specifying if it is a pawn or a king and whether it is Black or White, in addition to specifying a turn indicator.

It is then immediate to consider computable positions as computable functions.

\medskip

The following result is an adaptation of \cite[Theorem 6]{hamkins14} and is a consequence of Theorem \ref{dr_thm}.

\begin{corollary}\label{dr_computable}
There is a computable position of Infinite Draughts which is a win for a designated player if both players are constrained to play according to computable strategies, and is a draw otherwise.
\end{corollary}

We first recall a classic graph-theoretic result.

\begin{lemma}[K\"onig's Lemma]\label{konig_lemma}
Any infinite finitely-branching tree has an infinite branch.
\end{lemma}

\begin{proof}
[Proof of Corollary \ref{dr_computable}]

Let $T\subset 2^{<\omega}$ be a computable infinite tree with no computable infinite branch; such trees exist by elementary results of computability theory.

Since $T\subset 2^{<\omega}$ is clearly finitely-branching, then $T$ is an infinite finitely-branching tree, and so $T$ has an infinite branch by K\"onig's Lemma; such branch is not computable by definition of $T$.

Even though $T$ is not well-founded, it still has branches of order type at most $\omega$, so that we can apply the construction in the proof of Theorem \ref{dr_thm} to construct a game $\mathcal{D}_T$ of Infinite Draughts equivalent to the climbing-through-$T$ game.

\medskip

It is clear that Climber has a winning strategy for the climbing-through-$T$ game, i.e.~climbing up an infinite branch of $T$; however, $T$ has only non-computable infinite branches, so such strategy for Climber must be not computable.

By the strategic equivalence of climbing-through-$T$ with $\mathcal{D}_T$, we have that also Black has a non-computable drawing strategy in $\mathcal{D}_T$.

\medskip

However, if Climber plays according to a computable strategy in climbing-through-$T$, then he can only climb up a computable branch of $T$, which must be finite by definition of $T$.
Correspondingly, if Black plays according to a computable strategy, then he will eventually reach in finitely many turns a leaf of the king tree in $\mathcal{D}_T$;
in such case, Black reaches an arrangement of White kings as in figure \ref{dr_value_1} and thus loses.

It is clear that any strategy is winning for White, as it is enough to require that White captures the Black king at the first mistake of Black, and make any other uninfluential move, otherwise;
hence, we can say that White has a computable strategy which allows him to win in $\mathcal{D}_T$ when Black plays according to any computable strategy.

This completes the proof.
\end{proof}

\subsubsection{Tree nodes for different rules}
We now focus on extending the previous results to games of Infinite Draughts that assume a less strict set of rules.

In particular, we will design nodes for the White king tree constructed in Theorem \ref{dr_thm} with more structure than the one in figure \ref{dr_ind_step}.

All the positions in this chapter were kindly reviewed by Sergio Scarpetta, current World Champion of English Draughts (3-move variant), who confirmed the interpretations presented here, in particular the one regarding figure \ref{big_7_branch_node}.

\medskip

We observe that by allowing nodes to have fixed finite size in Proposition \ref{full_binary}, we conclude that the full $n$-ary tree $n^{<\omega}$ can be embedded in the integer lattice $\mathbb{Z}\times \mathbb{Z}$.

\medskip

Now that we have ensured to have enough space on the infinite draughtboard, we can present the node needed when dropping the forced jump rule \ref{forced_jump}, but maintaining the finite iteration \ref{fin_iteration}.

Firstly, we need to ensure that the Black king actually starts jumping on the White king tree, so that the root of the king tree needs two extra \emph{guardian} White kings, as in figure \ref{root_not_forced}, where Black immediately loses by making an initial non-jump move; this was not needed in the previous root configuration.

\begin{figure}[h]
\centering
\includegraphics[width=3cm]{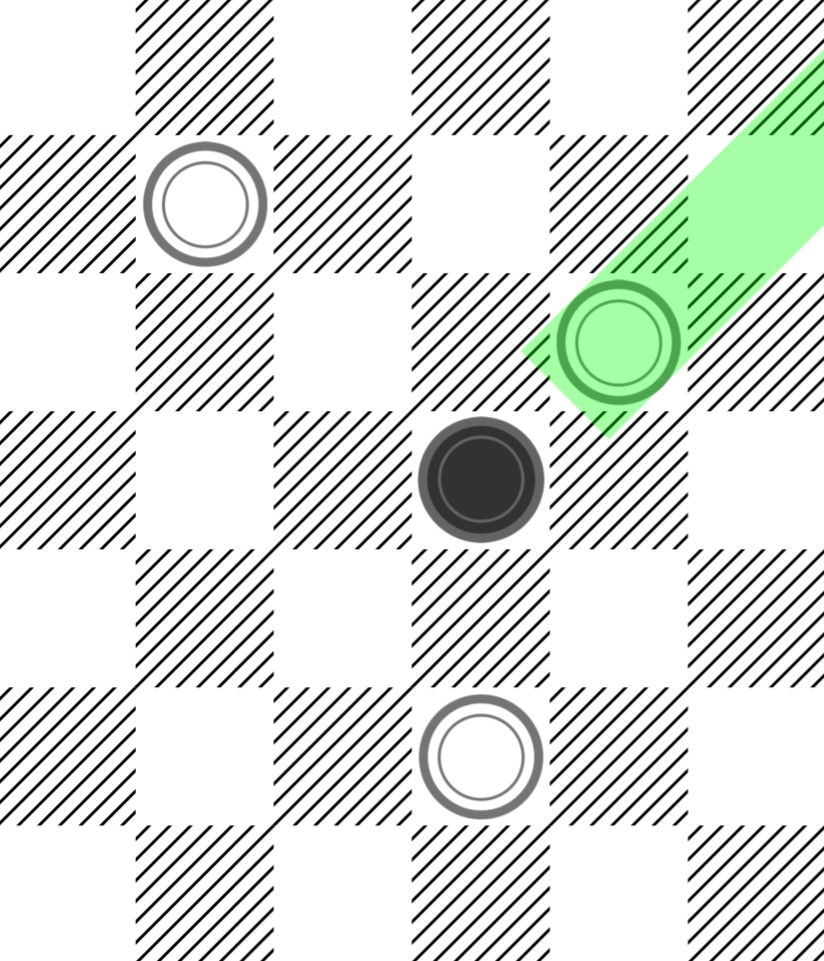}
\caption{Root with guardians.}
\label{root_not_forced}
\end{figure}

Even though the Black king will still find it disadvantageous to stop on non-resting squares, we still need to ensure that the Black king does not get \emph{off-track}, that is out of the king tree, when making a move from a resting square;
that is exactly why we add two guardian White kings to the king tree branching node, which are highlighted in yellow in figure \ref{node_guardians}.

\begin{figure}[h]
\centering
\includegraphics[width=8cm]{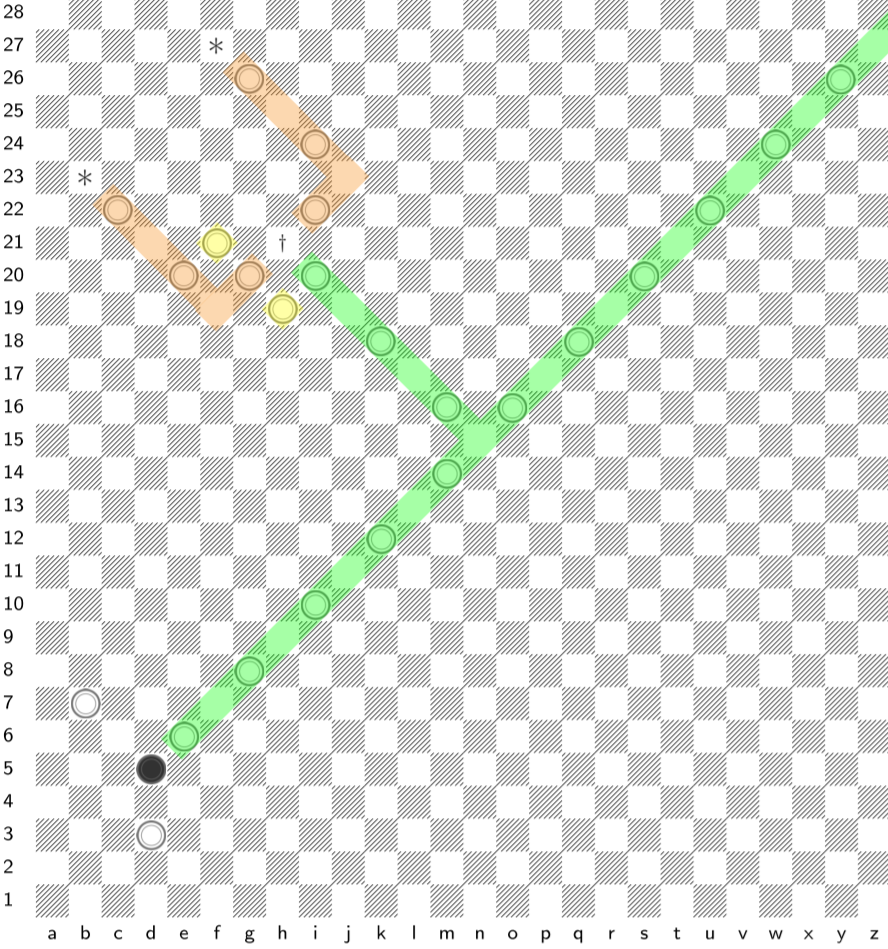}
\caption{Branching node with guardians.}
\label{node_guardians}
\end{figure}

\begin{remark}\label{rmk_off_track}
Observe that when the Black king is on some resting square, White could move one of the guardian kings to a square adjacent to the Black king, offering it the possibility to go off-track.

The interesting thing to note is that White would make such a move only when, following the Black king going off-track, White could reach a position of value lower than the ones reachable by allowing Black to play as usual;
such a situation means exactly that it is not advantageous for Black to go off-track.

This is enough to say that, with optimal play, Black will not capture a guardian king moved by White.
In other words, White moves guardian kings only when that is uninfluential, exactly as the Observer says ``OK" in the climbing-through-$T$ game.
\end{remark}

\bigskip

Recalling that we stated rule \ref{fin_iteration} in order to avoid discussing infinitely iterated jumps, we are now ready to drop it and only assume the forced iteration rule \ref{forced_iteration} with its strict interpretation in Infinite Draughts.

In other words, we now assume that players are not obliged to make jump moves, but, if they do make a jump move, then they can only make maximal jumps, which include infinitely iterated jumps.

This complicates the structure of the nodes of, what we should now call, \emph{extended king tree} constructed as in the proof of Theorem \ref{dr_thm};
in fact, we will require the Black king to reach a resting square and possibly make a non-jump move, before continuing to climb the extended king tree.

Observe in figure \ref{big_7_branch_node} the first node that the Black king can decide to pick;
it is a 7-branching node, so that each branch leads to an equivalent extended tree, as for the binary case in the original argument.

\begin{figure}[h]
\centering
\includegraphics[width=10cm]{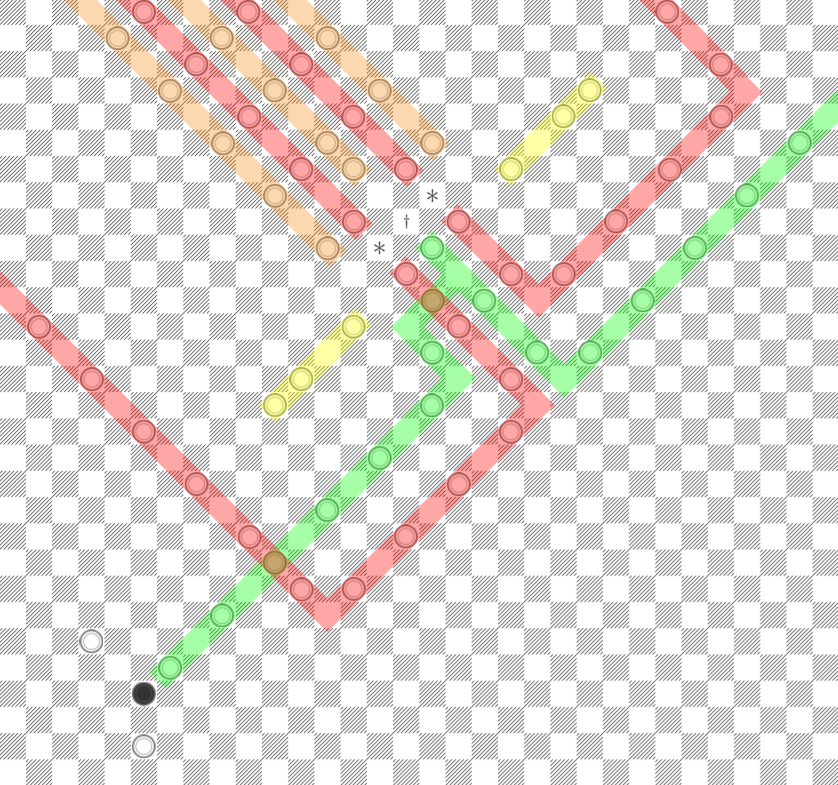}
\caption{7-branching node.}
\label{big_7_branch_node}
\end{figure}

\begin{remark}
It may be possible to further simplify the node in figure \ref{big_7_branch_node}; however, the one presented is likely to be the minimal node when reintroducing the forced jump rule.
\end{remark}

Observe that the branches are coloured red or orange depending on their row parity.

Suppose that the Black king decides to pick that first node, rather than continuing on the green ladder, and so ends the initial jump move on the resting square labelled with a dagger.


The main line of play for Black is as follows. After White's non-jump move, move the Black king to the starred square adjacent to an external triplet of branches which was not modified by White; White can then move a piece from only one of the red branches next to the Black king, which can then make a jump move to climb up the other, going to another node.

\medskip

There are several ways in which White could interfere with this strategy, and we will see that they are uninfluential in terms of reducing the game value.

\medskip

Even though players are not forced to make jumps when legal, White might offer the possibility of a jump move to Black.

Recalling that jumps must be maximal, we have that, if White moves a king so that Black can make a jump towards North-West, then Black can decide to deviate from the main line of play to go up the orange branch, in which case Black leaves the node reducing the value by 1 rather than 2, which is still a positive outcome for Black.

White could also offer possible jump moves towards NE or SW by moving a White king to a starred tile; if Black takes up the offer, then he simply jumps up the relevant external orange branch.

Moreover, if White moves a king to the South-East of the Black king, then it is likely to be advantageous for Black to make such a jump move and go to choose another node up the green branch;
this would correspond to the Observer allowing the Climber to go back down to the tree root and make a different choice.

\medskip

Following White's response to Black's initial move, Black may decide to move his king to a non-starred square.

If the Black king is moved to the North-West of the resting square, then it would be captured by White on the next move.

If instead the Black king is moved towards South-East, that is ``backwards", then White can place a king to the North-West of the (original) resting square, and then we can see that the Black king is trapped, also thanks to the guardian kings highlighted in yellow.

\medskip

Finally, the only other option available to White when Black is standing on a starred square is moving the piece at the bottom of the external orange branch next to the Black king;
by an analogous argument to Remark \ref{rmk_off_track}, we have that White would make that move only if that does not affect Black's play.

\medskip

With this, we conclude our discussion of Infinite Draughts.

\newpage
\appendix
\chapter{Appendix}
\section{2-colourability in Maker-Breaker games}\label{app_2_col_mb}

We prove the following result, which is adapted for any generalised hypergraph $(B,\mathcal{F})$ from \cite[Proposition 2.3.1]{hefetz14}.

\medskip

\begin{proposition}
Take some Maker-Breaker game $\mathcal{G}$ played on a hypergraph $(B,\mathcal{F})$ such that $\mathcal{G}$ is open for Breaker, and Maker is the first player to move.\footnote{Note that this argument relies on the fact that Breaker's winning strategy is a second-player strategy.}
If Breaker has a winning strategy, so that he can play to take at least one vertex in each of Maker's winning subsets in $\mathcal{F}$, then $(B,\mathcal{F})$ is properly 2-colourable.
\end{proposition}
\begin{proof}
Suppose that $\sigma$ is a winning strategy for Breaker, and let Breaker play according to it, so that at least one vertex in each hyperedge in $\mathcal{F}$ is chosen by Breaker, who does that after finitely many turns.

Moreover, let Maker make an arbitrary initial move, which cannot be disadvantageous, and then play according to $\sigma$, as if the arbitrary move was not placed on the board. Whenever $\sigma$ prescribes to choose a vertex already marked by Maker himself, then he can analogously make another arbitrary move.
Note that Maker's extra move is the only one that the strategy $\sigma$ does not consider, thus $\sigma$ will never prescribe Maker to choose a vertex already marked by Breaker.

Then, Maker will satisfy Breaker's winning condition; that is, given any hyperedge in $\mathcal{F}$, Maker will choose one vertex from it, and will do that after finitely many turns.

\medskip

Note that if $B$ is infinite, then not all vertices will necessarily be chosen by either player at any finite stage; if a vertex does not belong to any hyperedge, then it will never be chosen.

Moreover, if $B$ is uncountable, then at most countably many vertices will be marked by either player, as all plays are countable sequences of moves, by Definition \ref{def_game}.\ref{def_play}.

However, by the definition of $\sigma$, all hyperedges comprise one vertex marked by Maker and one by Breaker, both chosen at some finite stages of the game.

\medskip

Assign the colour white to all the vertices marked by Maker, and the colour black to all the vertices marked by Breaker or left unchosen.\footnote{So that we do not need the Axiom of Choice here.}

This assignment is a proper 2-colouring of $(B,\mathcal{F})$, as needed.
\end{proof}

\section{Strategy-stealing in stone-placing games}\label{app_str_steal_st_pl}

We generalise the symmetry condition of the finite Hex board highlighted in Remark \ref{rmk_finite_hex}.\ref{symmetry_board} to construct the Strategy-stealing argument for stone-placing games.

\begin{definition}\label{def_strict_not_open}A game is \emph{strictly not open} for one player if that player can only win by playing for infinitely many turns; that is, if the winning condition of that player has empty interior, and so does not contain any basic open set, within the space of plays.
\end{definition}

\begin{proposition}[Strategy-stealing]\label{stone_pl_strategy_stealing}
Let $\mathcal{G}$ be a stone-placing game played on $(B,\mathcal{F},\mathcal{S})$, strictly not open for both players. Say that there is a fixed-point free involution $g$ of the board $B$.
Suppose that $g$ is such that for each $s\in\mathcal{S}$ there is some $f\in\mathcal{F}$ for which $f\subset g[s]$.

Then, the second player in $\mathcal{G}$ does not have a winning strategy.
\end{proposition}

Hence, we can informally interpret the involution $g$ as a symmetry (or better, a reassemblage) of the board that allows the first player to act as the second player on the new board.

\begin{proof}[Proof of Proposition \ref{stone_pl_strategy_stealing}]

Let the first player make an arbitrary initial move, which is advantageous, and then play in the following way; 
when the second player chooses a vertex $v\in B$, then the first player, in the next turn, marks $g(v)$.

Note that $s\in\mathcal{S}$ is infinite by strict non-openness of $\mathcal{G}$, and that the second player will mark all the vertices of $s$, each after finitely many turns.
Thus, the first player will eventually mark all the vertices of $g[s]$.

Therefore, the first player will also win by marking all the vertices of some $f\in\mathcal{F}$ such that $f\subset g[s]$, which is a contradiction.
\end{proof}

\begin{remark}
Observe that we cannot relax the assumption of strict non-openness in Proposition \ref{stone_pl_strategy_stealing}.

Otherwise, the second player could pick a finite $s\in\mathcal{S}$, which preferably does not contain the initial move of the first player, and then could proceed to mark all the vertices in $s$, so that the first player would be able to fully satisfy his winning condition $f\subset g[s]$ only with his move that follows the last move of the second player; that is a move too late, so that the second player actually wins.

Note that strengthening the condition by imposing that the involution $g$ is such that for each $s\in\mathcal{S}$ there is some $f\in\mathcal{F}$ for which $f\subsetneq g[s]$ is still not enough; there could be an appropriate order in which the second player could mark the vertices of a finite $s\in\mathcal{S}$ that defeats the opponent.
\end{remark}

\newpage
\phantomsection
\addcontentsline{toc}{chapter}{Bibliography}
\bibliographystyle{apalike}
\bibliography{main}

\end{document}